\documentclass{svproc}

\usepackage{fancyhdr}
\usepackage{stmaryrd}
\usepackage{calrsfs}

\usepackage{amsmath}
\usepackage[nospace,noadjust]{cite}
\usepackage{amsfonts}
\usepackage{amssymb}
\usepackage{cite}
\usepackage{comment}
\usepackage{color}
\usepackage[all]{xy}
\usepackage{hyperref}
\usepackage{lineno}
\usepackage{tikz-cd}
\usepackage{enumerate}

\usepackage{graphicx, multicol, footmisc}
\usepackage{url}

\newcommand{\m}{\mathfrak{m}}

\newcommand{\ff}{\mathbb{F}}
\newcommand{\nn}{\mathbb{N}}
\newcommand{\bbn}{\mathbb{N}}
\newcommand{\pp}{\mathbb{P}}
\newcommand{\qq}{\mathbb{Q}}
\newcommand{\rr}{\mathbb{R}}
\newcommand{\zz}{\mathbb{Z}}
\newcommand{\bbz}{\mathbb{Z}}

\newcommand{\ch}{\text{char}}
\newcommand{\qf}{\text{qf}}
\newcommand{\gp}{\text{gp}}
\newcommand{\mcd}{\text{mcd}}
\newcommand{\ord}{\text{ord}}

\newcommand{\ii}{\mathcal{A}}

\newcommand{\uu}{\mathcal{U}}

\providecommand\ldb{\llbracket}
\providecommand\rdb{\rrbracket}

\newcommand{\onto}{\twoheadrightarrow}

\begin{document}
	\mainmatter
	
	\title{Atomicity in integral domains}
	\titlerunning{Atomicity in integral domains} 
	
	\author{Jim Coykendall \inst{1} \and Felix Gotti \inst{2}}
	\authorrunning{J. Coykendall and F. Gotti}
	\tocauthor{Jim Coykendall, Felix Gotti}
	
	\institute{School of Mathematical and Statistical Sciences, \\Clemson University, Clemson, SC 29634\\
		\email{jcoyken@clemson.edu},
		\and
		Department of Mathematics, MIT, Cambridge, MA 02139\\
		\email{fgotti@mit.edu}}

%
	

\maketitle

\begin{abstract}
	In algebra, atomicity is the study of divisibility by and factorizations into atoms (also called irreducibles). In one side of the spectrum of atomicity we find the antimatter algebraic structures, inside which there are no atoms and, therefore, divisibility by and factorizations into atoms are not possible. In the other (more interesting) side of the spectrum, we find the atomic algebraic structures, where essentially every element factors into atoms (the study of such objects is known as factorization theory). In this paper, we survey some of the most fundamental results on the atomicity of cancellative commutative monoids and integral domains, putting our emphasis on the latter. We mostly consider the realm of atomic domains. For integral domains, the distinction between being atomic and satisfying the ascending chain condition on principal ideals, or ACCP for short (which is a stronger and better-behaved algebraic condition) is subtle, so atomicity has been often studied in connection with the ACCP: we consider this connection at many parts of this survey. We discuss atomicity under various classical algebraic constructions, including localization, polynomial extensions, $D+M$ constructions, and monoid algebras. Integral domains having all their subrings atomic are also discussed. In the last section, we explore the middle ground of the spectrum of atomicity: some integral domains where some of but not all the elements factor into atoms, which are called quasi-atomic and almost atomic. We conclude providing techniques from homological algebra to measure how far quasi-atomic and almost atomic domains are from being atomic.
\end{abstract}


\bigskip
\section{Introduction}
\label{sec:intro}
\smallskip

In commutative algebra, the study of factorization is central to the discipline. As any commutative ring is additively an abelian group, the behavior of its underlying multiplicative semigroup is the essence of what separates commutative algebra from abelian group theory. 
\smallskip

For an integral domain $R$ whose quotient field is $K$ and whose group of units is $R^\times$, the abelian group $G(R) := K^\times/R^\times$ is fundamentally associated with the semigroup that contains the essence of the multiplicative information encoded in $R$: the group $G(R)$ is called the \emph{group of divisibility} of $R$, and it is has the following natural partial ordering:
\[
	rR^\times \leq sR^\times \quad \text {if and only if} \quad \frac{s}{r} \in R.
\]
It is important to note that nonnegative elements under this partial ordering are precisely the cosets that are represented by elements of $R$. In the 1960s and 1970s, there was a flurry of interest in the group of divisibility (see, for example, \cite{Mott,mott1974convex}), and perhaps the most well-known result of this time was the Jaffard-Ohm-Kaplansky Theorem (sometimes referred to as the Krull-Kaplansky-Jaffard-Ohm Theorem). We record this celebrated theorem here, and the interested reader can find a short divisor-theoretical proof of the same in~\cite[Corollary~2]{GH96}. 

\begin{theorem}[Krull-Kaplansky-Jaffard-Ohm]
	If $(G, \leq)$ is a lattice-ordered abelian group, then there exists a B\'{e}zout domain with group of divisibility order-isomorphic to $(G, \leq )$.
\end{theorem}
The evolution of this result was spread out in several papers over the course of a few decades, and a more complete discussion of this can be found in \cite{brandal1976}.
\smallskip

Perhaps the second wave of in-depth study of the multiplicative structure of integral domains began in the early 1990s. Appearing in 1990, the paper \cite{AAZ90} by Anderson, Anderson, and Zafrullah was the landmark of this renaissance of the study of factorization in the setting of integral domains. Although we will only venture here into the arena of integral domains, it is worth noting that the study of factorization initiated in~\cite{AAZ90} was later generalized to the setting of commutative rings with identity in the presence of zero-divisors (see, for example, \cite{DDAZ}). In \cite{AAZ90}, various types of integral domains, defined by factorization properties, were explored. Some of the key objects of study are listed below.

\begin{enumerate}
	\item Atomic domains, which are integral domains with the property that every nonzero nonunit of~$R$ is a product of atoms.
	\smallskip
	
	\item ACCP domains, which are integral domains with the property that every ascending chain of principal ideals stabilizes.
	\smallskip
	
	\item Bounded factorization domains (BFDs), which are atomic domains such that every nonzero nonunit has a bound on the lengths of its irreducible factorizations.
	\smallskip
	
	\item Finite factorization domains (FFDs), which are atomic domains such that every nonzero nonunit has only finitely many irreducible factorizations (up to associates).
	\smallskip
	
	\item Half-factorial domains (HFDs), which are atomic domains such that any two irreducible factorizations of the same element have the same length.
	\smallskip
	
	\item Unique factorization domains (UFDs), which are integral domains such that every nonzero nonunit is a product of prime elements.
\end{enumerate}
The notions of bounded factorization and finite factorization domains were both introduced in~\cite{AAZ90}, they were investigated in the context of the following diagram:

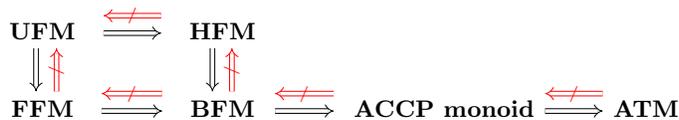
\begin{figure}[h]
	\begin{tikzcd}
		\textbf{ UFM } \ \arrow[r, Rightarrow] \arrow[red, r, Leftarrow, "/"{anchor=center,sloped}, shift left=1.7ex] \arrow[d, Rightarrow, shift right=1ex] \arrow[red, d, Leftarrow, "/"{anchor=center,sloped}, shift left=1ex]& \ \textbf{ HFM } \arrow[d, Rightarrow, shift right=0.6ex] \arrow[red, d, Leftarrow, "/"{anchor=center,sloped}, shift left=1.3ex] \\
		\textbf{ FFM } \ \arrow[r, Rightarrow]	\arrow[red, r, Leftarrow, "/"{anchor=center,sloped}, shift left=1.7ex] & \ \textbf{ BFM } \arrow[r, Rightarrow]  \arrow[red, r, Leftarrow, "/"{anchor=center,sloped}, shift left=1.7ex] & \textbf{ ACCP monoid}  \arrow[r, Rightarrow] \arrow[red, r, Leftarrow, "/"{anchor=center,sloped}, shift left=1.7ex]  & \textbf{ATM}
	\end{tikzcd}
	\caption{The implications in the diagram show the known inclusions among the subclasses of atomic monoids we have previously mentioned. The diagram also emphasizes (with red marked arrows) that none of the shown implications is reversible.}
	\label{fig:AAZ's atomic chain for monoids}
\end{figure}
Some of the implications denoted by arrows in the above diagram follow immediately from the definitions, but a couple of them are a bit more subtle. More remarkably, the authors show that, in general, none of the implications in the above diagram can be reversed: the counterexample for [BFD $\Leftarrow$ ACCP] given in~\cite{AAZ90} and the counterexample for [ACCP $\Leftarrow$ atomic] cited in~\cite{AAZ90} (and originally provided by Grams~\cite{aG74}) are decidedly delicate.
\smallskip

At the top of the food chain are UFDs, which encapsulate the notion of unique factorization into irreducibles (i.e., the statement of the Fundamental Theorem of Arithmetic). There are a number of properties that indicate the strength of this notion. Perhaps the most famous is the following theorem, generally credited to Kaplansky \cite[Theorem~5]{iK74}.
\begin{theorem}\label{kap}
	If $R$ is an integral domain, then $R$ is a UFD if and only if every nonzero prime ideal of $R$ contains a prime element.
\end{theorem}
\noindent It is interesting to note that this particular characterization of a UFD immediately shows that if $R$ is a PID, then $R$ is a UFD, and this allows us to conclude that Theorem~\ref{kap} depends on the Axiom of Choice. Indeed, it is shown in \cite[Corollary~10]{H1976} that in Zermelo–Fraenkel set theory (i.e., in the absence of the Axiom of Choice), it is impossible to show that every PID is a UFD (or even that every PID has a maximal ideal).
\smallskip

The ideal-theoretic characterization given by Theorem \ref{kap} is a source of the strength of the unique factorization property and almost universally gives a route to a quick(er) proof of some standard results. Most integral domains defined by factorization properties are described as certain specialized behavior of irreducible product representations. As a rule of thumb, integral domains with an ideal theoretic description tend to come out on top with regard to strength and sweeping nature of results (e.g., UFDs vs. HFDs and ACCP domains vs. atomic domains). This line of thought can be further illustrated by considering the class of HFDs in the setting of rings of algebraic integers. For rings of algebraic integers, Carlitz~\cite{carlitz} provided an ideal theoretic characterization of HFDs: that is, $R$ is an HFD if and only if $\vert\text{Cl}(R)\vert\leq2$. It is not a surprise that it is in this arena where some of the results concerning HFDs are strongest.
\smallskip

The present survey primarily focuses on atomicity in the setting of integral domains: we study the most general class of integral domains represented in the above diagram. We note that this is the basic assumption for all the integral domains defined by the factorization properties already mentioned. Of course, it makes sense to begin the study of factorization by considering the class of integral domains where every nonzero nonunit actually possesses a factorization into irreducibles. However, the class of atomic domains in some sense is strange: it is characterized by the property that all proper nonzero principal ideals can be expressed as a product of maximal principal ideals, and one could argue that this is problematic as maximal principal ideals need not be prime. Despite the subtle distinction, the class of ACCP domains is much more robust, and this will be clear at various points of this survey. This may explain why historically atomicity has been mostly studied in connection with the ACCP.
\smallskip

The contrast between being atomic and satisfying the ACCP can be thought of in terms of trying to factor a nonzero, nonunit element. If this element is in a domain that is ACCP, then any factorization choices made are guaranteed to eventually succeed because of the fact that the induced chain stabilizes. Indeed, suppose that one begins with a nonzero, nonunit $a\in R$ where $R$ is a domain that is ACCP. If $a$ is irreducible then we are done, but if not, factor it into two nonunits and inductively continue this process on the nonunit factors of $a$. The ACCP condition, coupled with the fact that every nonunit in an ACCP domain is divisible by an irreducible (\cite{AAZ90}) gives that this process terminates in an irreducible factorization of $a$, regardless of the intermediate choices made on factoring the divisors of $a$. On the other hand, there are potential hazards in atomic domains that do not have the ACCP property. In fact, the very existence of an ascending chain of principal ideals $(r) \subsetneq (r_1)\subsetneq (r_2) \subsetneq \cdots$ shows that if one begins to factor an element $r$ and carelessly selects $r_n$ as a factor at the $n^{\text{th}}$ stage of factorization, then although $r$ may be factored into a finite product of irreducibles in some way, this unfortunate choice of successively choosing $r_n$ will never terminate.
\smallskip

The construction of atomic domains that do not satisfy the ACCP is a challenging task and a problem that has been the subject of systematic attention in the literature since Grams~\cite{aG74} constructed the first of such integral domains back in 1974, disproving an assertion made by Cohn in~\cite{pC68} (the equivalence of atomicity and the ACCP inside the class of integral domains). Further constructions of atomic domains not satisfying the ACCP have been provided by Zaks~\cite{aZ82}, Roitman~\cite{mR93}, Boynton and the first author~\cite{BC19}, and Li and the second author~\cite{GL23,GL23a,GL22}. Unlike atomicity, the ACCP ascends to polynomial extensions over integral domains \cite[page~82]{GP74} (this is not the case in the presence of zero-divisors; see~\cite{HL94}). In addition, for each prescribed field $F$, a reduced, cancellative, and torsion-free monoid $M$ satisfies the ACCP if and only if its monoid algebra $F[M]$ satisfies the ACCP \cite[Theorem~13]{AJ15}. Thus, when the atomicity of any (atomic) non-ACCP monoid ascends to its monoid algebra $F[M]$ for some field~$F$, one automatically obtains an atomic monoid algebra that does not satisfy the ACCP. This is how atomic monoid algebras that do not satisfy the ACCP were constructed in~\cite{GL23,GL23a}.
\smallskip

Another striking example of the volatile nature of the class of atomic domains (in contrast to that of ACCP domains) concerns its stability in polynomial extensions. Indeed, of the six classes of integral domains studied in \cite{AAZ90} and listed above, only the class of HFDs and the class of atomic domains are unstable under taking polynomial extensions. For atomicity, this was proved by Roitman in~\cite{mR93}; however, the techniques and constructions in~\cite{mR93} are intricate, and so we do not discuss them here. It should also be noted that Roitman showed in~\cite{mR00} that the same construction provides a counterexample to the ascent of atomicity to power series extension. This is perhaps less surprising as stability of ``nice'' properties in passing to power series extensions is a rarer beast; indeed, Samuel showed in~\cite{S1961} that, in general, UFDs do not survive the passage to power series.
\smallskip

In this survey, we will assemble some central results known (and some unknown) concerning the class of atomic domains flavored with some minor new results/insights. We will highlight some mainstream results and explore some pathologies in the settings of monoids and integral domains, with our main focus on the later. It is our intent to illustrate many interesting behaviors in the core of our study with myriad examples.

\bigskip
\section{Background}
\label{sec:background}

In this section we introduce most of the relevant terminology and notation we shall be using throughout this survey. In addition, we revise the standard results in factorization theory, ideal theory, and commutative rings that we need as tools in latter sections.

\medskip
\subsection{General Notation and Terminology}

We let $\nn$ denote the set of positive integers, and we set $\nn_0 := \nn \cup \{0\}$. As it is customary, we let $\zz, \qq$, and $\rr$ denote the sets of integers, rational numbers, and real numbers, respectively. If $b,c \in \zz$ with $b \le c$, then we let $\ldb b,c \rdb$ denote the discrete interval from $b$ to $c$; that is,
\[
	\ldb b,c \rdb := \{k \in \zz : b \le k \le c\}.
\]
For a nonzero $q \in \qq$, we let $\mathsf{n}(q)$ and $\mathsf{d}(q)$ denote, respectively, the unique $n \in \zz$ and $d \in \nn$ such that $q = \frac{n}d$ and $\gcd(n,d) = 1$. For $S \subseteq \rr$ and $r \in \rr$, we set $S_{\ge r} := \{s \in S : s \ge r\}$ and, in a similar manner, we use the notation $S_{> r}$. For disjoint sets $S$ and~$T$, we often write $S \sqcup T$ instead of $S \cup T$ to emphasize that we are taking the union of disjoint sets. A sequence of sets $(S_n)_{n \ge 1}$ is called an \emph{ascending chain} if $S_n \subseteq S_{n+1}$ for every $n \in \nn$ and is said to \emph{stabilize} if there exists $n_0 \in \nn$ such that $S_n = S_{n_0}$ for every $n \ge n_0$.

\medskip
\subsection{Commutative Monoids}

Let $S$ be an (additive) commutative semigroup. The semigroup $S$ is called \emph{cancellative} if for any $b,c,d \in M$ the equality $b+d = c+d$ implies $b=c$. On the other hand, $S$ is called \emph{torsion-free} if for any $b,c \in S$ and $n \in \nn$, the equality $nb = nc$ implies that $b=c$. Throughout this paper, we reserve the term \emph{monoid} for a cancellative, commutative semigroup with identity and, unless otherwise is clear from the context, we will use additive notation for monoids. 
\smallskip

Let $M$ be a monoid. We set $M^\bullet = M \setminus \{0\}$ and, as for groups, we say that $M$ is \emph{trivial} if $M = \{0\}$. We let $\uu(M)$ denote the group of invertible elements of $M$. If $\uu(M)$ is trivial, then $M$ is called \emph{reduced}. The group of all formal differences of elements in $M$ is called the \emph{Grothendieck group} of $M$ and it is denoted by $\gp(M)$. Equivalently, $\gp(M)$ is the unique abelian group up to isomorphism satisfying the following condition: any abelian group containing an isomorphic image of $M$ also contains an isomorphic image of $\gp(M)$. One can readily check that a monoid is torsion-free if and only if its Grothendieck group is torsion-free. The \emph{rank} of $M$, denoted by $\text{rank} \, M$, is the rank of $\gp(M)$ as a $\zz$-module or, equivalently, the dimension of the $\qq$-vector space $\qq \otimes_\zz \gp(M)$.
\smallskip

For $b,c \in M$, we say that $c$ \emph{divides} $b$ \emph{in} $M$ and write $c \mid_M b$ if $b = c + d$ for some $d \in M$. A submonoid~$M'$ of $M$ is called a \emph{divisor-closed submonoid} of~$M$ if every element of~$M$ dividing an element of~$M'$ in~$M$ belongs to $M'$. The monoid $M$ is called a \emph{valuation monoid} if for all $b,c \in M$ either $b \mid_M c$ or $c \mid_M b$. A \emph{maximal common divisor} of a nonempty subset $S$ of $M$ is a common divisor $d \in M$ of~$S$ such that for any other common divisor $d' \in M$ of $S$ the divisibility relation $d \mid_M d'$ implies that $d' - d \in \uu(M)$. The set of all maximal common divisor of a nonempty subset $S$ of~$M$ is denoted by $\mcd_M(S)$.
\smallskip

We say that the monoid $M$ is a \emph{linearly ordered monoid} with respect to a given total order relation $\le$ on $M$ provided that $\le$ is compatible with the operation of $M$ in the following sense: for all $b,c,d \in M$, the inequality $b < c$ implies that $b+d < c+d$. We say that that the monoid $M$ is \emph{linearly orderable} if~$M$ is a linearly ordered monoid with respect to some total order. The linearly orderable monoids can be characterized as the commutative semigroups that are cancellative and torsion-free (see \cite[Section~3]{rG84}).

\medskip
\subsection{Atomicity, ACCP, and Factorizations}

For a subset $S$ of $M$, we let $\langle S \rangle$ denote the smallest submonoid of $M$ containing~$S$. If $M = \langle S \rangle$ for a finite set $S$, then $M$ is called \emph{finitely generated}. An element $a \in M \setminus \uu(M)$ is called an \emph{atom} if whenever $a = b + c$ for some $b,c \in M$ either $b \in \uu(M)$ or $c \in \uu(M)$. We let $\mathcal{A}(M)$ denote the set consisting of all the atoms of $M$. Note that if $M$ is reduced, then $\mathcal{A}(M)$ is contained in every generating set of~$M$. Following Cohn~\cite{pC68}, we proceed to introduce the most important notion in the scope of this survey: atomicity.

\begin{definition}
	Let $M$ be a monoid.
	\begin{itemize}
		\item $b \in M$ is an \emph{atomic} element if either $b$ is invertible or $b$ can be written as a sum of atoms.
		\smallskip
		
		\item $M$ is an \emph{atomic} monoid if every element of $M$ is atomic.
	\end{itemize}
\end{definition}

A subset $I$ of $M$ is called an \emph{ideal} of~$M$ if the set
\[
	I + M := \{b + c : b \in I \text{ and } c \in M\}
\]
is contained in~$I$ or, equivalently, if $I + M = I$. If~$I$ is an ideal of~$M$ such that
\[
	I = b + M := \{b + c : c \in M\}
\]
for some $b \in M$, then $I$ is called \emph{principal}. We say that an element $b_0 \in M$ satisfies the \emph{ascending chain condition on principal ideals} (ACCP) if every ascending chain $(b_n + M)_{n \ge 0}$ of principal ideals of~$M$ eventually stabilizes; that is, there is an $n_0 \in \nn$ such that $b_n + M = b_{n+1} + M$ for every $n \ge n_0$. Accordingly, the monoid $M$ is said to \emph{satisfy the ACCP} if every element of $M$ satisfies the ACCP. Every monoid that satisfies the ACCP must be atomic \cite[Proposition~1.1.4]{GH06}. The converse does not hold. The first study of the connection of atomicity and the ACCP was carried out in 1974 by Grams~\cite{aG74}, where she constructs the first atomic domain that does not satisfy the ACCP. For a generous insight of the connection of atomicity and the ACCP, see \cite{GL23a} and references therein.
\smallskip

Now assume that $M$ is atomic. The free (commutative) monoid on the set of atoms of $M/\uu(M)$ is denoted by $\mathsf{Z}(M)$. An element $z = a_1 + \cdots + a_\ell \in \mathsf{Z}(M)$, where $a_1, \dots, a_\ell \in \ii(M/\uu(M))$, is called a \emph{factorization} in $M$ of \emph{length} $|z| := \ell$. Because the monoid $\mathsf{Z}(M)$ is free, there is a unique monoid homomorphism $\pi \colon \mathsf{Z}(M) \to M/\uu(M)$ such that $\pi(a) = a$ for all $a \in \ii(M/\uu(M))$. For each $b \in M$, we set
\[
	\mathsf{Z}(b) := \pi^{-1}(b) \subseteq \mathsf{Z}(M) \quad \text{ and } \quad \mathsf{L}(b) := \{|z| : z \in \mathsf{Z}(b)\}.
\]
As $M$ is an atomic monoid, the sets $\mathsf{Z}(b)$ and $\mathsf{L}(b)$ are nonempty for every $b \in M$. Following~\cite{AAZ90} and~\cite{HK92}, we say that $M$ is a \emph{bounded factorization monoid} (BFM) if $|\mathsf{L}(b)| < \infty$ for every $b \in M$. It follows from \cite[Corollary~1.3.3]{GH06} that every BFM satisfies the ACCP. The bounded factorization property (along with the finite factorization property) was recently surveyed by Anderson and the second author in~\cite{AG22}. If $|\mathsf{Z}(b)| = 1$ (resp., $|\mathsf{L}(b)| = 1$) for every $b \in M$, then $M$ is called a \emph{unique factorization monoid} (UFM) (resp., a \emph{half-factorial monoid} (HFM)). It follows directly from the corresponding definitions that every UFM is an HFM and also that every HFM is a BFM.

\medskip
\subsection{Integral Domains}

Let $R$ be a commutative ring with identity. As it is customary, we let $R^\times$ denote the group of units of~$R$. The \emph{multiplicative monoid} of $R$, denoted by $R^*$, is the monoid consisting of all regular elements of $R$, meaning the elements that are not zero-divisors. Observe that $R^* = R \setminus \{0\}$ when $R$ is an integral domain, which is the case of interest in this survey. Suppose from now on that $R$ is an integral domain. In this case, we let $\text{qf}(R)$ denote the field of fractions of $R$. All the divisibility notation introduced for monoids can be naturally adapted to the setting of integral domains via the multiplicative monoid $R^*$. For instance, for a nonempty set $S$ consisting of nonzero elements of $R$, we say that a nonzero $d \in R$ is a \emph{maximal common divisor} of~$S$ in $R$ if $d$ is a maximal common divisor of~$S$ in $R^*$. For nonzero elements $r,s \in R$ with $r$ dividing $s$ in $R$ and a nonempty subset $S$ of $R^*$, we write $r \mid_R s$ and $\mcd_R(S)$ instead of $r \mid_{R^*} s$ and $\mcd_{R^*}(S)$, respectively. Let us provide terminology for the most relevant algebraic objects of this survey.

\begin{definition}
	Let $R$ be an integral domain.
	\begin{itemize}
		\item $r \in R^*$ is an \emph{atomic} element if either $r$ is a unit or $r$ factors into irreducibles.
		\smallskip
		
		\item $R$ is an \emph{atomic} integral domain or an \emph{atomic domain} if every nonzero element of $R$ is atomic.
	\end{itemize}
\end{definition}

As mentioned in the introduction, the abelian group $\qf(R)^\times/R^\times$ is denoted by $G(R)$ and called the \emph{group of divisibility} of $R$. Also, recall that $G(R)$ has the following natural order: for all nonzero $r,s \in \qf(R)$, the relation $rR^\times \leq sR^\times$ holds precisely when $\frac{s}{r} \in R$. Here is a characterization of an atomic domain in terms of the divisibility group, which we provide right away given its fundamental nature.

\begin{theorem} \label{atomicity:characterization}
	Let $R$ be an integral domain that is not a field, and let $G(R)$ be the group of divisibility of $R$. Then $R$ is atomic if and only if every positive element of $G(R)$ can be expressed as a product of minimal positive elements of $G(R)$.
\end{theorem}

\begin{proof}
	Let $K$ be the quotient field of $R$. Observe that, for each $a \in R^*$, the coset $a R^\times \in G(R)$ is minimal positive if and only if $a$ is irreducible in~$R$. Indeed, if $a = rs$, then $a R^\times = rR^\times sR^\times$ and, as $a R^\times$ is minimal positive, it must be the case that either $rR^\times$ or $sR^\times$ is not positive; that is, either $r$ or $s$ belongs to $R^\times$. The converse is almost identical.
	
	Assume first that $R$ is atomic. Let $dR^\times$ be a positive element of $G(R)$, that is, $d \in R$ is a nonzero nonunit. Since $R$ is atomic, we can write $d = a_1 \cdots a_n$ for some irreducibles $a_1, \dots, a_n$ of $R$. Since $a_1 R^\times, \dots, a_n R^\times$ are all positive, we obtain our desired expression: $d R^\times = a_1 R^\times \cdots a_n R^\times$.
	
	On the other hand, assume that every positive element of $G(R)$ can be expressed as a product of minimal positive elements of $G(R)$. Let $d$ be a nonzero nonunit of $R$. By assumption, $d R^\times = a_1 R^\times \cdots a_n R^\times$, where $a_1 R^\times, \dots, a_n R^\times$ are minimal positive elements of $G(R)$, and so $a_1, \dots, a_n$ are irreducibles in $R$. Hence we can take $u \in R^\times$ such that $d = u a_1 \cdots a_n$. Thus, $R$ is atomic.
\end{proof}

Since the irreducibles of~$R$ are precisely the atoms of $R^*$, it follows immediately that $R$ is an atomic domain (resp., a UFD) if and only if the monoid~$R^*$ is an atomic monoid (resp., a UFM). Similarly, one can see that $R$ satisfies the ACCP (as defined traditionally for rings) if and only if the monoid $R^*$ satisfies the ACCP. Furthermore, we say that $R$ is a \emph{bounded factorization domain} (resp., \emph{half-factorial domain}) provided that $M$ is a BFM (resp., an HFM); as it is customary in the literature, we let BFD (resp., HFD) stands for bounded factorization domain (resp., half-factorial domain). 
\smallskip

Let $R$ be a commutative ring with identity, and let $M$ be an (additive) commutative semigroup with identity. The commutative ring consisting of all polynomial expressions in an indeterminate $x$ with exponents in $M$ and coefficients in~$R$ is called the \emph{semigroup algebra} of~$M$ over $R$ (or \emph{monoid algebra} when~$M$ is a monoid). Following Gilmer~\cite{rG84}, we will denote the semigroup algebra of~$M$ over $R$ by $R[x;M]$, or simply $R[M]$ if we see no risk of ambiguity. We are mostly interested here in monoid algebras $R[M]$ that are integral domains and a necessary condition for this is that $M$ is cancellative, so we assume for the rest of this section that $M$ is a monoid. It follows from~\cite[Proposition 8.3]{GP74} that if $R[M]$ is an integral domain, then $\dim R[M] \ge 1 + \dim R$. Moreover, when $R$ is a Noetherian domain and $M$ is a torsion-free monoid, it follows from \cite[Corollary~2]{jO88} that $\dim R[M] = \dim R + \text{rank} \, M$.
\smallskip

Assume now that $M$ is a linearly ordered monoid with respect to the total order relation $\le$, and let $R$ be a commutative ring with identity. Then we can write any nonzero element $f \in R[M]$ as $f := c_1 x^{m_1} + \dots + c_k x^{m_k}$ for some nonzero coefficients $c_1, \dots, c_k \in R$ and exponents $m_1, \dots, m_k \in M$ with $m_1 > \dots > m_k$. In this case, we call the set of exponents
\[
	\text{supp} \, f := \{m_1, \dots, m_k\}
\]
the \emph{support} of $f$, and we call
\[
	\deg f := m_1 \quad \text{ and } \quad \text{ord} \, f := m_k
\]
the \emph{degree} and \emph{order} of $f$, respectively. In addition, if $R$ is an integral domain, then one can readily see that the following equalities hold in the monoid algebra $R[M]$:
\begin{equation} \label{eq:degree and order}
	\deg \, f_1 f_2 = \deg \, f_1 + \deg \, f_2 \quad \ \text{ and } \ \quad \text{ord} \, f_1 f_2 = \text{ord} \, f + \text{ord} \, f_2
\end{equation}
for all nonzero $f_1, f_2 \in R[M]$. Observe that when the monoid $M$ is $\nn_0$ under the usual order, we recover the standard notions of support, degree, and order of a polynomial in $R[x]$.

\bigskip
\section{Atomic Monoids}
\label{sec:atomic monoids}

In this section, we discuss some fundamental properties about atomicity, mostly in connection to the ACCP, that hold in the general context of (cancellative and commutative) monoids. We also provide some examples of atomic monoids that will be useful in the next sections of this survey, which are dedicated to atomicity in the special setting of integral domains.

\medskip
\subsection{Atomicity and the ACCP}

We begin by showing that the condition of satisfying the ACCP is stronger than that of being atomic, which is a well-known fact that was first observed back in the sixties by Cohn~\cite{pC68} in the context of integral domains.

\begin{proposition}\footnote{A more general characterization of atomicity via principal ideals has been recently given in~\cite{sT23}.} \label{prop:ACCP monoids are atomic}
	Every monoid satisfying the ACCP is atomic.
\end{proposition}

\begin{proof}
	Let $M$ be a monoid that satisfies the ACCP. Suppose, by way of contradiction, that there exists an element $b_0 \in M$ that is not atomic. Thus, we can take non-invertible elements $b_1, c_1 \in M$ such that $b_1$ is not atomic and $b_0 = b_1 + c_1$. As $b_1$ is not atomic, we can take non-invertible elements $b_2,c_2 \in M$ such that $b_2$ is not atomic and $b_1 = b_2 + c_2$. Proceeding similarly, we end up producing two sequences $(b_n)_{n \ge 1}$ and $(c_n)_{n \ge 1}$ whose terms are non-invertible elements of~$M$ such that $b_n = b_{n+1} + c_{n+1}$ for every $n \in \nn$. Now the fact that
	\[
		b_n + M = b_{n+1} + c_{n+1} + M \subsetneq b_{n+1} + M
	\]
	for every $n \in \nn$ implies that $(b_n + M)_{n \ge 1}$ is an ascending chain of principal ideals of $M$ that does not stabilize, which contradicts that $M$ satisfies the ACCP.
\end{proof}

The converse of the statement of Proposition~\ref{prop:ACCP monoids are atomic} does not hold in general. The most elementary counterexamples witnessing this observation can be found in the class consisting of additive submonoids of $\qq_{\ge 0}$, known as \emph{Puiseux monoids}. The atomicity of Puiseux monoids has been actively studied during the past few years (see the survey~\cite{CGG21} and references therein). The first atomic (Puiseux) monoid not satisfying the ACCP appears as the main ingredient in Grams' construction of the first atomic domain not satisfying the ACCP~\cite[Section~1]{aG74} (we will discuss Grams' construction in the next section).

\begin{example} \label{ex:Grams monoid}
	Let $(p_n)_{n \ge 1}$ be the strictly increasing sequence whose terms are the odd primes, and consider the following Puiseux monoid, often referred to as \emph{Grams' monoid}:
	\begin{equation}
		M := \Big\langle \frac{1}{2^{n-1}p_n} : n \in \nn \Big\rangle.
	\end{equation}
	One can readily verify that $M$ is atomic with $\mathcal{A}(M) = \big\{ \frac{1}{2^{n-1}p_n} : n \in \nn \big\}$ (it also follows from the more general result \cite[Proposition~3.1]{GL23}). On the other hand, observe that $\frac1{2^n} \in M$ for every $n \in \nn_0$, and so $N := \big\langle \frac1{2^n} : n \in \nn_0 \big\rangle$ is a submonoid of $M$. This in turn implies that $\big( \frac1{2^n} + M \big)_{n \ge 1}$ is an ascending chain of principal ideals of $M$ that does not stabilize. Hence $M$ is an atomic Puiseux monoid that does not satisfy the ACCP. Moreover, it is useful to know that each element $b \in M$ can be uniquely written in the following canonical form:
	\begin{equation} \label{eq:Grams monoid canonical decomposition}
		b = \nu_N(b) + \sum_{n \in \nn} c_n \frac{1}{2^{n-1}p_n},
	\end{equation}
	 where $\nu_N(b) \in N$, while $c_n \in \ldb 0, p_n - 1 \rdb$ for all $n \in \nn$ and $c_n = 0$ for almost all $n \in \nn$. To argue the existence of such a canonical sum decomposition, write $b = \sum_{n \in \nn} c_n \frac{1}{2^{n-1}p_n}$ for some sequence $(c_n)_{n \ge 1}$ of nonnegative integer coefficients with only finitely many nonzero terms, which is possible because $M$ is atomic with set of atoms $\big\{ \frac{1}{2^{n-1}p_n} : n \in \nn \big\}$. We can actually rewrite the same sum as in~\eqref{eq:Grams monoid canonical decomposition}, where $\nu_N(b) \in N$, and each nonnegative integer coefficients $c_n$ now satisfying that $0 \le c_n < p_n$ for every $n \in \nn$ (still only finitely many terms of $(c_n)_{n \ge 1}$ are nonzero). Observe that any two sum decompositions of $b$ as in~\eqref{eq:Grams monoid canonical decomposition} satisfying the imposed restrictions must be equal (to check this, just apply for each $n \in \nn$ the $p_n$-adic valuation map to the equality of two such potential sum decompositions).
\end{example}

If a given monoid satisfies the ACCP, then every submonoid preserving invertible elements also satisfies the ACCP. This is a well-known result, and we include it here as it will be helpful at several occasions later.

\begin{proposition} \label{prop:submonoids satisfying ACCP}
	Let $M$ be a monoid satisfying the ACCP, and let $S$ be a submonoid of $M$.  If $\uu(M) \cap S = \uu(S)$, then $S$ also satisfies the ACCP.
\end{proposition}

\begin{proof}
	Assume that $\uu(M) \cap S = \uu(S)$. Let $(b_n + S)_{n \ge 1}$ be an ascending chain of principal ideals in $S$. Then $(b_n + M)_{n \ge 1}$ is an ascending chain of principal ideals in $M$, and so it must stabilize. This means that there exists $N \in \nn$ such that $b_n - b_{n+1} \in \uu(M) \cap S = \uu(S)$ for every $n \ge N$, whence $(b_n + S)_{n \ge 1}$ must also stabilize in $S$. Thus, $S$ satisfies the ACCP.
\end{proof}

Since divisor-closed submonoids preserve invertible elements, we obtain the following corollary.

\begin{corollary} \label{cor:divisor-closed submonoids of an ACCP monoid}
	If a monoid $M$ satisfies the ACCP, then every divisor-closed submonoid of $M$ satisfies the ACCP.
\end{corollary}

The statements in Proposition~\ref{prop:submonoids satisfying ACCP} does not hold if we replace satisfying the ACCP by being atomic. For instance, Grams' monoid $M$ is atomic but contains the non-atomic submonoid $S := \big\langle \frac1{2^n} : n \in \nn_0 \big\rangle$ (clearly, both $\uu(M) \cap S$ and $\uu(S)$ are copies of the trivial monoid).

\medskip
\subsection{Atomicity in Rank-$1$ Monoids from Examples}

In this subsection, we discuss some examples that will be helpful later. We begin by another example of an atomic Puiseux monoid not satisfying the ACCP, which is less natural than the Grams' monoid but will be useful later to show that atomicity does not ascend to monoid algebras over fields.

\begin{example} \label{ex:ad-hoc atomic PM without the ACCP}
	Let $(\ell_n)_{n \ge 1}$ be a strictly increasing sequence of positive integers such that the following inequality $3^{\ell_n - \ell_{n-1}} > 2^{n+1}$ holds for every $n \in \nn$. Now set $A = \{a_n, b_n : n \in \nn\}$, where
	\[	
		a_n := \frac{2^n3^{\ell_n} - 1}{2^{2n} 3^{\ell_n}} \quad \text{and} \quad b_n := \frac{2^n3^{\ell_n} + 1}{2^{2n} 3^{\ell_n}}.
	\]
	The sequence $b_1, a_1, b_2, a_2, \dots$ is strictly decreasing and bounded from above by~$1$: indeed, for every $n \in \nn$, it is clear that $1 > b_n > a_n$ and also that
	\[
		a_n = \frac{1}{2^n} - \frac{1}{2^{2n}3^{\ell_n}} = \frac{1}{2^{n+1}} + \bigg( \frac{1}{2^{n+1}} - \frac{1}{2^{2n}3^{\ell_n}} \bigg) > \frac{1}{2^{n+1}} + \frac{1}{2^{2n+2}3^{\ell_{n+1}}} = b_{n+1}.
	\]
	Consider the Puiseux monoid $M$ generated by $A$. We will prove that, as Grams' monoid, $M$ is an atomic monoid that does not satisfy the ACCP. 
	
	Proving that $M$ is atomic amounts to showing that $\mathcal{A}(M) = A$. Assume, towards a contradiction, that this is not the case. Then we can take $n \in \nn$ such that $M = \langle A \setminus \{a_n\} \rangle$ or $M = \langle A \setminus \{b_n\} \rangle$. We split the rest of the argument into the following two cases.
	\smallskip
	
	\noindent \textsc{Case 1:} $M = \langle A \setminus \{a_n\} \rangle$. In this case,
	\begin{equation} \label{eq:testing atoms I}
		a_n = \sum_{i=1}^N \alpha_i a_i + \sum_{i=1}^N \beta_i b_i
	\end{equation}
	for some $N \in \nn_{\ge n}$ and nonnegative integer coefficients $\alpha_i$'s and $\beta_i$'s (for every $i \in \ldb 1,N \rdb$) such that $\alpha_n = 0$ and either $\alpha_N > 0$ or $\beta_N > 0$. Since the sequence $b_1, a_1, b_2, a_2, \dots$ is strictly decreasing, $\alpha_i = \beta_i = 0$ for $i \in \ldb 1, n \rdb$. We observe that the equality $\alpha_i = \beta_i$ does not hold for every $i \in \ldb n+1, N \rdb$ as, otherwise,
	\[
		a_n = \sum_{i=n+1}^N \alpha_i a_i + \sum_{i=n+1}^N \alpha_i b_i= \sum_{i = n+1}^N \alpha_i \frac{1}{2^{i-1}},
	\]
	which is not possible given that $3 \mid \mathsf{d}(a_n)$. Now set
	\[
		m = \max \big\{ i \in \ldb n+1, N \rdb : \alpha_i \neq \beta_i \big\}.
	\]
	First, we assume that $\alpha_m > \beta_m$. This allows us to rewrite~(\ref{eq:testing atoms I}) as
	\begin{equation} \label{eq:testing atoms III}
		a_n = (\alpha_m - \beta_m) \frac{2^m3^{\ell_m} - 1}{2^{2m} 3^{\ell_m}} + \sum_{i=m}^N \beta_i \frac{1}{2^{i-1}} + \! \! \sum_{i=n+1}^{m-1} \frac{\alpha_i(2^i3^{\ell_i} - 1) + \beta_i(2^i3^{\ell_i} + 1)}{2^{2i} 3^{\ell_i}}.
	\end{equation}
	After multiplying~(\ref{eq:testing atoms III}) by $2^{2N}3^{\ell_m}$, we see that each summand involved in such an equality, except perhaps $2^{2N - 2m}(\alpha_m - \beta_m)(2^m 3^{\ell_m} - 1)$, is divisible by $3^{\ell_m -\ell_{m-1}}$. As a result, $3^{\ell_m - \ell_{m-1}}$ also divides $\alpha_m - \beta_m$. Now because $a_m > b_{m+1} > \frac{1}{2^{m+1}}$,
	\begin{equation*}
		a_n \ge \alpha_m a_m \ge (\alpha_m - \beta_m) b_{m+1} \ge 3^{\ell_m - \ell_{m-1}} b_{m+1} > \frac{3^{\ell_m - \ell_{m-1}}}{2^{m+1}} > 1,
	\end{equation*}
	which contradicts that $1$ is an upper bound for the sequence $(a_n)_{n \ge 1}$. Similarly, we can produce a contradiction after assuming that $\beta_m > \alpha_m$.
	\smallskip
	
	\noindent \textsc{Case 2:} $M = \langle A \setminus \{b_n\} \rangle$. In this case, one can write
	\begin{equation} \label{eq:testing atoms II}
		b_n - \alpha_n a_n = \sum_{i=n+1}^N \alpha_i a_i + \sum_{i=n+1}^N \beta_i b_i
	\end{equation}
	for some nonnegative coefficients $\alpha_i$'s (for every $i \in \ldb n, N \rdb$) and $\beta_j$'s (for every $j \in \ldb n+1,N \rdb$) such that at least one of the inequalities $\alpha_N > 0$ and $\beta_N > 0$ holds. From $2 a_n > b_n$, we obtain that $\alpha_n \in \{0,1\}$. As before, we can take $m \in \ldb n+1, N \rdb$ with $\alpha_m \neq \beta_m$ and further assume that $m$ has been taken as large as it could possibly be. If $\alpha_m > \beta_m$, then
	\begin{equation}  \label{eq:testing atoms IIII}
		b_n - \alpha_n a_n = (\alpha_m - \beta_m) \frac{2^m3^{\ell_m} - 1}{2^{2m} 3^{\ell_m}} + \sum_{i=m}^N \beta_i \frac{1}{2^{i-1}} + \! \! \sum_{i=n+1}^{m-1} \frac{\alpha_i(2^i3^{\ell_i} - 1) + \beta_i(2^i3^{\ell_i} + 1)}{2^{2i} 3^{\ell_i}}.
	\end{equation}
	Because $\mathsf{d}(b_n - \alpha_n a_n) \in \{2^{n-1} 3^{\ell_n}, 2^{2n} 3^{\ell_n}\}$, we can multiply~\eqref{eq:testing atoms IIII} by $2^{2N}3^{\ell_m}$ to deduce that $3^{\ell_m - \ell_{m-1}}$ must divide $\alpha_m - \beta_m$. At this point, we can proceed as we did in Case~1 to show that $b_n > 1$, the desired contradiction. Similarly, we can produce a contradiction after assuming that $\alpha_m < \beta_m$.
	\smallskip
	
	Hence $M$ is an atomic monoid. On the other hand, one can easily verify that~$M$ does not satisfy the ACCP: indeed, as $\frac{1}{2^n} = a_{n+1} + b_{n+1} \in M$ for every $n \in \nn_0$, it follows that $\big( \frac1{2^n} + M \big)_{n \ge 1}$ is an ascending chain of principal ideals of~$M$ that does not stabilize.
\end{example}

We have mentioned in the introduction that every BFM satisfies the ACCP. The converse does not hold in general. The following example, which we will revisit in next sections, illustrates this observation.

\begin{example} \label{ex:prime reciprocal PM}
	Consider the Puiseux monoid generated by the reciprocal of all prime numbers; that is,
	\[
		M := \Big\langle \frac1p : p \in \pp \Big\rangle.
	\]
	One can readily show that $M$ is atomic with $\mathcal{A}(M) = \big\{ \frac1{p} : p \in \pp \big\}$. In addition, it is not hard to argue that for each $q \in M$ there is a unique $N(q) \in \nn_0$ and a unique sequence of nonnegative integers $(c_p(q))_{p \in \pp}$ such that
	\[
		q = N(q) + \sum_{p \in \pp} c_p(q) \frac{1}{p},
	\]
	where $c_p(q) \in \ldb 0, p - 1 \rdb$ for all $p \in \pp$ and $c_p(q) = 0$ for all but finitely many $p \in \pp$. For each $q \in M$, set $s(q) := \sum_{p \in \pp} c_p(q)$. Observe that if $q' \mid_M q$ for some $q' \in M$, then $N(q') \le N(q)$. Also, it follows that if $q'$ is a proper divisor of $q$ in $M$, then $N(q') = N(q)$ implies that $s(q') < s(q)$. As a result, each sequence $(q_n)_{n \ge 1}$ in $M$ such that $q_{n+1} \mid_M q_n$ for every $n \in \nn$ must stabilize. Hence $M$ satisfies the ACCP.
	
	To argue that $M$ is not a BFM, it suffices to notice that $\pp \subseteq \mathsf{L}(1)$ as, for each $p \in \pp$, the element~$1$ can be written as the sum of $p$ copies of the atom $\frac1p$ (indeed, $\mathsf{L}(1) = \pp$).
\end{example}

We conclude with a very simple example of a Puiseux monoid that is a BFM but is not an HFM and contains elements with infinitely many factorizations.\footnote{An atomic monoid whose elements have only finitely many factorizations is called a \emph{finite factorization monoid}.}

\begin{example} \label{ex:PM that is BF but not FF}
	Let $M$ denote the Puiseux monoid $\{0\} \cup \qq_{\ge 1}$. Observe that $M$ is atomic with set of atoms $\qq \cap [1,2)$. In addition, as $0$ is not a limit point of $M^\bullet$, it follows from \cite[Proposition~4.5]{fG19} that~$M$ is a BFM. Since $3 = 1 + 1 + 1$ and $3 = \frac32 + \frac32$, we see that $M$ is not an HFM. Finally, observe that $3$ has infinitely many factorizations: indeed, $\big( \frac32 - \frac1n\big) + \big( \frac32 + \frac1n \big)$ determines a length-$2$ factorization of $3$ in $M$ for every $n \in \nn$ with $n \ge 3$.
\end{example}

\bigskip
\section{Classes of Atomic Domains}
\label{sec:atomic domains}

In this section, we identify some classes of atomic domains, putting some emphasis on the nested classes of integral domains indicated in the following diagram of implications:

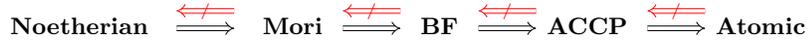
\begin{figure}[h]
	\begin{tikzcd}
		\textbf{ Noetherian } \ \arrow[r, Rightarrow]	\arrow[red, r, Leftarrow, "/"{anchor=center,sloped}, shift left=1.7ex] & \ \textbf{ Mori } \arrow[r, Rightarrow]  \arrow[red, r, Leftarrow, "/"{anchor=center,sloped}, shift left=1.7ex] & \textbf{ BF }  \arrow[r, Rightarrow] \arrow[red, r, Leftarrow, "/"{anchor=center,sloped}, shift left=1.7ex]  & \textbf{ACCP } \arrow[r, Rightarrow] \arrow[red, r, Leftarrow, "/"{anchor=center,sloped}, shift left=1.7ex]  & \textbf{Atomic}
	\end{tikzcd}
	\caption{The implications in the diagram show the known inclusions among the indicated classes of atomic domains. 
		The diagram also emphasizes (with red marked arrows) that none of the shown implications is reversible.}
	\label{fig:Noetherian to ATM}
\end{figure}

\medskip
\subsection{Atomic Integral Domains and the ACCP} 

As an immediate consequence of Proposition~\ref{prop:ACCP monoids are atomic} we obtain that every integral domain that satisfies the ACCP is atomic. We record this as our next proposition to be referenced later. However, the main purpose of this section is to provide examples illustrating that the converse does not hold in general.

\begin{proposition} \label{prop:ACCP domain are atomic}
	If an integral domain satisfies the ACCP, then it is atomic.
\end{proposition}

The construction of the first atomic domain that does not satisfy the ACCP was carried out by Grams~\cite{aG74} back in 1974. We proceed to discuss Gram's construction. Let
\[
	M = \Big\langle \frac{1}{2^{n-1}p_n} : n \in \nn \Big\rangle
\]
be Grams' monoid, which was introduced in Example~\ref{ex:Grams monoid} (here $(p_n)_{n \ge 1}$ is the strictly increasing sequence whose underlying set is $\pp \setminus \{2\}$). We have already verified in the same example that $M$ is atomic but does not satisfy the ACCP. Let $F$ be a field, and consider the following multiplicative subset of the monoid algebra $F[M]$:
\[
	S := \{f \in F[M] : \text{ord} \, f = 0\}.
\]
The localization $F[M]_S$ of the monoid algebra $F[M]$ at $S$ is often referred to as the \emph{Grams' domain} over the field~$F$. The following theorem is \cite[Theorem~1.3]{aG74}.

\begin{theorem} \label{thm:Grams domain}
	For any field $F$, the Grams' domain over $F$ is atomic but does not satisfy the ACCP.
\end{theorem}

\begin{proof}
	Fix a field $F$, let $M$ denote Grams' monoid, and let $S$ be the multiplicative set in the construction of the Grams' domain over~$F$. As seen in Example~\ref{ex:Grams monoid}, the monoid $M$ is atomic with set of atoms $\mathcal{A}(M) = \big\{ \frac{1}{2^{n-1}p_n} : n \in \nn \big\}$. Let us prove that the Grams' domain $R := F[M]_S$ is atomic.
	\smallskip
	
	First, let us show that $x^m$ is an atomic element of $F[M]_S$ for every $m \in M$. To do so, fix $a \in \mathcal{A}(M)$, and let us verify that the element $x^a$ is irreducible in $F[M]_S$. Write $x^a = \frac{f_1}{s_1} \cdot \frac{f_2}{s_2}$ for some $f_1, f_2 \in F[M]$ and $s_1, s_2 \in S$. From the equality $x^a s_1 s_2 = f_1 f_2$, we infer that $\text{ord} \, f_1 + \text{ord} \, f_2 = \text{ord} \, f_1 f_2 = a$. As $\text{ord} \, f_1$ and $\text{ord} \, f_2$ are contained in $M$ and $a$ is an atom of $M$, either $\text{ord} \, f_1 = 0$ or $\text{ord} \, f_2 = 0$, which means that either~$f_1$ or $f_2$ belongs to $S$. Therefore $x^a$ is irreducible in $F[M]_S$. Now the fact that $M$ is atomic immediately implies that for each $m \in M$ the element $x^m$ is atomic in $F[M]_S$.
	\smallskip
	
	Proving that $F[M]_S$ is atomic amounts to showing that every nonzero nonunit in $F[M]$ factors into irreducibles in $F[M]_S$. We first set $N := \big\langle \frac1{2^n} : n \in \nn \big\rangle$ and prove the following.
	\smallskip
	
	\noindent \textsc{Claim.} For each nonzero $b \in M$, the set $\{d \in N : d \mid_M b \}$ has a maximum element.
	\smallskip
	
	\noindent \textsc{Proof of Claim.} Fix a nonzero $b \in M$. As we have seen in Example~\ref{ex:Grams monoid}, we can uniquely write
	\begin{equation} \label{eq:Grams monoid canonical decomposition again}
		b = \nu_N(b) + \sum_{n \in \nn} c_n \frac{1}{2^{n-1}p_n},
	\end{equation}
	where $\nu_N(b) \in N$, while $c_n \in \ldb 0, p_n - 1 \rdb$ for all $n \in \nn$ and $c_n = 0$ for almost all $n \in \nn$. We proceed to show that $\nu_N(b)$ is the maximum element of the set
	\[
		D_b := \{d \in N : d \mid_M b \}.
	\]
	It is clear that $\nu_N(b) \in D_b$. Now for any $d \in D_b$, we can write $b = d + b'$ for some $b' \in M$ and after decomposing $b'$ as in~\eqref{eq:Grams monoid canonical decomposition again}, we also obtain a decomposition for $b$, and so the uniqueness of such a decomposition of $b$ will ensure that $d \mid_N \nu_N(b)$. Hence $\nu_N(b) = \max D_b$, and the claim follows.
	\smallskip
	
	We are in a position to argue that any nonunit $f := \sum_{j=1}^n r_j x^{m_j}$ in $F[M]^*$ (with $m_1 > \dots > m_n$ and $r_1 \cdots r_n \neq 0$) factors into irreducibles in $F[M]_S$. To do this, set
	\[
		q = \min \{\nu_N(m_j) : j \in \ldb 1,n \rdb\} \quad \text{ and } \quad k = \max \{ j \in \ldb 1,n \rdb : \nu_N(m_j) = q \},
	\]
	where $\nu_N(m_j)$ is as in the sum decomposition~\eqref{eq:Grams monoid canonical decomposition}. The minimality of $q$ and the fact that $q \mid_M d$ for every $d \in N_{\ge q}$ ensures that $x^q$ divides $x^{m_j}$ in $F[M]$ for every $j \in \ldb 1,n \rdb$. Since we already know that $x^q$ is an atomic element of $F[M]_S$, after replacing $f$ by $\frac{f}{x^q}$, we can assume that $q=0$, which means that~$0$ is the only common divisor in $M$ of $\text{supp} \, f$ that belongs to $N$. Now write
	\[
		f = \frac{f_1}{s_1} \cdots \frac{f_\ell}{s_\ell}
	\]
	for some nonunits $f_1, \dots, f_\ell \in F[M]$ and $s_1, \dots, s_\ell \in S$. Then the equality $f s_1 \cdots s_\ell = f_1 \cdots f_\ell$ holds in $F[M]$. Since $\nu_N(m_k) = 0$, it follows from the uniqueness of~\eqref{eq:Grams monoid canonical decomposition} that $\mathsf{Z}_M(m_k)$ is a singleton. Also, as $\nu_N(m_k) = 0$, for each $i \in \ldb k+1, n \rdb$ the inequality $\nu_N(m_i) > 0$ implies that $m_i \nmid_M m_k$. As a result, the coefficient of $x^{m_k}$ in the polynomial expression $f s_1 \cdots s_\ell$ is $r_k s_1(0) \cdots s_\ell(0)$, which is different from $0$ because $s_1, \dots, s_\ell \in S$. Hence $m_k = q_1 + \dots + q_\ell$ for some $q_1, \dots, q_\ell \in M$ with $q_i \in \text{supp} \, f_i$ for every $i \in \ldb 1, \ell \rdb$. As $q_i \in \text{supp} \, f_i$ and $f_i \notin S$ for each $i \in \ldb 1, \ell \rdb$, it follows that the elements $q_1, \dots, q_\ell$ are nonzero. This, along with the fact that $M$ is atomic, ensures that $\ell$ is at most the length of the only factorization in $\mathsf{Z}_M(m_k)$. Then, after assuming that~$\ell$ was taken as large as it could possible be, one obtains that $\frac{f_1}{s_1}, \dots, \frac{f_\ell}{s_\ell} \in \mathcal{A}(F[M]_S)$, and so $f$ is atomic in $F[M]_S$. Hence we conclude that $F[M]_S$ is an atomic domain.
	\smallskip
	
	Finally, the fact that $R$ does not satisfy the ACCP follows immediately once we observe that $\big(x^{1/2^n} R \big)_{n \ge 1}$ is an ascending chain of principal ideals of $F[M]_S$ that does not stabilize.
\end{proof}

We can use the fact that the Grams' domains are atomic domains that do not satisfy the ACCP to show that there are polynomial rings that are atomic but do not satisfy the ACCP. This result was first proved by Li and the second author \cite[Proposition~3.6]{GL22}.

\begin{proposition} \label{prop:atomic polynomial rings without the ACCP}
	Let $F$ be a field, and let $R$ be the Grams' domain over $F$. Then $R[x]$ is an atomic domain that does not satisfy the ACCP.
\end{proposition}

\begin{proof}
	Let $M$ and $S$ be the Puiseux monoid and the multiplicative set in the construction of the Grams' domain over $F$, respectively. Observe that the submonoid $N := \langle \frac1{2^n} : n \in \nn \rangle$ of~$M$ is a valuation monoid and, therefore, for any $q_1, q_2 \in N$ the conditions $q_1 \le q_2$ and $q_1 \mid_N q_2$ are equivalent. As we have shown in Example~\ref{ex:Grams monoid}, each $b \in M$ can be uniquely written as follows:
	\[
		b = \nu(b) + \sum_{n \in \nn} c_n \frac{1}{2^{n-1}p_n},
	\]
	where $\nu(b) \in N$, while $c_n \in \ldb 0, p_n - 1 \rdb$ for all $n \in \nn$ and $c_n = 0$ for almost all $n \in \nn$. Now we define the map $\bar{\nu} \colon F[t;M]^* \to N$ by
	\[
		\bar{\nu} \colon \sum_{i=1}^k c_i t^{b_i} \mapsto \min \{\nu(b_i) : i \in \ldb 1,k \rdb\}
	\]
	for any canonically-written nonzero polynomial expression $ \sum_{i=1}^k c_i t^{b_i}$.
	
	To argue that $R[x]$ is an atomic domain, fix a nonzero nonunit polynomial $p(x) := \sum_{i=0}^n f_i(t) x^i \in R[x]$. After replacing $p(x)$ by one of its associates, we can assume that $f_i(t) \in F[t;M]$ for every $i \in \ldb 0,n \rdb$. For each $i \in \ldb 0,n \rdb$, the fact that~$N$ is a valuation monoid ensures that $\frac{f_i(t)}{t^{\bar{\nu}(f_i)}} \in R$, and so $\mathsf{L}_R\big(\frac{f_i(t)}{t^{\bar{\nu}(f_i)}}\big)$ is bounded be virtue of \cite[Theorem~1.3]{aG74}. Now set
	\[
		q := \min \{\bar{\nu}(f_i) : i \in \ldb 0, n \rdb\} \in N,
	\]
	and then take $s \in \ldb 0,n \rdb$ such that $\bar{\nu}(f_s) = q$. Once again the fact that $N$ is a valuation monoid allows us to write $p(x) = t^q p'(x)$ for some $p'(x) \in R[x]$. Since the monomials in $F[t;M]$ that are irreducibles remain irreducibles in $R$, the fact that $M$ is atomic ensures that  $t^q$ factors into irreducibles in $R$, and so in $R[x]$. To argue that the polynomial $p'(x)$ also factors into irreducibles in $R[x]$, write $p'(x) = a_1 \cdots a_k b_1(x) \cdots b_\ell(x)$ for some nonunits $a_1, \dots, a_k \in R$ and some polynomials $b_1(x), \dots, b_\ell(x) \in R[x]$ with $\deg b_i(x) \ge 1$ for every $i \in \ldb 1, \ell \rdb$. Because the coefficient $\frac{f_s(t)}{{t^q}}$ of $x^s$ has a bounded set of lengths in $R$, and the inequality
	\[
		k + \ell \le \max \mathsf{L}_R\Big(\frac{f_s(t)}{{t^q}}\Big) + \deg p'(x)
	\]
	holds, we can assume that $k+\ell$ was taken as large as it could possibly be. This guarantees that $a_1 \cdots a_k b_1(x) \cdots b_\ell(x)$ is a factorization of $p'(x)$ in $R[x]$. Hence $R[x]$ is atomic.
	\smallskip
	
	Finally, the fact that $R[x]$ does not satisfy the ACCP follows immediately from the fact that $R$ does not satisfy the ACCP as $R^*$ is a divisor-closed submonoid of $R[x]^*$.
\end{proof}

Constructing atomic domains that do not satisfy the ACCP is always nontrivial: the constructions shown in Theorem~\ref{thm:Grams domain} (Grams' construction) and Proposition~\ref{prop:atomic polynomial rings without the ACCP} are, among those in the literature, one of the less technical constructions. A generalization of Grams' construction has been recently provided in \cite[Theorem~3.3]{GL23}, yielding a larger class of atomic domains not satisfying the ACCP. The search for atomic domains not satisfying the ACCP has attracted the attention of several authors. Back in the eighties, Zaks~\cite{aG74} constructed the first atomic monoid algebra not satisfying the ACCP. In the nineties, Roitman~\cite{mR93} constructed an atomic domain whose polynomial extension was not atomic, yielding as a result another example of an atomic domain not satisfying the ACCP. Further constructions have been crafted recently: in \cite{BC19}, Boynton and the first author constructed an atomic pullback ring that does not satisfy the ACCP, while a more thorough investigation of the class of atomic domains not satisfying the ACCP was carried out by Li and the first author in~\cite{GL23a}, where further constructions of atomic domains not satisfying the ACCP were provided. Finally, an example of a non-commutative atomic ring not satisfying the ACCP was recently constructed by Bell et al. in~\cite{BBNS23}.
\smallskip

We conclude this subsection with an example of an integral domain that satisfies the ACCP but is not a BFD.

\begin{example}
	We saw in Example~\ref{ex:prime reciprocal PM} that the Puiseux monoid $M := \big\langle \frac1p : p \in \pp \big\rangle$ satisfies the ACCP but is not a BFM. Let $F$ be a field. From the fact that $M$ satisfies the ACCP, it is not hard to infer that the monoid algebra $F[M]$ also satisfies the ACCP. To argue that $F[M]$ is not a BFM, first observe that the set $S := \{cx^q : c \in F^\times \text{ and } q \in M\}$ is a divisor-closed submonoid of $F[M]^*$ whose reduced monoid, $S/\uu(S)$, is isomorphic to $M$. Thus, from the fact that $M$ is not a BFM, we now obtain that $F[M]$ is not a BFD.
\end{example}

\medskip
\subsection{Mori and Noetherian Domains}

Noetherian domains, which are arguably the most important class of rings in commutative algebra, are atomic as a consequence of Proposition~\ref{prop:ACCP domain are atomic}. Indeed, Noetherian domains are BFDs, and this will follow as a direct consequence of the next proposition on Mori domains. Recall that a Mori domain is an integral domain where every ascending chain of divisorial ideals stabilizes, and so every Noetherian domain is a Mori domain. The term `Mori domain' was coined by Querr\'e~\cite{jQ} honoring the work of Mori (see~\cite{vB00} for a survey on Mori domains). Before proving that every Mori domain is a BFD, we highlight the same result also holds in the more general setting of monoids \cite[Theorem~2.2.9]{GH06} and, as a consequence, each Krull monoid (resp., domain) is a BFM (resp., BFD).\footnote{It is actually true that every Krull monoid is a finite factorization monoid, which is even stronger than the fact that Krull monoids are BFDs}.

\begin{proposition} \label{prop:Mori domains are BFD}
	Every Mori domain is a BFD and, therefore, an atomic domain.
\end{proposition}

\begin{proof}
	Let $R$ be a Mori domain, and let us prove that $R$ is a BFD. Since every principal ideal is divisorial,~$R$ must satisfy the ACCP and, therefore, it must be atomic. As a result, we are done once we show that the set of lengths of every element of $R$ is bounded.
	
	To do so, fix a nonzero nonunit $x \in R$. From the fact that every ascending chain of divisorial ideals of $R$ stabilizes, one obtains that the set of divisorial prime ideals containing $x$ is nonempty and finite. Thus, the set of prime divisorial ideals minimal over $xR$ in $R$ is nonempty and has finitely many ideals: let $P_1, \dots, P_n$ denote such ideals. In addition, one can verify that $\bigcap_{j=1}^n P_j$ is the radical of $xR$. In order to complete our proof, we need the following claim.
	\smallskip
	
	\noindent \textsc{Claim.} If $P$ is a prime divisorial ideal minimal over $xR$, then $\bigcap_{n \in \nn} P^n$ is zero.
	\smallskip
	
	\noindent \textsc{Proof of Claim.} Since $R$ is Mori, the localization ring $R_P$ is also Mori. Since $P$ is minimal over $xR$, it follows that $P_P$ is the only prime ideal of $R_P$ containing~$x$. Then the radical ideal of $xR_P$ is $P_P$. In addition, every divisorial ideal of $R_P$ is the $v$-ideal of a finitely generated ideal, and so we can pick a finite subset $S$ of $R_P$ such that $xR_P = S_v$. Because $S$ is finite and $P_P$ is the radical of $S_v$, there exists $k \in \nn$ such that $s^k \in P_P$ for all $s \in S$. This implies that $S^{k |S|} \subseteq x R_P$, which implies that $P_P^{k |S|} \subseteq (S^{k |S|})_v \subseteq x R_P$. This, in turn, implies that
	\begin{equation} \label{eq:intersection inclusion Mori}
		\bigcap_{n \in \nn} P^n \subseteq \bigcap_{n \in \nn} P_P^n \subseteq \bigcap_{n \in \nn} P_P^{nk|S|} \subseteq \bigcap_{n \in \nn} x^n R_P.
	\end{equation}
	One can easily check that in a Mori domain, every nonempty set of principal ideals with nonzero intersection has a minimal element. Therefore if $\bigcap_{n \in \nn} P^n$ were nonzero, then by~\eqref{eq:intersection inclusion Mori} the set of principal ideals $\big\{ x^n R_P : n \in \nn \big\}$ would contain a minimal element, and so $x$ would be a unit of $R_P$, which is not true. Hence $\bigcap_{n \in \nn} P^n$ is zero, and the claim is proved.
	\smallskip
	
	By virtue of the established claim, for every $j \in \ldb 1,n \rdb$, we can pick $m_j \in \nn$ such that $x \notin P_j^{m_j}$. Now set $m := \max \{m_1, \dots, m_n\}$, and let us argue that each factorization of $x$ has length at most $mn$. 
	
	Suppose, by way of contradiction, that there exist $\ell \in \nn$ with $\ell > mn$ and $a_1, \dots, a_\ell \in \mathcal{A}(R)$ such that $x = a_1 \cdots a_\ell$. For each $j \in \ldb 1, \ell \rdb$, the inclusion $xR \subseteq a_j R$ implies that any prime divisorial ideal minimal over $a_jR$ is minimal over $xR$, and so $a_j \in P$ for some  $P \in \{P_1, \dots, P_n\}$. As a consequence, we see that $x = a_1 \cdots a_\ell \in P_1^{k_1} \cdots P_n^{k_n}$ for some $k_1, \dots, k_n \in \nn_0$ such that $k_1 + \dots + k_n = \ell$. Now the fact that $\ell > mn$ ensures that $k_j > m_j$ for some $j \in \ldb 1,n \rdb$. This in turn implies that $x \in P_j^{k_j} \subseteq P_j^{m_j}$, which contradicts the choice of $m_j$. Hence we conclude that $R$ is a BFD.
\end{proof}

As every Noetherian domain is a Mori domain, we obtain the following corollary.

\begin{corollary} \label{cor:Noetherian implies BF}
	Every Noetherian domain is a BFD and, therefore, an atomic domain.
\end{corollary}

The converse of Proposition~\ref{prop:Mori domains are BFD} does not hold in general. The following example sheds some light upon this observation.

\begin{example}
	Consider the rank-one additive monoid $M := \{0\} \cup \qq_{\ge 1}$. We have already seen in Example~\ref{ex:PM that is BF but not FF} that $M$ is a BFM with $\mathcal{A}(M) = \qq \cap [1,2)$. Hence it follows from \cite[Theorem~13]{AJ15} that the monoid algebra $F[M]$ is a BFD. We proceed to argue that $F[M]$ is not a Mori domain. Fix $r \in \qq$ and consider the following set:
	\[
		I_r := \{ f \in F[M] : \text{ord} \, f \ge r\}.
	\]
	As $\text{ord}(f+g) \ge \min \{ \text{ord} \, f, \text{ord} \, g \}$ and $\text{ord}(fg) = \text{ord} \, f + \text{ord} \, g$ for all $f,g \in F[M]$, it follows that~$I_r$ is a fractional ideal of $F[M]$. Now observe that, because the Grothendieck group of $M$ is $\qq$, each nonzero element in the quotient field $F(M)$ of $F[M]$ can be uniquely written as $c x^q \frac fg$ for some $c \in F^\times$, $q \in \qq$, and $f,g \in F[\qq]$ such that $\text{ord} \, f = \text{ord} \, g = 1$; indeed, as the group algebra $F[\qq]$ is a GCD-domain by \cite[Theorem~5.2]{GP74}, we can further assume that $f$ and $g$ are relatively primes in $F[\qq]$. Now fix $c x^q \frac fg$ as before, and observe that $cx^q \frac fg \in I_r^{-1}$ if and only if $c x^q f I_r \subseteq g F[M]$, which implies that $cx^{q+r}f = gh$ for some $h \in F[M]$. Now the fact that $f$ and $g$ are relatively primes in $F[\qq]$ ensures that $g$ belongs to $F[\qq]^\times$; that is, $g$ is a monomial in $F[\qq]$. Thus,
	\[
		I_r^{-1} = \big\{ f \in F[\qq] : f I_r \subseteq F[M] \big\} = \big\{ f \in F[\qq] : \text{ord} \, f \ge 1-r \big\}.
	\]
	Therefore, for each $r \in \qq_{> 1}$, it follows that $(I_r^{-1})^{-1} = I_r$, and so $I_r$ is a $v$-ideal of $F[M]$. As a result, $\big( I_{1 + 1/n} \big)_{n \in \nn}$ is an ascending chain of $v$-ideals in $F[M]$ that does not stabilize. Hence $F[M]$ is not a Mori domain.
\end{example}

\bigskip
\section{The $D+M$ Construction}

In this section we consider the classical $D+M$ construction. The $D+M$ construction allows us to obtain a vast repository of interesting counterexamples in both commutative algebra and factorization theory. Let $T$ be an integral domain, and let $K$ and~$M$ be a subfield of $T$ and a nonzero maximal ideal of $T$, respectively, such that $T = K + M$. For a subdomain $D$ of $K$, set $R = D + M$. Observe that the sum in $D+M$ is a direct sum of abelian groups, and also that $M$ is also a maximal ideal of $R$. This construction was introduced and first studied by Gilmer \cite[Appendix II]{rG68} in the special setting of valuation domains, and then it was investigated simultaneously by Brewer and Rutter~\cite{BR76} and by Costa, Mott, and Zafrullah~\cite{CMZ78} for arbitrary integral domains.

Before discussing how atomicity behaves under the $D+M$ construction, we need to collect some information about units and irreducibles in the $D+M$ construction. When we work with the $D+M$ construction, we will often denote an element of $T$ by $\alpha + m$, tacitly assuming that $\alpha \in K$ and $m \in M$.

\begin{lemma} \label{lem:units of the D+M construction}
	Let $T$ be an integral domain, and let $K$ and $M$ be a subfield of $T$ and a nonzero maximal ideal of $T$, respectively, such that $T = K + M$. For a subdomain $D$ of $K$, set $R = D + M$. Then the following statements hold.
	\begin{enumerate}
		\item $R^\times \cap (1+M) = T^\times \cap (1+M)$.
		\smallskip
				
		\item $R^\times = T^\times \cap R$ if and only if $D$ is a field.
	\end{enumerate}
\end{lemma}

\begin{proof}
	(1) Clearly, $R^\times \cap (1+M) \subseteq T^\times \cap (1+M)$. For the reverse inclusion, it suffices to observe that if $1+m$ is a unit of $T$ for some $m \in M$, then the inverse of $1+m$ in $T$ has the form $1+m'$ for some $m' \in M$ (this is because $T^\times \cap M$ is empty), and so the fact that $1+m' \in R$ implies that $1+m$ is also a unit of $R$.  
	\smallskip
	
	(2) For the direct implication, take a nonzero $\alpha \in D$. As $\alpha \in K^\times \subseteq T^\times$ and $D \subseteq R$, it follows that $\alpha \in T^\times \cap R = R^\times$, and so $\alpha^{-1} \in K \cap R = D$. Hence $D$ is a field. For the reverse implication, suppose that $D$ is a field. Clearly, $R^\times \subseteq T^\times \cap R$. Towards the reverse inclusion, take $\alpha + m_1 \in T^\times \cap R$, and then let $\beta + m_2$ be the inverse of $\alpha + m_1$ in $T$. As $\alpha$ is contained in the field $D$, it follows from the equality $(\alpha + m_1)(\beta + m_2) = 1$ that $\beta = \alpha^{-1} \in D$. Hence $\beta + m_2 \in R$, which implies that $\alpha + m_1 \in R^\times$. Thus, the inclusion $T^\times \cap R \subseteq R^\times$ also holds.
\end{proof}

Let us turn our attention to irreducibility.

\begin{lemma} \label{lem:irreducibles of the D+M construction}
	Let $T$ be an integral domain, and let $K$ and $M$ be a subfield of $T$ and a nonzero maximal ideal of $T$, respectively, such that $T = K + M$. For a subdomain $D$ of $K$, set $R = D + M$. Then $\ii(R) \subseteq T^\times \cup \ii(T)$. Moreover, the following statements hold.
	\begin{enumerate}
		\item If $D$ is a field, then $\ii(R) \subseteq \ii(T)$. 
		\smallskip
		
		\item For each $m \in M$, the element $m$ (resp., $1+m$) is irreducible in $R$ if and only if $m$ (resp., $1+m$) is irreducible in $T$.
	\end{enumerate}
\end{lemma}

\begin{proof}
	To prove the inclusion $\ii(R) \subseteq T^\times \cup \ii(T)$, take $a = \alpha + m$ in $\ii(R)$, where $\alpha \in D$ and $m \in M$. Observe that $a \in D^* \subseteq T^\times$ when $m = 0$ and show that $a \in \mathcal{A}(T)$ otherwise. Assume that $m \neq 0$. Take $s,t \in T$ such that $a = st$, and consider the following two cases.
	\smallskip
	
	\noindent \textsc{Case 1:} $\alpha = 0$. In this case, $st = m \in M$, and so either $s \in M$ or $t \in M$. Assume first that $s \in M$. Now write $a = (\beta^{-1}s)(\beta t)$ for some $\beta \in K^\times$ such that $\beta t \in R$. Because $a \in \mathcal{A}(R)$, either $\beta^{-1}s \in R^\times$ or $\beta t \in R^\times$. Since $\beta^{-1} s \in M$, it follows that $\beta t \in R^\times \subseteq T^\times$, which implies that $t \in T^\times$. We can similarly obtain that $s \in T^\times$ if we assume that $t \in M$. Thus, either $s$ or $t$ belongs to $T^\times$.
	\smallskip
	
	\noindent \textsc{Case 2:} $\alpha \neq 0$. In this case, neither $s$ nor $t$ belong to $M$. Take $\alpha_1, \alpha_2 \in K^\times$ with $\alpha_1 \alpha_2 = \alpha$ and $m_1, m_2 \in M$ such that $s = \alpha_1(1 + m_1)$ and $t = \alpha_2(1 + m_2)$. Because $a \in \mathcal{A}(R)$, the equality $a = \alpha(1+m_1)(1+m_2)$ guarantees that either $\alpha(1+m_1) \in R^\times \subseteq T^\times$ or $1+m_2 \in R^\times \subseteq T^\times$. Hence either $s$ or $t$ (or both) belongs to $T^\times$.
	\smallskip
	
	In both of the cases, we have obtained that either $s$ or $t$ belongs to $T^\times$, which means that $a \in \ii(T)$. As a consequence, we can conclude that the inclusion $\ii(R) \subseteq T^\times \cup \ii(T)$ holds.
	\smallskip
	
	(1) Observe that because $D$ is a field, then part~(2) of Lemma~\ref{lem:units of the D+M construction} guarantees that $T^\times \cap \ii(R) \subseteq T^\times \cap R = R^\times$, and so $T^\times \cap \ii(R)$ is empty. This, in tandem with the inclusion $\ii(R) \subseteq T^\times \cup \ii(T)$ already proved, implies the inclusion $\mathcal{A}(R) \subseteq \mathcal{A}(T)$.
	\smallskip
	
	(2) Let us argue first that $\mathcal{A}(R) \cap M = \mathcal{A}(T) \cap M$. Fix $m \in M$. If $m \in \mathcal{A}(R)$, then the inclusion $\ii(R) \subseteq T^\times \cup \ii(T)$ and the fact that $M \cap T^\times$ is empty imply that $m \in \mathcal{A}(T)$. Conversely, if $m \in \mathcal{A}(T)$, and $m = r_1 r_2$ for some $r_1, r_2 \in R$, then the fact that $m \in \mathcal{A}(T)$ implies that exactly one of $r_1$ and $r_2$ belongs to~$M$: say that $r_2 \in M$. Then we can write $r_1 = \alpha(1+m_1)$ for some $\alpha \in D^*$ and $m_1 \in M$, and the fact that $r_2 \in T^\times$ implies that $r_1$ belongs to $T^\times$. Thus, it follows from part~(1) of Lemma~\ref{lem:units of the D+M construction} that $r_1 \in R^\times$, which implies that $m \in \mathcal{A}(R)$.  As a result, $\mathcal{A}(R) \cap M = \mathcal{A}(T) \cap M$.
	
	We will argue now that $\mathcal{A}(R) \cap (1+M) = \mathcal{A}(T) \cap (1+M)$. This follows directly from part~(1) of Lemma~\ref{lem:units of the D+M construction} after observing that for each $m \in M$, each divisor of $1+m$ in $R$ (resp., in $T$) is associate in~$R$ (resp., in $T$) to an element of $1+M$.
\end{proof}

Now we consider atomicity and the ACCP under the $D+M$ construction. The following result was first proved by Anderson, Anderson, and Zafrullah \cite[Proposition~1.2]{AAZ90}.

\begin{theorem}  \label{thm:D+M construction}
	Let $T$ be an integral domain, and let $K$ and $M$ be a subfield of $T$ and a nonzero maximal ideal of $T$, respectively, such that $T = K + M$. For a subdomain $D$ of~$K$, set $R = D + M$. Then the following statements hold.
	\begin{enumerate}
		\item $R$ is atomic if and only if $T$ is atomic and $D$ is a field.
		\smallskip
		
		\item $R$ satisfies the ACCP if and only if $T$ satisfies the ACCP and $D$ is a field.
	\end{enumerate}
\end{theorem}

\begin{proof}
	(1) For the direct implication, suppose that $R$ is atomic.
	
	We first argue that $D$ must be a field. Suppose, by way of contradiction, that this is not the case. Lake $d$ be a nonzero nonunit of $D$. Then $d$ cannot be a unit of $R$, and so we can factor each $m \in M$ as $m = d(d^{-1}m)$. Therefore $M$ is disjoint from $\ii(R)$. This implies that no element of $M \setminus \{0\}$ factors into irreducible in $R$ because otherwise the fact that $M$ is a prime ideal would guarantee that one of the atoms in the factorization belongs to $M$. Thus, the elements of the nonempty set $M \setminus \{0\}$ are not atomic in $R$, contradicting that $R$ is atomic. Hence $D$ is a field.
	
	Now we prove that $T$ is also atomic. To do so, fix a nonzero nonunit $x \in T$, and take $k \in K^\times$ such that $xk^{-1} \in R$ (we can take $k$ as $1$ if $x \in M$ and as the projection of $x$ on $K$ otherwise). Since $R$ is atomic, $xk^{-1}$ factors into irreducibles in $R$, and so part~(1) of Lemma~\ref{lem:irreducibles of the D+M construction} ensures that $x$ factors into irreducibles in $T$. Hence $T$ is also atomic. 
	
	For the reverse implication, suppose that $T$ is atomic and $D$ is a field. Fix nonzero nonunit $x \in R$, and let us show that $x$ factors into irreducibles in $R$. Since $T$ is atomic, we can write
	\[
		x = \prod_{i=1}^r m_i \prod_{j=1}^s (\alpha_j + m'_j)
	\]
	for irreducibles $m_1, \dots, m_r \in M$ and $\alpha_1 + m'_1, \dots, \alpha_s + m'_s \in K^\times + M$ of $T$. Now set $\alpha := \prod_{j=1}^s \alpha_j$. Notice that if $r=0$, then $\alpha \in D^* \subseteq R^\times$, and in light of part~(2) of Lemma~\ref{lem:irreducibles of the D+M construction} the element $x$ factors into irreducibles in~$R$ as $x = \alpha \prod_{j=1}^s (1 + \alpha_j^{-1}m'_j)$. If $r > 0$, then $x$ still factors into irreducibles in $R$ as $x = (\alpha m_1) \prod_{i=2}^r m_i \prod_{j=1}^s (1 + \alpha_j^{-1}m'_j)$. Hence we conclude that $R$ is atomic.
	\smallskip
	
	(2) Since every integral domain satisfying the ACCP is atomic, in light of part~(1) we can assume that $D$ is a field to prove both implications. The following two observations will be crucial to argue both implications. Since $D$ is a field, every principal ideal of~$R$ (resp., of $T$) has one of the forms $R m$ or $R (1+m)$ (resp., $T m$ or $T (1+m)$) for some $m \in M$. Furthermore, for all $m_1, m_2 \in M$, one can readily verify the following statements:
	\begin{enumerate}
		\item $R (1+m_1) \subseteq R(1 + m_2)$ if and only if $T (1+m_1) \subseteq T (1 + m_2)$,
		\smallskip
		
		\item $R m_1 \subseteq R m_2$ implies that $T m_1 \subseteq T m_2$, and 
		\smallskip
		
		\item $T m_1 \subseteq T m_2$ implies that $R m_1 \subseteq R k m_2$ for some $k \in K^\times$.
	\end{enumerate}
	
	For the direct implication, suppose that $R$ satisfies the ACCP. Therefore $R$ is atomic, and so it follows from part~(1) that $T$ is atomic and $D$ is a field. With this in mind, let $(T r_n)_{n \ge 1}$ be an ascending chain of principal ideals of $T$. In light of our first observation, we can assume that $r_n \in M \cup (1+M)$ for every $n \in \nn$. In fact, given that no element of~$M$ can divide an element of $1+M$ in $T$, after dropping finitely many terms from the sequence $(T r_n)_{n \ge 1}$ we can further assume that either $r_n \in M$ for every $n \in \nn$ or that $r_n \in 1+M$ for every $n \in \nn$. Suppose first that $r_n \in M$ for every $n \in \nn$. In this case, by virtue of statement~(3) and the inclusion $K^\times \subseteq T^\times$, we can replace the generators of the principal ideals in the sequence $(T r_n)_{n \ge 1}$ by suitable associates in~$T$ in such a way that $(R r_n)_{n \ge 1}$ is an ascending chain of principal ideals in $R$. As $R$ satisfies the ACCP, $(R r_n)_{n \ge 1}$ must stabilize, and so it follows from statement~(2) that $(T r_n)_{n \ge 1}$ also stabilizes. Now suppose that $r_n \in 1+M$ for every $n \in \nn$. Then it follows from statement~(1) that $(R r_n)_{n \ge 1}$ is an ascending chain of principal ideals of $R$, and so it must stabilize because $R$ satisfies the ACCP. Thus, it follows from statement~(1) that $(T r_n)_{n \ge 1}$ also stabilizes. Hence $T$ also satisfies the ACCP.
	
	Conversely, suppose that $T$ satisfies the ACCP and~$D$ is a field. Let $(R r_n)_{n \ge 1}$ be an ascending chain of principal ideals of $R$. The observations we made earlier allow us to assume that either $r_n \in M$ for every $n \in \nn$ or that $r_n \in 1+M$ for every $n \in \nn$. Assume first that $r_n \in M$ for every $n \in \nn$. In this case, statement~(2) ensures that $(T r_n)_{n \ge 1}$ is also an ascending chain of principal ideals of $T$, and so it stabilizes because $T$ satisfies the ACCP. Now it follows from statement~(3) that $(R r_n)_{n \ge 1}$ also stabilizes. Now assume that $r_n \in 1+M$ for every $n \in \nn$. By statement~(1), the sequence of principal ideals $(T r_n)_{n \ge 1}$ is an ascending chain, which must stabilize as $T$ satisfies the ACCP. Now statement~(1) ensures that $(R r_n)_{n \ge 1}$ also stabilizes. Hence we conclude that $R$ also satisfies the ACCP.
\end{proof}

We can use the $D+M$ construction to give an example of an atomic domain (indeed, an integral domain satisfying the ACCP) whose integral closure is not atomic.

\begin{example}
	 Let $F$ be a field, and consider the monoid algebra $R := F[\qq_{\ge 0}]$ of the Puiseux monoid $\qq_{\ge 0}$ over the field $F$. If we consider the maximal ideal
	 \[
	 	\mathfrak{m} := \{f \in R : f(0)=0 \}
	 \] 
	 of $R$, then the localization $V := R_\m$ of $R$ at $\m$ is a $1$-dimensional valuation domain with value group~$\mathbb{Q}$. We now consider the subring formed by the $D+M$ construction $V_1 := F + xV$. Note that $V_1$ is a quasi-local subring of~$V$ and also that the integral closure of $V_1$ is precisely $V$. Because of the fact that every nonunit in $V_1$ has value at least $1$, any ascending chain of principal ideals in $V_1$ must stabilize, and so $V_1$ satisfies the ACCP. However, $V$ is not even atomic as it is not a Noetherian valuation domain. Hence $V_1$ is an atomic (ACCP) domain having a non-atomic integral closure.
\end{example}

\begin{remark}
	In \cite{Co4} an HFD, $R$, is constructed such that the integral closure of~$R$ is not even atomic. The example in this paper is a rather complicated multi-stage construction, but does illustrate that even a domain that is atomic and as strong as the HFD class is still not immune from losing atomicity in the integral closure.
\end{remark}

\bigskip
\section{Localization}

In this section, we discuss how atomicity behaves under localization. Let $R$ be an integral domain, and let $S$ be a multiplicative set of $R$ (i.e., a submonoid of $R^*$). Now let $R_S$ denote the localization of~$R$ at $S$. The notion of localization is one of the central tools in commutative algebra, and much of its utility lies in its good behavior. For instance, the correspondence theorem gives a natural one-to-one correspondence between the prime ideals of $R$ missing $S$ and the prime ideals of $R_S$. Many well-studied properties in commutative ring theory, including being Noetherian and being a UFD, behave well under localization. Unfortunately, the property of atomicity and that of satisfying the ACCP do not interact well with localization in general. We proceed to illustrate this with some examples.

\medskip
\subsection{Some Motivating Counterexamples} 

Let us take a look at two examples of integral domains that satisfy the ACCP but have a localization that is not even atomic.

\begin{example} \label{ex:ACCP monoid algebra with a non-atomic localization}
	Take $p \in \pp$, and let $\ff_p[M]$ be the monoid algebra of the Puiseux monoid $M := \{0\} \cup \qq_{\ge 1}$ over $\ff_p$. Let~$S$ be the multiplicative set $\{x^q : q \in M\}$. From the fact that $0$ is not a limit point of $M^\bullet$, we can readily deduce that~$M$ satisfies the ACCP (indeed, that $M$ is a BFM). This, along with the fact that $\ff_p$ is a field, ensures that $\ff_p[M]$ satisfies the ACCP. On the other hand, it follows from \cite[Proposition~3.1]{GGT21} that $\gp(M) = \qq$, and so $R_S$ is the group algebra $\ff_p[\qq]$, which is not even atomic: actually, $\ff_p[\qq]$ is antimatter because every nonzero element is a $p$-power.
\end{example}

\begin{example}
	Consider the integral domain $R:=\mathbb{Z}+x\overline{\mathbb{Z}}[x]$ where $\overline{\mathbb{Z}}$ is the ring of all algebraic integers. We first show that $R$ satisfies the ACCP. To do so, let $(f_nR)_{n \ge 1}$ be an ascending chain of principal ideals of $R$. This chain gives rise to the decreasing sequence of degrees $(\deg f_n)_{n \ge 1}$, which must stabilize. Thus, after dropping some of the first terms from the sequence  $(f_nR)_{n \ge 1}$, we can assume that all the polynomials in the sequence $(\deg f_n)_{n \ge 1}$ have the same degree. For each $n \in \nn$, let $L_n$ denote the leading terms of $f_n$. For each $n \in \nn$, as $f_n/f_{n+1} \in R$, the fact that $\deg \, f_n = \deg \, f_{n+1}$ implies that $L_n/L_{n+1} \in \zz$. In particular, $(L_n D)_{n \ge 1}$ is an ascending chain of principal ideals in the ring of algebraic integers $D$ obtained by taking the integral closure of $\mathbb{Z}$ in $\mathbb{Q}(L_1)$. Since $D$ is a Dedekind domain, $(L_n D)_{n \ge 1}$ must stabilize. whence $(f_nR)_{n \ge 1}$ also stabilizes. Thus, $R$ satisfies the ACCP.
	
	We proceed to argue that the localization $T:=R[\frac{1}{x}]$ of $R$ at the multiplicative set $\{x^n : n \in \nn_0\}$ is not atomic. Consider a rational prime $p \in T$, and suppose that we can factor $p$ as $p = a_1 \cdots a_n$ for some irreducibles $a_1, \dots, a_n$ of $T$. For each $i \in \ldb 1,n \rdb$, write $a_i := \frac{p_i(x)}{x^{b_i}}$ for some $b_i \in \zz$ and $p_i(x) \in R$ such that $p_i(x)$ is irreducible in~$R$. Clearing denominators in $p = a_1 \cdots a_n$, we obtain the factorization $x^b p = p_1(x) \cdots p_n(x)$ for some $b \in \zz$. Considering this as a factorization in $\overline{\mathbb{Q}}[x]$, we see that $p_j(x)$ must be a monomial of the form $\alpha_jx^{s_j}$ for every $j \in \ldb 1,n \rdb$, and also that we can take $i \in \ldb 1,n \rdb$ such that $\alpha_i$ a nonunit in $\overline{\mathbb{Z}}$. Thus, $\alpha_i$ is an irreducible in $T$. However, it follows from the inclusion $\overline{\mathbb{Z}} \subseteq T$ that $\sqrt{\alpha_i} \in T$, which contradicts the irreducibility of~$\alpha_i$. Hence we conclude that~$R$ satisfies the ACCP, it has a localization that is not even atomic.
\end{example}

There are also integral domains that are not atomic even though their localizations at any prime ideal are PIDs and, therefore, satisfy the ACCP. The following example illustrates this observation.

\begin{example} \label{ex:non-atomic domain with all localizations at prime ideals PIDs}
	Fix $p_0 \in \pp$, and set $R_0 := \mathbb{Z}_{(p)}$. We construct a quadratic extension $K_1$ of $\mathbb{Q}$ such that in $R_1$, the integral closure of $R_0$ in $K_1$, the prime $p$ splits into two principal primes as $p_0 = p_1 q_1$. We then construct $K_2$, a quadratic extension of $K_1$ such that in $R_2$, the integral closure of $R_1$ in $K_2$, the prime~$q_1$ is inert and $p_1$ splits into two principal primes as $p_1 = p_2 q_2$. Inductively, we construct for each $n \in \nn$ a semilocal PID $R_n$ in which $q_1, q_2,\ldots, q_n$ are prime and $p_{n-1} = p_n q_n$. Each step is made possible by \cite[Theorem 42.5]{rG68}. We note that $D :=\bigcup_{n=1}^\infty R_n$, the union of this chain of integral domains, is an almost Dedekind domain in which the maximal ideals are $(q_n)$ for each $n \in \mathbb{N}$ and $(p_1, p_2, p_3, \ldots)$. Observe that $D$ is not atomic because $p_0$ cannot be factored into irreducibles. Note, however, that $D_P$ is PID, and so atomic, for every prime ideal $P$ of~$D$.
\end{example}

Although atomicity and the ACCP do not behave well under localization in general, they do in special localization extensions. We will briefly revise such extensions.

\medskip
\subsection{Inert Extensions} 

Following Cohn~\cite{pC68}, we say that an extension $A \subseteq B$ of commutative rings is \emph{inert} provided that $rs \in A$ for $r,s \in B^\ast$ implies that $ur, u^{-1}s \in A$ for some element $u \in B^\times$. Extensions that are inert are more likely to allow the transfer of algebraic an arithmetic properties; for instance, see the recent paper~\cite{BR22} for connection between inert extensions and transfer Krull homomorphisms. For any integral domain $R$ it is clear that the polynomial extension $R \subseteq R[x]$ is inert. Also, with notation as in the $D+M$ construction, the extensions $D \subseteq R = D + M$ and $R \subseteq T = K + M$ are both inert. As the following two examples illustrate, there are extensions of integral domains are not inert.

\begin{example}
	Let $R$ be an integral domain, and consider the following extension: $R[x^2] \subseteq R[x]$. Observe that even though $x^2 \in R[x^2]$, the fact that $R[x^2]^\times = R^\times$ ensures that $ux \notin R[x^2]$ for any $u \in R^\times$. Thus, $R[x^2] \subseteq R[x]$ is not inert extension. Similarly, we can argue that for any proper additive submonoid $M$ of $\nn_0$ (this monoids are often called numerical monoids) the extension $R[M] \subseteq R[x]$ is not inert.
\end{example}

The following lemma, which will be needed later, is \cite[Lemma~1.1]{AAZ92}. We have decided to omit the proofs of all the lemmas in this subsection not only because they are based on routine, but also because they were explicitly discussed in the recent survey about the bounded and the finite factorization properties~\cite{AG22} by Anderson and the second author.

\begin{lemma} \label{lem:irreducibles in inert extensions}
	If $A \subseteq B$ is an inert extension of integral domains, then $\ii(A) \subseteq B^\times \cup \ii(B)$.
\end{lemma}

In light of Lemma~\ref{lem:irreducibles in inert extensions}, one can see that if $A \subseteq B$ is an inert extension of integral domains with $A^\times = B^\times \cap A$, then $\ii(A) = \ii(B) \cap A$.
\smallskip

A multiplicative subset $S$ of $R$ is called \emph{saturated} provided that $S$ is a divisor-closed submonoid of $R^*$. Observe that if a multiplicative subset of $R$ is generated by primes, then it is saturated. Following~\cite{AAZ92}, we say that a saturated multiplicative subset $S$ of an integral domain $R$ is \emph{splitting} if every $t \in R$ can be factored as $t = rs$ for some $r \in R$ and $s \in S$ with $rR \cap s'R = rs' R$ for every $s' \in S$. The following lemma, which is \cite[Proposition~1.5]{AAZ92}, asserts that extensions by localization are inert when the multiplicative set is splitting.

\begin{lemma} \label{lem:localizations at splitting MS are inert}
	Let $R$ be an integral domain, and let $S$ be a splitting multiplicative set of $R$. Then $R \subseteq R_S$ is an inert extension.
\end{lemma}

The following lemma, first established in  \cite[Proposition~1.6]{AAZ92}, yields a characterization of the multiplicative sets generated by primes that are splitting multiplicative sets.

\begin{lemma} \label{lem:when multiplicative sets generated by primes are SMS}
	Let $R$ be an integral domain, and let $S$ be a multiplicative set of~$R$ generated by primes. Then the following statements are equivalent.
	\begin{enumerate}
		\item[(a)] $S$ is a splitting multiplicative set.
		\smallskip
		
		\item[(b)] $\bigcap_{n \in \nn} p^n R = \bigcap_{n \in \nn} p_n R = \{0\}$ for every prime $p \in S$ and every sequence $(p_n)_{n \ge 1}$ of non-associate primes in $S$.
		\smallskip
		
		\item[(c)] For every nonunit $x \in R^\ast$, there is an $n_x \in \nn$ such that $x \in p_1 \cdots p_n R$ for $p_1, \dots, p_n \in S$ implies that $n \le n_x$.
	\end{enumerate}
\end{lemma}

It was proved in \cite[Proposition~1.9]{AAZ92} that, as it is the case of localizations at splitting multiplicative sets, the extensions by localization at multiplicative sets generated by primes are inert extensions.

\begin{lemma} \label{lem:localization at prime-generated MS are inert}
	Let $R$ be an integral domain, and let $S$ be a multiplicative set of~$R$ generated by primes. Then $R \subseteq R_S$ is an inert extension.
\end{lemma}

\medskip
\subsection{Ascent of Atomicity and ACCP  Under Localization}

Extensions by localization may not be inert. The following example sheds some light upon this observation.

\begin{example}
	Fix a field $F$ and consider the monoid algebra $R := F[M]$, where~$M$ is the Puiseux monoid $\nn_0 \setminus \{1\}$. Observe that the localization of $R$ at the multiplicative subset $S := \{x^m : m \in M\}$ is the group algebra $F[\zz]$, that is, the ring of Laurent polynomials over $F$. Therefore
	\[
		R_S^\times = \{c x^n : c \in F^\times \text{ and } n \in \zz\}.
	\]
	Now consider the localization extension $R \subseteq R_S$, and note that $1 \pm x \in R_S$ and $(1-x)(1+x) = 1 - x^2 \in R$. Thus, the extension $R \subseteq R_S$ is not inert because there is no way to take $c \in F^\times$ and $n \in \zz$ such that $cx^n(1-x)$ and $c^{-1}x^{-n}(1+x)$ belong to $R$ simultaneously.
\end{example}

However, as the following theorem indicates, if an extension $R \subseteq R_S$ of integral domains by localization is inert, then the property of being atomic and that of satisfying the ACCP both ascend from $R$ to $R_S$. This result is part of \cite[Theorem~2.1]{AAZ92}.

\begin{theorem} \label{thm:localization ascent of atomicity/ACCP} 
	Let $R$ be an integral domain, and let $S$ be a multiplicative set of $R$ such that $R \subseteq R_S$ is an inert extension. Then the following statements hold.
	\begin{enumerate}
		\item  If $R$ is atomic, then $R_S$ is atomic.
		\smallskip
		
		\item If $R$ satisfies the ACCP, then $R_S$ satisfies the ACCP.
	\end{enumerate}
\end{theorem}

\begin{proof}
	(1) Suppose first that $R$ is atomic. Take a nonzero nonunit $x$ in $R_S$ and write it as $x = \frac{r}s$, where $r \in R$ and $s \in S$. Since $R$ is atomic, there are $a_1, \dots, a_n \in \ii(R)$ such that $r = a_1 \cdots a_n$. As the extension $R \subseteq R_S$ is inert, in light of Lemma~\ref{lem:irreducibles in inert extensions} we can assume that $a_1, \dots, a_j \in \ii(R_S)$ and $a_{j+1}, \dots, a_n \in R_S^\times$ for some $j \in \ldb 0, n \rdb$, and so $a_1 \cdots a_j \in \mathsf{Z}_{R_S}(x)$. Thus, $R_S$ is an atomic domain.
	\smallskip
	
	(2) Now suppose that $R$ satisfies the ACCP. Take $x,y \in R_S$, and assume that $xR_S \subsetneq yR_S$ after replacing both $x$ and $y$ by suitable associates, we can further assume that $x,y \in R$. Now write  $x = y\big(\frac{r}s \big)$ for some $r \in R$ and $s \in S$. Because $R \subseteq R_S$ is inert extension, one can take $u \in R_S^\times$ such that both $uy$ and $u^{-1}\big(\frac{r}s \big)$ belong to $R$. After setting $y' := uy$,  we see that $xR = (uy)\big(u^{-1} \big(\frac{r}s \big) \big)R \subsetneq uyR$ (here the inclusion is strict because $\frac{r}s \notin R_S^\times$. As a consequence, the principal ideal $xR$ is properly contained in $uyR$, and $uyR_S = yR_S$. Thus, if $(x_n R_S)_{n \ge 1}$ were an ascending chain of principal ideals of $R_S$ that does not stabilize, then after taking, for each $n \in \nn$, an associate $x'_n \in R$ of $x_n$ in $R_S$ such that $x'_n R \subsetneq x'_{n+1} R$ when $x_n R_S \subsetneq x_{n+1} R_S$, we would obtain an ascending chain of principal ideals $(x'_n R)_{n \ge 1}$ that does not stabilize.
\end{proof}

Theorem~\ref{thm:localization ascent of atomicity/ACCP}, in tandem with Lemmas~\ref{lem:localizations at splitting MS are inert} and~\ref{lem:localization at prime-generated MS are inert}, yields the following corollary, which is \cite[Corollary~2.2]{AAZ92}.

\begin{corollary} \label{cor:localization ascent of atomicity/ACCP}
	Let $R$ be an integral domain, and let $S$ be a multiplicative set of~$R$ such that~$S$ is either generated by primes or a splitting multiplicative set. If~$R$ is atomic (resp., satisfies the ACCP), then $R_S$ is atomic (resp., satisfies the ACCP).
\end{corollary}

We conclude this subsection with the following remark.

\begin{remark}
	Theorem~\ref{thm:localization ascent of atomicity/ACCP} holds if we replace being atomic (or satisfying the ACCP) by being a BFD, an FFD, or a UFD (see~\cite[Theorem~2.1]{AAZ92}).
\end{remark}

\medskip
\subsection{A Nagata-Type Theorem}

The next ``Nagata-type" theorem provides a scenario where atomicity and the ACCP are inherited from certain special localizations. Part~(1) of the following theorem is a fragment of \cite[Theorems~3.1]{AAZ92}, while part~(2) of the same theorem is a fragment of \cite[Theorem~3.2]{GP74}.

\begin{theorem} \label{thm:localization Nagata-type for atomicity/ACCP}
	Let $R$ be an integral domain, and let $S$ be a splitting multiplicative set of $R$ generated by primes. Then the following statements hold.
	\begin{enumerate}
		\item  If $R_S$ is atomic, then $R$ is atomic.
		\smallskip
		
		\item  If $R_S$ satisfies the ACCP, then $R$ satisfies the ACCP.
	\end{enumerate}
\end{theorem}

\begin{proof}
	(1) Assume first that $R_S$ is atomic. To show that $R$, take a nonzero nonunit $x \in R$. Since $S$ is a splitting multiplicative set that is generated by primes, we can factor $x$ as $x = rs$ for some $r \in R$ and $s \in S$ in such a way that~$s$ is a product of primes and also that no prime in $S$ divides~$r$ in~$R$. Since $R_S$ is atomic, we can factor $r$ as $r = a_1 \cdots a_n$ for some irreducibles $a_1, \dots, a_n$ of $R_S$. Because no prime in $S$ divides~$r$, we can assume that $a_1, \dots, a_n$ are irreducibles of $R$. Thus, $x$ factors in $R$ as $x = a_1 \cdots a_n s$, and so the fact that $s$ is a product of primes in $R$ ensures that $x$ is atomic. Hence we conclude that $R$ is atomic.
	\smallskip
	
	(2) Assume that $R_S$ satisfies the ACCP, and consider the following set:
	\[
		T := \{r \in R^* : rR \cap sR = rsR \text{ for every } s \in S\};
	\]
	that is, $T$ is the set consisting of all the elements in $R^*$ that are not divisible in $R$ by any of the primes in~$S$. Clearly, $T$ is a multiplicative subset of $R$. Let $(r_n R)_{n \ge 1}$ be an ascending chain of principal ideals of~$R$ none of its terms is the zero ideal. For each $n \in \nn$, we can write $r_n = s_n t_n$ for some $s_n \in S$ and $t_n \in T$. Observe now that  $(s_n R)_{n \ge 1}$ and $(t_n R)_{n \ge 1}$ are also ascending chains of principal ideals of~$R$. Furthermore, note that $(r_n R)_{n \ge 1}$ stabilizes if and only if both $(s_n R)_{n \ge 1}$ and $(t_n R)_{n \ge 1}$ stabilize. Since $(t_n R_S)_{n \ge 1}$ is an ascending chain of principal ideals of $R_S$, and $R_S$ satisfies the ACCP, $(t_n R_S)_{n \ge 1}$ must stabilize. Therefore $(t_n R)_{n \ge 1}$ also stabilizes. Now observe that every principal ideal of $R_T$ is generated by an element of $S$ and, as a result, $R_T$ must be a UFD. Thus, $R_T$ satisfies the ACCP, and so the ascending chain of principal ideals $(s_n R_T)_{n \ge 1}$ must stabilize. This implies that $(s_nR)_{n \ge 1}$ stabilizes. As a consequence, one obtains that $(r_n R)_{n \ge 1}$ must stabilize. Hence we conclude that $R$ satisfies the ACCP.
\end{proof}

\begin{remark}
	Theorem~\ref{thm:localization Nagata-type for atomicity/ACCP} holds if one replaces being atomic (or satisfying the ACCP) by being a BFD, an FFD, or a UFD (see~\cite[Theorem~3.1]{AAZ92}).
\end{remark}

\bigskip
\section{Polynomial and Power Series Extensions}

In this section, we take a look at the ascent of atomicity and the ACCP to polynomial and power series extensions. In addition, we study both atomicity and the ACCP in subrings of rings of polynomials as well as subrings of rings of power series, giving special attention to the subrings of the form $S + xR[x]$ and $S + xR\ldb x \rdb$, where $S \subseteq R$ is an extension of integral domains. We also consider the generalized case obtained by replacing the single extension $R \subseteq S$ by the possibly-infinite tower of integral domains $R_1 \subseteq R_2 \subseteq \cdots$.

\medskip
\subsection{Subrings of $R[x]$ of the Form $S + xR[x]$}

Let $R$ be an integral domain. In this subsection, we consider atomicity and the ACCP property of subrings of $R[x]$ of the form $S + xR[x]$, where $S$ is a subring of $R$. Subring of the form $S + xR[x]$, where $S$ is a subring of $R$, are often a good source of counterexamples.

The property of being atomic does not ascend, in general, from $R$ to the polynomial ring $R[x]$. This was first proved by Roitman~\cite{mR93} answering \cite[Question 1]{AAZ90}. In the same paper, Roitman provided a sufficient condition for atomicity to ascend from $R$ to $R[x]$. A nonzero polynomial $f(x) \in R[x]$ is called \emph{decomposable} provided that whenever $f(x) = g(x)h(x)$ for some $g(x), h(x) \in R[x]$, either $g(x)$ or $h(x)$ is a constant polynomial. The following proposition is~\cite[Proposition~1.1]{mR93}.

\begin{proposition} \label{prop:A+xB[x] and atomicity}
	For an integral domain $R$, the following condition are equivalent.
	\begin{enumerate}
		\item[(a)] $R$ is atomic, and the set of coefficients of every indecomposable polynomial in $R[x]$ has a maximal common divisor in~$R$.
		\smallskip
		
		\item[(b)] $R[x]$ is atomic.
	\end{enumerate}
\end{proposition}

\begin{proof}
	(a) $\Rightarrow$ (b): Assume, towards a contradiction, that $R[x]$ is not atomic. Let~$f$ be a minimum-degree nonunit polynomial in $R[x]^*$ that does not factor into irreducibles. Then $f$ must be indecomposable. In addition, the fact that $R$ is atomic ensures that $\deg f \ge 1$. By assumption, we can write $f = c g$, where $c$ is a maximal common divisor of the set of coefficients of~$f$. As any common divisor of the set of coefficients of $g$ belongs to $R^\times$, the fact that $g$ is indecomposable implies that $g$ is irreducible in~$R[x]$. This, along with the fact that $f$ does not factor into irreducibles in $R[x]$, guarantees that $c$ is a nonzero nonunit of $R$ that does not factor into irreducibles, which contradicts that $R$ is atomic.
	\smallskip
	
	(b) $\Rightarrow$ (a): Since the multiplicative monoid $R^*$ is a divisor-closed submonoid of the atomic multiplicative monoid $R[x]^*$, we conclude that $R$ is atomic.
	
	Now suppose, for the sake of a contradiction, that there exists an indecomposable nonzero polynomial $f = \sum_{i=0}^n c_i x^i \in R[x]$ such that $\mcd_R(c_0, \ldots, c_n)$ is an empty set. The fact that $\mcd_R(c_0, \dots, c_n)$ is empty guarantees that $\deg f \ge 1$. Thus,~$f$ is a nonzero nonunit in $R[x]$, and as $R[x]$ is atomic we can write $f = a_1 \cdots a_\ell$ for some irreducibles $a_1, \dots, a_\ell \in R[x]$. Because $f$ is indecomposable, we can further assume that $\deg a_i = 0$ for every $i \in \ldb 1,\ell-1 \rdb$ and $\deg a_\ell = \deg f$. Now since $a_\ell$ is an irreducible in $R[x]$ and $\deg a_\ell \ge 1$, every common divisor of the set of coefficients of $a_\ell$ belongs to $R^\times$. Hence $a_1 \cdots a_{\ell-1}$ must belong to $\mcd_R(c_0, \dots, c_n)$, which contradicts the fact that $\mcd_R(c_0, \dots, c_n)$ is an empty set.
\end{proof}

Even under the condition on indecomposable polynomials given in part~(a) of Proposition~\ref{prop:A+xB[x] and atomicity}, the fact that $R$ is atomic does not imply that $S + xR[x]$ is atomic when $S$ is a subring of $R$. For instance, the subring $\zz + x\qq[x]$ is not atomic even though the set of coefficients of every indecomposable polynomial in $\qq[x]$ has a maximal common divisor in~$\qq$. Under certain conditions, however, we can guarantee that the subring $S + xR[x]$ of $R[x]$ is atomic, one of them being when~$S$ is a field (this is an immediate consequence of Corollary~\ref{cor:if S is a field, S+xR[x] is ACCP}). Instead of looking directly at the atomicity of the subrings $S + xR[x]$, it is more convenient to study whether the same subrings satisfy the ACCP and obtaining some information on their atomicity as a byproduct. We start by the following proposition (a fragment of \cite[Proposition~1.1]{AeA99}), which has various interesting consequences.

\begin{proposition} \label{prop:A+xB[x] and ACCP}
	Let $R$ be an integral domain, and let $S$ be a subring of $R$. Then the following conditions are equivalent.
	\begin{enumerate}
		\item[(a)] $S + xR[x]$ satisfies the ACCP.
		\smallskip
		
		\item[(b)] Every ascending chain $(r_n S)_{n \ge 1}$, where $r_n \in R$ for every $n \in \nn$, stabilizes.
		\smallskip
		
		\item[(c)] $R^\times \cap S = S^\times$, and every ascending chain $(r_nR)_{n \ge 1}$, where $r_n \in R$ and $r_n/r_{n+1} \in S$ for every $n \in \nn$, stabilizes. 
	\end{enumerate}
\end{proposition}

\begin{proof}
	Set $T := S + xR[x]$.
	\smallskip
	
	(a) $\Rightarrow$ (b): Assume that $T$ satisfies the ACCP. Now let $(r_n)_{n \ge 1}$ be a sequence whose terms belong to~$R$ such that $(r_nS)_{n \ge 1}$ is an ascending chain. Observe that, for each $n \in \nn$, the fact that $r_n \in r_{n+1} S \subseteq r_{n+1} T$ implies that $r_nx \in r_{n+1} x T$. This, along with the fact that $Rx \subseteq T$, ensures that $(r_n x T)_{n \ge 1}$ is an ascending chain of principal ideals in $T$. Because $T$ satisfies the ACCP, $(r_n x T)_{n \ge 1}$ must stabilize. As a result, there exists $N \in \nn$ such that for each $n \ge N$ we can take $u_n \in T^\times = S^\times$ such that the equality $r_N = u_n r_n$ holds, which implies that $r_N S = r_n S$. Hence the chain $(r_n S)_{n \ge 1}$ stabilizes. 
	\smallskip
	
	(b) $\Rightarrow$ (a): Now assume that every ascending chain $(r_n S)_{n \ge 1}$, where $r_n \in R$ for every $n \in \nn$, stabilizes. Let $(f_n T)_{n \ge 1}$ be an ascending chain of principal ideals in $T$. If every term of $(f_n)_{n \ge 1}$ is the zero polynomial, there is nothing to do. Thus, suppose that this is not the case, and further assume that none of the terms of $(f_n)_{n \ge 1}$ is the zero polynomial, which can be done after dropping finitely many terms from the same sequence. For each $n \in \nn$, let $r_n$ be the leading coefficient of $f_n$. From one point on all terms of the sequence $(f_n)_{n \ge 1}$ have the same degree, so after dropping finitely many terms from the same sequence, we can assume that all the polynomials in $(f_n)_{n \ge 1}$ have the same degree. For each $n \in \nn$, now the inclusion $f_n T \subseteq f_{n+1} T$ guarantees the existence of $s_{n+1} \in S$ such that $f_n = s_{n+1} f_{n+1}$ and so $r_n = s_{n+1} r_{n+1}$, which implies that $r_n S \subseteq r_{n+1} S$. Therefore $(r_n S)_{n \ge 1}$ is an ascending chain, and so it must stabilize. Thus, there exists $N \in \nn$ such that $s_n \in S^\times$ for every $n \in N$, which implies that $(f_n T)_{n \ge 1}$ stabilizes at its $N$-th terms. Hence $T$ satisfies the ACCP.
	\smallskip
	
	(a) $\Rightarrow$ (c): Assume that $T$ satisfies the ACCP. In order to establish the equality $R^\times \cap S = S^\times$, it suffices to check that $R^\times \cap S \subseteq S^\times$. Take $s \in R^\times \cap S$ and suppose, by way of contradiction, that $s^{-1}$, which belongs to $R$ does not belong to $S$. Then for any $r \in R$, the linear monomial $rx$ factors in $T$ as $rx = s(rs^{-1}x)$, and so no linear monomial in $T$ is irreducible. This clearly implies that $x$ does not factor in $T$, which implies that $T$ is not atomic. However, this contradicts that $T$ satisfies the ACCP.
	
	For the second assertion, suppose that $(r_nR)_{n \ge 1}$ is an ascending chain, where $r_n \in R$ and $r_n/r_{n+1} \in S$ for every $n \in \nn$. As $r_n/r_{n+1} \in S$ for every $n \in \nn$, the sequence $(r_n S)_{n \ge 1}$ is ascending, and so it must stabilize in light of condition~(b), which holds as we already proved that~(a) implies~(b). Thus, we can pick $N \in \nn$ so that, for every $n \ge N$ there exists $s_n \in S^\times \subseteq R^\times$ such that the equality $r_N = s_n r_n$ holds, which implies that $r_N R = r_n R$. Hence the ascending chain $(r_nR)_{n \ge 1}$ stabilizes.
	\smallskip
	
	(c) $\Rightarrow$ (a): Finally, assume that $R^\times \cap S = S^\times$ and every ascending chain $(r_nR)_{n \ge 1}$, where $r_n \in R$ and $r_n/r_{n+1} \in S$ for every $n \in \nn$, stabilizes. Let $(f_n T)_{n \ge 1}$ be an ascending chain of principal ideals in~$T$ As in our proof that (b) implies (a), we can assume that all the polynomials in the sequence $(f_n)_{n \ge 1}$ are nonzero and have the same degree. Then if for each $n \in \nn$, we let $r_n$ denote the leading coefficient of $f_n$, we see that $r_n/r_{n+1} \in S$. Thus, the sequence $(r_n R)_{n \ge 1}$ is an ascending chain of principal ideals in $R$ such that $r_n/r_{n+1} \in S$ for every $n \in \nn$, and so it must stabilize by assumption. This implies the existence of $N \in \nn$ that $r_n/r_{n+1} \in R^\times \cap S \subseteq S^\times$ for every $n \ge N$. Now the fact that all the polynomials in $(f_n)_{n \ge 1}$ have the same degree guarantees that $f_n/f_{n+1} = r_n/r_{n+1} \in S^\times \subseteq T^\times$ for every $n \ge N$, and so $(f_n T)_{n \ge 1}$ stabilizes. Hence $T$ satisfies the ACCP.
\end{proof}

As an immediate consequence of Proposition~\ref{prop:A+xB[x] and ACCP} we obtain that ACCP ascends from $R$ to $R[x]$.

\begin{corollary} \label{cor:ACCP ascends to Polynomial Rings}
	Let $R$ be an integral domain. Then $R$ satisfies the ACCP if and only if $R[x]$ also satisfies the ACCP.
\end{corollary}

\begin{remark}
	In the presence of nonzero zero-divisors, the ACCP does not ascend to polynomial extensions. A commutative ring $R$ satisfying the ACCP such that $R[x]$ does not satisfy the ACCP was constructed by Heinzer and Lantz in~\cite{HL94}.
\end{remark}

We have seen that rings of the form $S + xR[x]$ may not satisfy the ACCP (indeed, $\zz + x\qq[x]$ is not even atomic). The following corollary of Proposition~\ref{prop:A+xB[x] and ACCP}, which was first pointed out in \cite[Proposition~1.1]{BIK97}, gives a sufficient condition for the subrings $S + xR[x]$ to satisfy the ACCP.

\begin{corollary} \label{cor:if S is a field, S+xR[x] is ACCP}
	Let $R$ be an integral domain, and let $S$ be a subring of $R$. If $S$ is a field, then $S + xR[x]$ satisfies the ACCP.
\end{corollary}

The converse of Corollary~\ref{cor:if S is a field, S+xR[x] is ACCP} does not hold in general as, for instance, $\zz[x]$ satisfies the ACCP even though $\zz$ is not a field. As another consequence of Proposition~\ref{prop:A+xB[x] and ACCP}, we can add that $S + xR[x]$ satisfies the ACCP to the equivalent conditions of Proposition~\ref{prop:A+xB[x] and atomicity}.

\begin{corollary} \label{cor:A+xB[x] under qf(S) subset of R}
	Let $R$ be an integral domain, and let $S$ be a subring of $R$ such that $\emph{qf}(S) \subseteq R$. Then the following conditions are equivalent.
	\begin{enumerate}
		\item[(a)] $S + xR[x]$ satisfies ACCP.
		\smallskip
		
		\item[(b)] $S + xR[x]$ is atomic.
		\smallskip
		
		\item[(c)] $S$ is a field.
	\end{enumerate}
\end{corollary}

\begin{proof}
	(a) $\Rightarrow$ (b): This is clear.
	\smallskip
	
	(b) $\Rightarrow$ (c): Assume that $T := S + x R[x]$ is atomic. Suppose, by way of contradiction, that $S$ is not a field. Let $s$ be a nonzero nonunit in $R$. Since $\text{qf}(S) \subseteq R$, we see that $s^{-1} \notin R \setminus S$. Then for any $r \in R$, the linear monomial $rx$ factors in $T$ as $rx = s(rs^{-1}x)$, and so no linear monomial in $T$ is irreducible, which implies that $x$ does not factor into irreducibles in $T$, contradicting that $T$ is atomic.
	\smallskip
	
	(c) $\Rightarrow$ (a): This follows from Corollary~\ref{cor:if S is a field, S+xR[x] is ACCP}.
\end{proof}

With the notation as in Corollary~\ref{cor:A+xB[x] under qf(S) subset of R}, we have already seen that, in general, condition~(a) (and so condition~(b)) does not imply condition~(c). It was proved in~\cite[Proposition~3.6]{GL22} that if $F$ is an infinite field and $R$ is the Grams' domain over $F$, then $R[x]$ is an atomic domain that does not satisfy the ACCP. Therefore the conditions (a) and (b) in Corollary~\ref{cor:A+xB[x] under qf(S) subset of R} are not equivalent in general.
\smallskip

If $R = \text{qf}(S)$, then it turns out that $S + xR[x]$ satisfies the ACCP if and only if $S$ satisfies the ACCP. This is a special case of the following result, which is \cite[Theorem~7.5]{AAZ91}.

\begin{proposition} \label{prop:ACCP of rings between R[X] and K[X]}
	Let $S$ be an integral domain with quotient field $K$, and let $T$ be an integral domain such that $S[x] \subseteq T \subseteq S + xK[x]$. In addition, assume that for every $n \in \nn_0$, there is an $s_n \in S^*$ such that $s_n f \in S[x]$ for every $f \in T$ with $\deg f \le n$. Then $T$ satisfies the ACCP if and only if $S$ satisfies the ACCP.
\end{proposition}

\begin{proof}
	Since $T \subseteq S + xK[x]$, every constant polynomial in $T$ belong to $S$. Therefore it follows that $S^*$ is a divisor-closed submonoid of $T^*$, and so Corollary~\ref{cor:divisor-closed submonoids of an ACCP monoid} guarantees that $S$ satisfies the ACCP provided that $T$ satisfies the ACCP.
	\smallskip
	
	For the reverse implication, assume that $S$ satisfies the ACCP. Now let $(f_n T)_{n \ge 1}$ be an ascending chain of principal ideals of $T$ and assume, without loss of generality, that not every term in this chain is the zero ideal. After dropping finitely many terms from $(f_n T)_{n \ge 1}$, we can assume that all the polynomials in the sequence $(f_n)_{n \ge 1}$ have the same degree, namely, $d$. By hypothesis, we can take a nonzero $s_d \in S$ such that $s_d f_n \in S[x]$ for every $n \in \nn$. Now the fact that $(f_n/f_{n+1}) \in T \cap K \subseteq S$ for every $n \in \nn$ guarantees that $(s_d f_n S[x])_{n \ge 1}$ is an ascending chain of principal ideals in $S[x]$. As $S[x]$ satisfies the ACCP by Corollary~\ref{cor:ACCP ascends to Polynomial Rings}, there exists $N \in \nn$ such that $f_N/f_n \in S[x]^\times \subseteq T^\times$ for every $n \ge N$. Hence the chain $(f_n T)_{n \ge 1}$ stabilizes in $T$, and so we conclude that $T$ satisfies the ACCP.
\end{proof}

\medskip
\subsection{Rings of Integer-Valued Polynomials}

We proceed to discuss both atomicity and the ACCP in rings of integer-valued polynomials. Let $R$ be an integral domain with quotient field $K$, and let $S$ be a subset of $R$. The \emph{ring of integer-valued polynomials} $\text{Int}(S,R)$ of~$R$ on~$S$ is the subring of $K[x]$ consisting of all polynomials mapping the elements of $S$ into $R$:
\[
	\text{Int}(S,R) := \{f \in K[x] : f(S) \subseteq R\}.
\]
It is customary to write $\text{Int}(R,R)$ simply as $\text{Int}(R)$ and call it the \emph{ring of integer-valued polynomials of} $R$. A survey on rings of integer-valued polynomials has been given by P. J. Cahen and J. L. Chabert in \cite{CC16}, and the reader can find considerably more information about $\text{Int}(S,R)$ in the book \cite{CC97}, written by the same authors. Finally, the arithmetic and atomicity of rings of integer-valued polynomials have been recently studied by several authors, including Fadinger, Frisch, Li, Nakato, Rissner, Windisch, and the second author: see~\cite{FW24}, which is one of the most recent studies, as well as the references therein.

Observe that $R[x] \subseteq \text{Int}(R) \subseteq \text{Int}(S,R)$. In particular, $\text{Int}(S,R)$ is an intermediate ring of the extension $R[x] \subseteq K[x]$. The inclusion $R[x] \subseteq \text{Int}(R)$ is often strict: for instance,
\[
	\binom{x}{n} := \frac{x(x-1) \cdots (x-n+1)}{n!} \in \text{Int}(\zz) \setminus \zz[x]
\]
for every $n \in \nn$ with $n \ge 2$. However, as the following proposition indicates, there are conditions on $R$ which guarantee that the equality $\text{Int}(R) = R[x]$ holds. We record them here for future reference.

\begin{proposition} \label{prop:int(R)=R[x] when R contains an infinite field}
	Let $R$ be an integral domain. If at least one of the following conditions holds, then $\emph{Int}(R) = R[x]$.
	\begin{itemize}
		\item \cite[Corollary~2]{CC71}  $R$ contains an infinite field.
		\smallskip
		
		\item \cite[page 10]{CC97} There exists a family $\Omega$ consisting of prime ideals of $R$ such that $R = \bigcap_{\mathfrak{p} \in \Omega} R_\mathfrak{p}$ and $R_\mathfrak{p}$ is infinite for each $\mathfrak{p} \in \Omega$.
	\end{itemize}
\end{proposition}

Observe that $\text{Int}(\{0\},\zz) = \zz + x\qq[x]$, which is not even atomic. We proceed to present a necessary condition for the atomicity of $\text{Int}(S,R)$ that generalizes this observation (it was first proved in \cite[Proposition~1.1]{ACCS95}).

\begin{proposition} \label{prop:sufficient condition for atomicity}
	Let $R$ be an integral domain that is not a field, and let $S$ be a nonempty subset of~$R$. If $\emph{Int}(S,R)$ is atomic, then $|S| = \infty$.
\end{proposition}

\begin{proof}
	Assume that $\text{Int}(S,R)$ is atomic. Now suppose, towards a contradiction, that $S$ is a finite set of size $n$ and write $S = \{s_1,\dots, s_n\}$. Since $R$ is not a field, we can fix a nonzero nonunit $d \in R$. Now for any nonempty proper subset $T$ of $\ldb 1,n \rdb$, set
	\[
		m_T := \prod_{t \in \ldb 1,n \rdb \setminus T} \! (x-s_t) \quad \text{ and } \quad f := \frac{\prod_{i=1}^n (x-s_i)}{d \prod_{T \subsetneq \ldb 1,n \rdb} \prod_{j \notin \ldb 1,n \rdb \setminus T} m_T(s_j)}.
	\]
	As $f(s_i) = 0$ for every $i \in \ldb 1,n \rdb$, it follows that $f \in \text{Int}(S,R)$ (indeed, the same argument shows that any polynomial in $\text{qf}(R)^*f$ belongs to $\text{Int}(S,R)$). In addition, the fact that $\deg f = n \ge 1$ ensures that $f$ is a nonunit element in $\text{Int}(S,R)^*$. Let us argue that $f$ cannot factor into irreducibles in $\text{Int}(S,R)$, which will give the desired contradiction. First, observe that for every nonzero $r \in R$, the fact that both $d$ and $f/rd$ are nonzero nonunits of $\text{Int}(S,R)$ implies that $f/r$ is not irreducible in $\text{Int}(S,R)$. This guarantees that $f$ is not irreducible in $\text{Int}(S,R)$ and, as $\text{Int}(S,R)$ is atomic, we can write $f = g h$ for some nonconstant polynomials $g,h \in \text{Int}(S,R)$. Because $K[x]$ is a UFD, there are nonempty disjoint subsets $U$ and $V$ of $\ldb 1,n \rdb$ with $U \sqcup V = \ldb 1,k \rdb$ and elements $r_1, r_2, d_1, d_2 \in R^*$ such that $g = \frac{r_1}{d_1} m_U$ and $h = \frac{r_2}{d_2} m_V$. From the definition of $f$ and the equality $f = gh$, we infer that $d_1 d_2 = ds r_1 r_2$, where 
	\[
		s := \prod_{T \subsetneq \ldb 1,n \rdb} \prod_{j \notin \ldb 1,n \rdb \setminus T} m_T(s_j).
	\]
	Now fix indices $j \in \ldb 1,n \rdb \setminus U$ and $k \in \ldb 1,n \rdb \setminus V$. As $\frac{r_1}{d_1}m_U(s_j) = g(s_j) \in R$ and $ \frac{r_2}{d_2}m_V(s_k) = h(s_k) \in R$, we see that $m_U(s_j) m_V(s_k)r_1 r_2 = d_1 d_2 r_3$ for some $r_3 \in R$. On the other hand, observe that $m_U(s_j) m_V(s_k) \mid_R s$, and take $r_4 \in R$ such that $s = m_U(s_j) m_V(s_k) r_4$. After substituting the last two equalities in $d_1 d_2 = ds r_1 r_2$, one obtains that
	\[
		d_1 d_2 = d sr_1 r_2 = d m_U(s_j) m_V(s_k) r_1 r_2 r_4 = d d_1 d_2 r_3 r_4,
	\]
	and so $d (r_3 r_4) = 1$, which implies that $d \in R^\times$. However, this contradicts the $d$ was initially taken to be a nonunit of $R$. Hence $|S| = \infty$, as desired.
\end{proof}

As a consequence of Proposition~\ref{prop:sufficient condition for atomicity} when $|S| < \infty$, we obtain that $\text{Int}(S,R)$ is atomic if and only if it satisfies the ACCP, which happens precisely when $R$ is a field. 

\begin{corollary} \label{cor:factorization properties of Int(R,S) when S is finite}
	Let $R$ be an integral domain, and let $S$ be a finite subset of $R$. Then the following conditions are equivalent.
	\begin{enumerate}
		\item[(a)] $\emph{Int}(S,R)$ satisfies the ACCP.
		\smallskip
		
		\item[(b)] $\emph{Int}(S,R)$ is atomic.
		\smallskip
		
		\item[(c)] $R$ is a field.
	\end{enumerate}
\end{corollary}

\begin{proof}
	(a) $\Rightarrow$ (b): This is clear.
	\smallskip
	
	(b) $\Rightarrow$ (c): This follows immediately from Proposition~\ref{prop:sufficient condition for atomicity}.
	\smallskip
	
	(c) $\Rightarrow$ (a): It is clear that if $R$ is a field, then $\text{Int}(S,R) = R[x]$ and so it satisfies the ACCP.
\end{proof}

Corollary~\ref{cor:factorization properties of Int(R,S) when S is finite} yields a characterization of when a ring of integer-valued polynomial $\text{Int}(S,R)$ is atomic (or satisfies the ACCP) provided that $|S|$ is finite. As the following proposition indicates, rings of integer-valued polynomials satisfying the ACCP can also be characterized when $|S|$ is not finite (see \cite[Corollary~7.6]{AAZ91}, \cite[Theorem~1.3]{CC95}, and  \cite[Theorem~1.2]{ACCS95}).

\begin{proposition} \label{prop:ACCP characterization for IVPs}
	Let $R$ be an integral domain, and let $S$ be an infinite subset of $R$. Then $\emph{Int}(S,R)$ satisfies the ACCP if and only if $R$ satisfies the ACCP.
\end{proposition}

\begin{proof}
	Let $K$ be the field of fractions of $R$, and set $T := \text{Int}(S,R)$. 
	\smallskip
	
	The direct implication follows immediately from Corollary~\ref{cor:divisor-closed submonoids of an ACCP monoid} because $R^*$ is a divisor-closed submonoid of $T^*$. 
	\smallskip

	For the reverse implication, suppose that $\text{Int}(S, R)$ does not satisfy the ACCP. Then take an ascending chain of principal ideals $(f_n T)_{n \ge 1}$ of $T$ that does not stabilize. We can further assume that $f_n T$ is strictly contained in $f_{n+1} T$ for every $n \in \nn$. As $\deg f_n \ge \deg f_{n+1}$ for every $n \in \nn$, the fact that $\nn_0$ is well ordered ensures that the sequence $(\deg f_n)_{n \ge 1}$ eventually becomes constant. Hence, after dropping finitely many terms from the sequence $(f_n T)_{n \ge 1}$, we can assume that $\deg f_n = \deg f_1$ for every $n \in \nn$. Thus, for each $n \in \nn$, we can set $r_n := f_n/f_{n+1}$ and we see that $r_n \in R$ but $r_n \notin R^\times$ because $f_n T \subsetneq f_{n+1} T$. Now the fact that $|S| = \infty$ allows us to take $s \in S$ such that $f_1(s) \neq 0$, which implies that $f_n(s) \neq 0$ for every $n \in \nn$. As a consequence, $(f_n(s))_{n \ge 1}$ is an ascending chain of principal ideals of $R$ that does not stabilize, which implies that $R$ does not satisfy the ACCP.
\end{proof}

\begin{corollary}
	Let $R$ be an integral domain. Then $\emph{Int}(R)$ satisfies the ACCP if and only if $R$ satisfies the ACCP.
\end{corollary}

As the following example illustrates, the statement in Proposition~\ref{prop:ACCP characterization for IVPs} does not hold if we replace the ACCP by atomicity.

\begin{example} \label{ex:the RIVP of an atomic domain may not be atomic}
	It follows from \cite[Example~5.1]{mR93} that every field can be embedded into an atomic domain $R$ satisfying that $R[x]$ is not atomic. Thus, let $R$ be an atomic domain containing an infinite field such that $R[x]$ is not atomic. Since $R$ contains an infinite field, it follows from Proposition~\ref{prop:int(R)=R[x] when R contains an infinite field} that $\text{Int}(R) = R[x]$, and so $\text{Int}(R)$ is not atomic even though $R$ is atomic.
\end{example}

\medskip
\subsection{Rings of Power Series}

In this last subsection we consider rings of power series. The property of being atomic does not transfer between an integral domain $R$ to its ring of power series $R\ldb x \rdb$ in any direction: counterexamples illustrating this observation were constructed by Roitman in~\cite{mR00}. We have decided not to include such constructions here because they are rather technical. Finally, it is worth emphasizing that there have been significant recent progress on factorization/Krull properties in rings of power series (for instance, see~\cite{AJ22} and \cite{PO23}).

Let $S \subseteq R$ be an extension of integral domains. We proceed to consider atomicity and the ACCP on subrings of the form $S + x R\ldb x \rdb$. We start by proving that $S + x R\ldb x \rdb$ satisfies the ACCP if and only if $S + x R[x]$ satisfies the ACCP (this result is part of \cite[Corollary~1.4]{AeA99}). This will enable us to use most of the results we have already established for $S + x R[x]$ in terms of the ACCP condition to draw similar conclusions for $S + x R\ldb x \rdb$.

\begin{proposition} \label{prop:S+xR[[x]] and the ACCP}
	Let $R$ be an integral domain, and let $S$ be a subring of $R$. Then $S + xR \ldb x \rdb$ satisfies the ACCP if and only if $S + x R[x]$ satisfies the ACCP.
\end{proposition}

\begin{proof}
	Let $T := S + x R\ldb x \rdb$, and observe that if $f \in T^\times$, then $f(0) \in S^\times$. By virtue of Proposition~\ref{prop:A+xB[x] and ACCP} it suffices to show that $S + xR\ldb x \rdb$ satisfies the ACCP if and only if every ascending chain $(r_n S)_{n \ge 1}$, where $r_n \in R$ for every $n \in \nn$, stabilizes.
	\smallskip
	
	For the direct implication, suppose that $T$ satisfies the ACCP, and let $(r_n)_{n \ge 1}$ be a sequence with terms in $R$ such that $(r_n S)_{n \ge 1}$ is an ascending chain. After assuming that not all terms of $(r_n)_{n \ge 1}$ are zero and dropping finitely many of them, we can further assume that $r_n \neq 0$ for any $n \in \nn$. Observe that $(r_n x T)_{n \ge 1}$ is an ascending chain of principal ideals of $T$ because $r_n/r_{n+1} \in S \subseteq T$. As $T$ satisfies the ACCP, the chain $(r_nxT)_{n \ge 1}$ stabilizes, and so we can take $N \in \nn$ such that for each $n \ge N$ the equality $r_N x = r_n x f_n$ holds for some $f_n \in T^\times$. Now, for each $n \ge N$, the equality $r_N = r_n f_n$ implies that $r_N S = r_n S$ because $f_n(0) \in S^\times$. Hence the chain $(r_n S)_{n \ge 1}$ stabilizes.
	\smallskip
	
	Conversely, suppose that every ascending chain $(r_n S)_{n \ge 1}$, where $r_n \in R$ for every $n \in \nn$, stabilizes. Let $(f_nT)_{n \ge 1}$ be an ascending chain of principal ideals of $T$ and assume, without loss of generality, that $f_n \neq 0$ for any $n \in \nn$. Observe that $(\text{ord} \, f_n)_{n \ge 1}$ must become stationary, and so after dropping finitely many terms from $(f_nT)_{n \ge 1}$, we can assume the existence of $m \in \nn$ that $\text{ord} \, f_n = m$ for every $n \in \nn$. For each $n \in \nn$, let $r_n \in R$ be the nonzero coefficient of $x^m$ in the power series $f_n$. Since all the terms of $(f_n)_{n \ge 1}$ have the same order, $r_n/r_{n+1} \in S^*$ for every $n \in \nn$, which implies that $(r_nS)_{n \ge 1}$ is a chain. By assumption, this chain must stabilize, and so for all $n \in \nn$ large enough, $r_n/r_{n+1} \in S^\times$, which implies $f_n T = f_{n+1}T$. Hence $(f_nT)_{n \ge 1}$ stabilizes, and we can conclude that $T$ satisfies the ACCP.
\end{proof}

After putting together Corollary~\ref{cor:if S is a field, S+xR[x] is ACCP} and Proposition~\ref{prop:S+xR[[x]] and the ACCP}, we immediately obtain the following corollary.

\begin{corollary} \label{cor:if S is a field, S+xR[[x]] is ACCP}
	Let $R$ be an integral domain, and let $S$ be a subring of $R$. If $S$ is a field, then $S + xR\ldb x \rdb$ satisfies the ACCP.
\end{corollary}

As the following corollary indicates, the statement we obtain from that of Corollary~\ref{cor:A+xB[x] under qf(S) subset of R} after replacing rings of the form $S + R[x]$ by rings of the form $S + xR\ldb x \rdb$ still holds.

\begin{corollary} \label{cor:A+xB[[x]] under qf(S) subset of R}
	Let $R$ be an integral domain, and let $S$ be a subring of $R$ such that $\emph{qf}(S) \subseteq R$. Then the following conditions are equivalent.
	\begin{enumerate}
		\item[(a)] $S + xR\ldb x \rdb$ satisfies ACCP.
		\smallskip
		
		\item[(b)] $S + xR\ldb x \rdb$ is atomic.
		\smallskip
		
		\item[(c)] $S$ is a field.
	\end{enumerate}
	In particular, $S + x R \ldb x \rdb$ is atomic if and only if $S + x R[x]$ is atomic.
\end{corollary}

\begin{proof}
	(a) $\Rightarrow$ (b): This is clear.
	\smallskip
	
	(b) $\Rightarrow$ (c): This follows \emph{mutatis mutandis} as the given proof of the corresponding statement of Corollary~\ref{cor:A+xB[x] under qf(S) subset of R}.
	\smallskip
	
	(c) $\Rightarrow$ (a): This follows from Corollary~\ref{cor:if S is a field, S+xR[[x]] is ACCP}.
	\smallskip
	
	The last statement of the corollary follows immediately from the equivalence already established and the equivalences of Corollary~\ref{cor:A+xB[x] under qf(S) subset of R}.
\end{proof}

\bigskip
\section{Monoid Algebras}

In this section, we are interested in understanding atomicity as well as the behavior of ascending chains of principal ideals in the class of monoid algebras. Recent progress on the arithmetic of (weakly) Krull monoid algebras can be found in~\cite{FW22,FW22a}.

As we have seen in Grams' construction, previously discussed in Section~\ref{sec:atomic domains}, it is convenient that the monoid of exponents of a given monoid algebra has an order compatible with its operation as it allows us to talk about degree, order, and support, which after all are quite valuable tools in the classical study of polynomial rings. Also, we are interested in monoid/semigroup algebras that are integral domains, and they can be characterized as follows.

\begin{proposition} \label{prop:MA that are integral domains}
	Let $S$ be a commutative semigroup with an identity, and let $R$ be a commutative ring with identity. The semigroup algebra $R[S]$ is an integral domain if and only if $R$ is an integral domain and $S$ is cancellative and torsion-free.
\end{proposition}

\begin{proof}
	For the direct implication, assume that $R[S]$ is an integral domain. Then~$R$ is also an integral domain because it is a subring of $R[S]$. To check that the semigroup $S$ is cancellative, take $b,c,d \in S$ such that $b+d = c+d$. Since $x^d$ is different from $0$ and $x^d(x^b - x^c) = x^{b+d} - x^{c+d} = 0$, the fact that $R[S]$ is an integral domain guarantees that $x^b - x^c = 0$, and so $b=c$. To argue that the monoid $S$ is torsion-free, take $b,c \in S$ and $n \in \nn$ such that $nb = nc$. Assume that we have taken $n$ to be the minimum of the set $\{m \in \nn : mb = mc \}$. Observe that
	\[
		\big(x^b - x^c \big) \sum_{i=0}^{n-1} x^{((n-1) - i)b + ic} = x^{nb} - x^{nc} = 0.
	\]
	As $S$ is cancellative, it follows from the minimality of $n$ that no two of the exponents in the polynomial expression $f := \sum_{i=0}^{n-1} x^{((n-1) - i)b + ic} $ are equal. Therefore $f$ must be different from $0$, and so the fact that $R[S]$ is an integral domain guarantees that $x^b - x^c = 0$, which means that $b=c$. Thus, $S$ is torsion-free.
	\smallskip
	
	Conversely, assume that $S$ is cancellative and torsion-free and $R$ is an integral domain. As $S$ is cancellative and torsion-free, we can assume that $S$ is a totally ordered monoid and so that $\deg \, f$ and $\text{ord} \, f$ are well defined for each $f \in R[S]$. Now observe that if $f$ and $g$ are two nonzero elements of $R[S]$, then $\deg \, f + \deg \, g$ belongs to the support of $fg$ and, therefore, $fg$ cannot be zero.
\end{proof}

Since monoids in this survey are assumed to be cancellative, we can rephrase Proposition~\ref{prop:MA that are integral domains} as follows.

\begin{corollary} \label{cor:MA that are integral domains}
	Let $M$ be a monoid, and let $R$ be a commutative ring with identity. Then the monoid algebra $R[M]$ is an integral domain if and only if $R$ is an integral domain and $M$ is torsion-free.
\end{corollary}

In light of Corollary~\ref{cor:MA that are integral domains}, we tacitly assume that every monoid we mention or deal with throughout this section is torsion-free. 

 Our next goal is to obtain necessary conditions on $M$ and $R$ for $R[M]$ to be atomic or to satisfy the ACCP. First, we need to gather information about the units and irreducible monomials of $R[M]$.

\begin{lemma} \label{lem:units and irreducibles in MAs}
	For an integral domain $R$ and a torsion-free monoid $M$, the following statements hold.
	\begin{enumerate}
		
		\item $\{rx^b : r \in R^* \text{ and } b \in M\}$ is a divisor-closed submonoid of $R[M]^*$.
		\smallskip
		
		\item $R[M]^\times = \{r x^u : r \in R^\times \text{ and } u \in \uu(M) \}$.
		\smallskip
		
		\item $r \in \mathcal{A}(R[M])$ if and only if $r \in \mathcal{A}(R)$.
		\smallskip
		
		\item $x^b \in \mathcal{A}(R[M])$ if and only if $b \in \mathcal{A}(M)$.
	\end{enumerate}
\end{lemma}

\begin{proof}
	As $M$ is torsion-free, we can assume that $M$ is a totally ordered monoid and so that $\deg \, f$ and $\text{ord} \, f$ are well defined for each nonzero $f \in R[M]$.
	\smallskip

	(1) Fix some nonzero $f,g \in R[M]$ such that $f g$ is a monomial in $R[M]$. Since~$R$ is an integral domain, the equality $\deg \, fg = \text{ord} \, fg$, along with the equalities in~\eqref{eq:degree and order}, ensures that $\deg \, f = \text{ord} \, f$ and $\deg \, g = \text{ord} \, g$, and so that both $f$ and~$g$ are monomials in $R[M]$. Thus, we conclude that every divisor of a nonzero monomial in $R[M]$ must be a nonzero monomial.
	\smallskip
	
	(2) It is clear that for any $r \in R^\times$ and $u \in \uu(M)$, the element $rx^u$ is a unit of $R[M]$. Conversely, suppose that $f$ is a unit of $R[M]$, and take $g \in R[M]$ such that $fg = 1$. It follows from part~(1) that $f = r x^b$ and $g = s x^c$ for some $r,s \in R$ and $b,c \in M$, and so the equalities $rs=1$ and $b+c = 0$ imply that $r \in R^\times$ and $b \in \uu(M)$.
	\smallskip
	
	(3) Fix a nonzero $r \in R$. By part~(2), $r$ is a unit in $R$ if and only if $r$ is a unit in $R[M]$, so we can assume that $r \notin R^\times$. Suppose first that $r$ is irreducible in $R[M]$ and  write $r = st$ for some $s,t \in R$. As $r$ is irreducible in $R[M]$ either $s$ or~$t$ is a unit in $R[M]$, and so it follows from part~(2) that either $s \in R^\times$ or $t \in R^\times$. Thus, $r \in \mathcal{A}(R)$. Conversely, suppose that $r \in \mathcal{A}(R)$ and write $r = f_1 f_2$ for some $f_1, f_2 \in R[M]$. In light of part~(1), $f_1 = r_1 x^{b_1}$ and $f_2 = r_2 x^{b_2}$ for some $r_1, r_2 \in R$ and $b_1, b_2 \in M$. Since $b_1 + b_2 = 0$, both $b_1, b_2 \in \uu(M)$ and $r = r_1 r_2$. As $r \in \mathcal{A}(R)$, either $r_1 \in R^\times$ or $r_2 \in R^\times$, which implies that either $f_1 \in R^\times$ or $f_2 \in R^\times$. Hence $r$ is irreducible in $R[M]$.
	\smallskip
	
	(4) Fix $b \in M$. By part~(2), $b \in \uu(M)$ if and only if $x^b \in R[M]^\times$, so we can assume that $b \notin \uu(M)$. It is clear that if $b \notin \mathcal{A}(M)$ and we write $b = c_1 + c_2$ for some $c_1, c_2 \in M \setminus \uu(M)$, then $x^b = x^{c_1} x^{c_2}$, and so $x^b$ is not irreducible in light of part~(2). Conversely, if $x^b$ is not irreducible, then it can be written as $x^b = f_1 f_2$ for two nonunit $f_1, f_2 \in R[M]$, and so it follows from part~(1) and~(2) that $f_1 = r_1x^{d_1}$ and $f_2 = r_2 x^{d_2}$ for some $r_1, r_2 \in R^\times$ and $d_1, d_2 \in M \setminus \uu(M)$, whence the equality $b = d_1 + d_2$ ensures that $b \notin \mathcal{A}(M)$. Thus, $b \notin \mathcal{A}(M)$ if and only if $x^b$ is not irreducible in $R[M]$.
\end{proof}

We proceed to establish necessary conditions on an integral domain $R$ and a monoid $M$ for the monoid algebra $R[M]$ to be atomic and to satisfy the ACCP.

\begin{proposition} \label{prop:MA necessary conditions}
	Let $R$ be an integral domain, and let $M$ be a monoid. Then the following statements hold.
	\begin{enumerate}
		\item If the monoid algebra $R[M]$ is atomic, then both $R$ and $M$ are atomic.
		\smallskip
		
		\item If the monoid algebra $R[M]$ satisfies the ACCP, then both $R$ and $M$ satisfy the ACCP.
	\end{enumerate}
\end{proposition}

\begin{proof}
	(1) Take a nonzero nonunit $r \in R$. Then it follows from part~(2) of Lemma~\ref{lem:units and irreducibles in MAs} that $r$ is a nonzero nonunit of $R[M]$, and so $r$ factors into irreducibles in $R[M]$. By virtue of part~(1) of Lemma~\ref{lem:units and irreducibles in MAs}, this factorization must have the form $r = \prod_{i=1}^\ell r_i x^{b_i}$ for some $r_1, \dots, r_\ell \in R$ and $b_1, \dots, b_\ell \in M$. Since $b_1 + \dots + b_\ell = 0$, we see that for each $i \in \ldb 1, \ell \rdb$ the monomial $x^{b_i}$ is a unit of $R[M]$ by part~(2) of Lemma~\ref{lem:units and irreducibles in MAs}. Thus, the equality $r = r_1 \cdots r_\ell$, in tandem with part~(3) of Lemma~\ref{lem:units and irreducibles in MAs}, guarantees that $r$ is atomic in $R$. Hence $R$ is atomic.
	
	To argue that $M$ is atomic, take a non-invertible element $b \in M$. It follows from part~(3) of Lemma~\ref{lem:units and irreducibles in MAs} that $x^b \notin R[M]^\times$. Now the fact that $R[M]$ is atomic, in tandem with part~(2) of Lemma~\ref{lem:units and irreducibles in MAs}, allows us to write $x^b = x^{a_1} \cdots x^{a_k}$ for some $a_1, \dots, a_k \in M$ such that the monomials $x^{a_1}, \dots, x^{a_k}$ are irreducible in $R[M]$. Then it follows from part~(4) of Lemma~\ref{lem:units and irreducibles in MAs} that $a_1, \dots, a_k \in \mathcal{A}(M)$, and so the equality $b = a_1 \cdots a_k$ implies that $b$ is atomic. Thus, $M$ is atomic.
	\smallskip
	
	(2) To argue that $R$ satisfies the ACCP, let $(r_n R)_{n \ge 1}$ be an ascending chain of principal ideals of~$R$. Then $(r_n R[M])_{n \ge 1}$ is an ascending chain of principal ideals in $R[M]$ and, as $R[M]$ satisfies the ACCP, there exists $N \in \nn$ such that $r_n R[M] = r_N R[M]$ for every $n \ge N$. Then, for each $n \ge N$, the element $r_N/r_n$ belongs to $R[M]^\times$ and so it is a unit monomial by part~(2) of Lemma~\ref{lem:units and irreducibles in MAs}, whence the fact that both $r_N$ and $r_n$ have degree zero in $R[M]$ guarantee that $r_N$ and~$r_n$ are associates in $R$. Hence $r_n R = r_N R$ for every $n \ge N$, and so $(r_n R)_{n \ge 1}$ stabilizes. Thus, $R$ satisfies the ACCP.
	
	To prove that $M$ satisfies the ACCP, let $(b_n + M)_{n \ge 1}$ be an ascending chain of principal ideals of~$M$. This implies that $(x^{b_n}R[M])_{n \ge 1}$ is an ascending chain of principal ideals of $R[M]$. As $R[M]$ satisfies the ACCP, there exists $N \in \nn$ such that $x^{b_n} R[M] = x^{b_N}R[M]$ for every $n \in \nn$ with $n \ge N$. Now for each $n \ge N$, it follows from part~(3) of Lemma~\ref{lem:units and irreducibles in MAs} that the elements $b_N$ and $b_n$ differ by an invertible element in $M$, which means that $b_n + M = b_N + M$. Therefore the chain of ideals $(b_n + M)_{n \ge 1}$ stabilizes in $M$. Hence $M$ satisfies the ACCP.
\end{proof}

The ascent of atomicity to monoid algebras, which is the converse of Proposition~\ref{prop:MA necessary conditions}, was brought to attention by Gilmer in~\cite[page 189]{rG84}, and it can be stated as the following question: for any pair $(M,R)$, where $M$ is a torsion-free cancellative monoid and $R$ is an integral domain, does the fact that both $M$ and $R$ are atomic imply that the monoid algebra $R[M]$ is atomic. As mentioned earlier in this survey, the ascent of atomicity to polynomial extensions, which is a specialization of the same ascent problem, was emphasized by Anderson, Anderson, and Zafrullah in their landmark paper~\cite{AAZ90} and negatively answered by Roitman in~\cite{mR93}. The dual specialized problem of whether, for a prescribed field~$F$, the monoid algebra $F[M]$ is atomic when $M$ is atomic was settled more recently by the authors in~\cite{CG19}, where they gave a negative answers for any field of prime cardinality. Even more recently, Rabinovitz and the second author in~\cite{GR23} provided a more general answer, constructing a rank-one torsion-free atomic monoid $M$ such that the monoid algebra $F[M]$ is not atomic for any field $F$.
\smallskip

Before proving that atomicity does not ascend to monoid algebras, we need an auxiliary lemma about the irreducibility of certain polynomials.

\begin{lemma} \label{lem:irreducible polynomials}
	The following statements hold.
	\begin{enumerate}
		\item For each $n \in \mathbb{N}$, the polynomial $x^{2 \cdot 3^n} + x^{3^n} + 1$ is irreducible in $\ff_2[x]$.
		\smallskip
		
		\item Let $F$ be a field of finite characteristic $p$ and $n \in \nn$ be such that $p \nmid n$. Then the polynomial $x^n+y^n+x^ny^n$ is irreducible in $F[x,y]$.
	\end{enumerate}
	
\end{lemma}

\begin{proof}
	(1) If $p$ is an odd prime and $r$ is a primitive root modulo $p^2$, then $r$ is a primitive root modulo~$p^n$ for every $n \ge 2$  \cite[page~179]{KR98}. This, along with the fact that $2$ is a primitive root modulo $3^2$, guarantees that $2$ is also a primitive root modulo $3^k$ for every $k \ge 2$. Fix $n \in \nn$, and then set $f_n(x) := x^{2 \cdot 3^n} + x^{3^n} + 1$ and let $Q_n(x)$ denote the $n^{\text{th}}$ cyclotomic polynomial over $\ff_2$. From the equality $x^{3^{n+1}} - 1 = \prod_{i=0}^{n+1} Q_{3^i}(x)$, we obtain that
	\[
		Q_{3^{n+1}}(x) = \frac{x^{3^{n+1}} - 1}{\prod_{i=0}^n Q_{3^i}(x)} = \frac{x^{3^{n+1}} - 1}{x^{3^n} - 1} = f_n(x).
	\]
	Therefore $f_n(x)$ is the ${3^{n+1}}$-th cyclotomic polynomial over $\ff_2$ (see \cite[Example~2.46]{LN86}). Since $2$ is a primitive root module $3^k$ for any $k \ge 2$, the least positive integer $d$ satisfying that $2^d  \equiv 1 \pmod{3^{n+1}}$ is $\phi(3^{n+1}) = 2 \cdot 3^n$. Hence~\cite[Theorem~2.47(ii)]{LN86} guarantees that the polynomial $f_n(x)$ is irreducible.
	\smallskip
	
	(2) Set $f(x,y) = y^n(1+x^n)+x^n$. Since $1 + x^n$ and $x^n$ are relatively primes in $F[x]$,  the polynomial $f(x,y)$ is primitive as a polynomial on $y$ over $F[x]$. By Gauss's Lemma, arguing that $f(x,y)$ is irreducible in $F[x][y]$ amounts to proving that it is irreducible in $F(x)[y]$, where $F(x)$ is the field of fractions of $F[x]$. We can write now
	\[
		f(x,y) = (1+x^n)y^n + x^n = (1+x^n) \bigg( y^n + \frac{x^n}{1+x^n} \bigg).
	\]
	Set $a_x = \frac{x^n}{1+x^n}$. Then $f(x,y)$ is irreducible in $F[x,y]$ if and only if $y^n + a_x$ is irreducible in $F(x)[y]$. Using~\cite[Theorem~8.1.6]{gK89}, one can guarantee the irreducibility of $y^n + a_x$ by verifying that $a_x \notin 4 F(x)^4$ when $4$ divides $n$ and that $-a_x \notin F(x)^q$ for any prime $q$ dividing $n$. To prove that these two conditions hold suppose, by way of contradiction, that $a_x \in c F(x)^q$, where $c \in \{-1,4\}$ and $q$ is either $4$ or a prime dividing $n$. Take $h_1(x), h_2(x) \in F[x] \setminus \{0\}$ such that $h_1(x)$ and $h_2(x)$ are relatively prime in $F[x]$ and $a_x = c \big( \frac{h_1(x)}{h_2(x)} \big)^q$. Therefore
	\begin{equation} \label{eq:equality of rational fractions}
		x^n h_2(x)^q = c(1 + x^n)h_1(x)^q
	\end{equation}
	From~\eqref{eq:equality of rational fractions}, one can deduce that $h_2(x)^q$ and $1+x^n$ are associates in $F[x]$, and so there exists $\alpha \in F^\times$ such that $h_2(x)^q = \alpha(1+x^n)$. Taking derivatives in both sides of $h_2(x)^q = \alpha(1+x^n)$ and using that $p \nmid n$, we obtain that $h_2(x) = x^m$ for some $m \in \nn$, yielding that $c(1+x^n)h_1(x)^q = x^{n+mq}$. However, this contradicts that $1+x^n$ does not divide $x^{n+mq}$ in $F[x]$. Hence $f(x,y)$ is irreducible in $F[x,y]$.
\end{proof}

We are in a position to argue that atomicity does not ascend to monoid algebras over fields.

\begin{theorem} \label{thm:main result 1}
	For each $p \in \pp$, there exists a (finite-rank) torsion-free atomic monoid $M$ such that the monoid algebra $\ff_p[M]$ is not atomic.
\end{theorem}

\begin{proof}
	We fix a prime $p$ and divide the proof into the following two cases.
	\smallskip
	
	\noindent \textsc{Case 1:} $p=2$. Let $(\ell_n)_{n \ge 1}$ be a strictly increasing sequence of positive integers such that the inequality $3^{\ell_n - \ell_{n-1}} > 2^{n+1}$ holds for every $n \in \nn$, and consider the sequences $(a_n)_{n \ge 1}$ and $(b_n)_{n \ge 1}$ with terms
	\[	
		a_n := \frac{2^n3^{\ell_n} - 1}{2^{2n} 3^{\ell_n}} \quad \text{and} \quad b_n := \frac{2^n3^{\ell_n} + 1}{2^{2n} 3^{\ell_n}}.
	\]
	Let $M$ be the Puiseux monoid generated by the set $A = \{a_n, b_n : n \in \nn\}$. We have already seen in Example~\ref{ex:ad-hoc atomic PM without the ACCP} that $M$ is an atomic monoid not satisfying the ACCP.
	
	Before proving that $\ff_2[M]$ is not atomic, we need to argue that each factor of the element $x^2 + x + 1$ in $\ff_2[M]$ has the form $\big( x^{2 \frac{1}{2^k}} + x^{\frac{1}{2^k}} + 1 \big)^t$ for some $k \in \mathbb{N}_0$ and $t \in \mathbb{N}$. First, note that because $\big\langle \frac1{2^k} : k \in \mathbb{N}_0 \big\rangle$ is a submonoid of $M$, the polynomial expression $x^{2 \frac{1}{2^k}} + x^{\frac{1}{2^k}} + 1$ belongs to $\ff_2[M]$ for all $k \in \nn_0$. Now suppose that $f(x)$ is a factor of $x^2 + x + 1$ in $\ff_2[M]$, and take $g(x) \in \ff_2[M]$ such that $x^2 + x + 1 = f(x) g(x)$. Then there exists $k \in \mathbb{N}_0$ such that
	\[
		f\big( x^{6^k}\big) g\big( x^{6^k} \big) = \big( x^{6^k} \big)^2 + x^{6^k} + 1 = \big( x^{2 \cdot 3^k} + x^{3^k} + 1 \big)^{2^k}
	\]
	in the polynomial ring $\ff_2[x]$. It follows from part~(1) of Lemma~\ref{lem:irreducible polynomials} that the polynomial $x^{2 \cdot 3^k} + x^{3^k} + 1$ is irreducible in $\ff_2[x]$. Since $\ff_2[x]$ is a UFD, there exists $t \in \mathbb{N}$ such that
	\begin{align} \label{eq:factor description}
		f \big( x^{6^k}\big) = \big( x^{2 \cdot 3^k} + x^{3^k} + 1 \big)^t = \big( \big( x^{6^k} \big)^{2 \frac{1}{2^k}} + \big( x^{6^k} \big)^{\frac{1}{2^k}} + 1 \big)^t.
	\end{align}
	After changing variables in~(\ref{eq:factor description}), one obtains that $f(x) = \big( x^{2 \frac{1}{2^k}} + x^{\frac{1}{2^k}} + 1 )^t$. Thus, each factor of $x^2 + x + 1$ in $\ff_2[M]$ has the desired form.
	
	Now suppose, by way of contradiction, that the monoid algebra $\ff_2[M]$ is atomic. Then we can write $x^2 + x + 1 = \prod_{i=1}^n f_i(x)$ for some $n \in \mathbb{N}$ and irreducible elements $f_1(x), \dots, f_n(x)$ in $\ff_2[M]$. Since $f_1(x)$ is a factor of $x^2 + x + 1$, there exist $k \in \mathbb{N}_0$ and $t \in \mathbb{N}$ such that $f_1(x) = \big( x^{2 \frac{1}{2^k}} + x^{\frac{1}{2^k}} + 1 )^t$. As $f_1(x)$ is irreducible, $t=1$. Now the equality $f_1(x) = \big( x^{2 \frac{1}{2^{k+1}}} + x^{\frac{1}{2^{k+1}}} + 1 \big)^2$ contradicts the fact that $f_1(x)$ is irreducible in $\ff_2[M]$. Hence $\ff_2[M]$ is not atomic.
	\smallskip
	
	\noindent \textsc{Case 2:} $p$ is odd. Let $(p_n)_{n \ge 1}$ be the strictly increasing sequence with underlying set $\pp$, and then consider the Puiseux monoid
	\[
		M_p := \Big\langle \frac{1}{p^n p_n} : p_n \neq p \Big\rangle.
	\]
	As in Example~\ref{ex:Grams monoid}, we can check that $M_p$ is atomic but does not satisfy the ACCP (note that $M_2$ is the Grams' monoid). Now set $M := M_p \times M_p$ and observe that~$M$ is a torsion-free, rank-$2$, atomic monoid. 
	
	We proceed to argue that the monoid algebra $\ff_p[M]$ is not atomic. To improve notation, for each $(q,r) \in M$, we write the monomial expression $x^{(q,r)}$ of $\ff_p[M]$ as a polynomial expression in two variables, namely, $x^q y^r$. We claim that each non-unit factor of $f := x + y + xy$ in $\ff_p[M]$ has the form
	\[
		\big( x^{\frac{1}{p^k}} + y^{\frac{1}{p^k}} + x^{\frac{1}{p^k}}y^{\frac{1}{p^k}} \big)^t
	\]
	for some $k \in \nn_0$ and $t \in \nn$. To prove our claim, let $g \in \ff_p[M]$ be a non-unit factor of $f$, and take $h \in \ff_p[M]$ such that $f = g \, h$. Then there exist $k \in \nn_0$ and $a \in \nn$ with $p \nmid a$ such that $g(x^{ap^k}, y^{ap^k})$ and $h(x^{ap^k}, y^{ap^k})$ are both in the polynomial ring $\ff_p[x,y]$. After changing variables, we obtain
	\[
		g(x^{ap^k},y^{ap^k})h(x^{ap^k},y^{ap^k}) = x^{ap^k} + y^{ap^k} + x^{ap^k}y^{ap^k} = (x^a + y^a + x^ay^a)^{p^k}.
	\]
	By part~(2) of Lemma~\ref{lem:irreducible polynomials}, the polynomial $x^a + y^a + x^ay^a$ is irreducible in the polynomial ring $\ff_p[x,y]$. Since $\ff_p[x,y]$ is a UFD, there exists $t \in \nn$ such that $g\big( x^{ap^k}, y^{ap^k} \big) = \big(x^a + y^a + x^ay^a \big)^t$. Going back to the original variables, we obtain $g(x,y) = \big( x^{\frac{1}{p^k}}+ y^{\frac{1}{p^k}} + x^{\frac{1}{p^k}} y^{\frac{1}{p^k}} \big)^t$, which establishes our claim. Now observe that $f$ is not irreducible because $f = \big( x^\frac1p + y^\frac1p + x^\frac1p y^\frac1p \big)^p$. In light of the established claim, any factor $g$ of $f$ in a potential decomposition into irreducibles of $\ff_p[M]$ must be of the form $\big( x^{\frac{1}{p^k}} + y^{\frac{1}{p^k}} + x^{\frac{1}{p^k}}y^{\frac{1}{p^k}} \big)^t$ and, therefore,
	\begin{align} \label{eq:main theorem I}
		g = \big( x^{\frac{1}{p^{k+1}}} + y^{\frac{1}{p^{k+1}}} + x^{\frac{1}{p^{k+1}}}y^{\frac{1}{p^{k+1}}} \big)^{pt}.
	\end{align}
	Since $x^{\frac{1}{p^{k+1}}} + y^{\frac{1}{p^{k+1}}} + x^{\frac{1}{p^{k+1}}}y^{\frac{1}{p^{k+1}}} \in \ff_p[M]$, the equality~(\ref{eq:main theorem I}) would contradict that $g$ is an irreducible element of $\ff_p[M]$. Thus, the monoid algebra $\ff_p[M]$ is not atomic.
\end{proof}

\bigskip
\section{Hereditary atomicity}

The main purpose of this section is to discuss hereditary atomicity, which has been studied in the recent papers~\cite{CGH23,fG23,GL24,GV23} motivated by~\cite{rG70}, where Gilmer studied hereditarily Noetherian domains.

\begin{definition}
	An integral domain $R$ (resp., a monoid $M$) is called \emph{hereditarily atomic} if every subring of $R$ (resp., every submonoid of $M$) is atomic.
\end{definition}

It was proved by Li and the second author in~\cite{GL24} that every hereditarily atomic monoid satisfies the ACCP (and that both conditions are equivalent in the class of reduced monoids), and it was conjectured by Hasenauer and the authors in~\cite[Conjecture~6.1]{CGH23} that every hereditarily atomic integral domain satisfies the ACCP. Hoping to motivate work on the same conjecture, which we proceed to highlight, we outline in this last section some recent progress towards a better understanding of hereditary atomicity.

\begin{conjecture} \label{conj:HA}
	Every hereditarily atomic domain satisfies the ACCP.
\end{conjecture}

Observe that the converse of the statement in Conjecture~\ref{conj:HA} does not hold: the integral domain $\qq[x]$ satisfies the ACCP but contains the non-atomic subring $\zz + x\qq[x]$.

\medskip
\subsection{Hereditarily Atomic Fields}

Although the notion of hereditary atomicity is not well understood yet in the more general context of integral domains, hereditarily atomic fields have been characterized by Hasenauer and the authors in~\cite[Theorem~4.4]{CGH23}. There are plenty of examples of fields that are not hereditarily atomic, including, of course, the fields of fractions of non-atomic domains. It is worth noting that there are hereditarily atomic domains whose corresponding fields of fractions are not hereditarily atomic.

\begin{example} \label{ex:rings of polynomials over Z are HAD}
	Consider the polynomial ring $R = \zz[x]$. Since $R^\times = \{\pm 1\}$, for each subring $S$ of $R$ we see that $S^\times = \{\pm 1\} = R^\times \cap S$. Therefore it follows now from Proposition~\ref{prop:submonoids satisfying ACCP} that~$S$ satisfies the ACCP and so it is atomic. Thus, $R$ is hereditarily atomic. However, the quotient field $\qq(x)$ of $R$ is not hereditarily atomic as it contains the non-atomic subring $\zz + x \qq[x]$.
\end{example}

The primary purpose of this subsection is to present the characterization of hereditarily atomic fields given in~\cite{CGH23}. To do so, we need the next two lemmas.

\begin{lemma}\label{val}
	For a valuation domain $V$, the following conditions are equivalent.
	\begin{enumerate}
		\item[(a)] $V$ is atomic.
		\smallskip
		
		\item[(b)] $\dim V \le 1$ and $V$ is discrete.
		\smallskip
		
		\item[(c)] $V$ is Noetherian (and so a DVR).
		\smallskip
		
		\item[(d)] $V$ is discrete and Archimedean.
	\end{enumerate}
\end{lemma}

\begin{proof}
	(a) $\Rightarrow$ (b): Suppose, by way of contradiction, that there are nonzero prime ideals $P$ and $Q$ such that $P \subsetneq Q$. Let $p$ be a nonzero element of $P$. Since $V$ is atomic, $p$ decomposes into irreducibles, namely, $p = a_1 \cdots a_m$ for some $a_1, \dots, a_m \in \ii(V)$. Because $P$ is prime, $a_j \in P$ for some $j \in \ldb 1,m \rdb$. Take $q \in Q \setminus P$. Since $a_j \nmid_V q$, the fact that $V$ is a valuation domain guarantees that $q \mid_V a_j$. Thus, $q$ and $a_j$ are associates, contradicting that $q \notin P$. Therefore $\dim V \leq 1$. We now note that if $V$ is not discrete, then $V$ has no irreducible elements: this is because the value group of~$V$ has no minimum positive element. 
	\smallskip
	
	(b) $\Leftrightarrow$ (c): This is well known.
	\smallskip
	
	(c) $\Rightarrow$ (a): Every Noetherian domain is atomic (see~\cite[Proposition~2.2]{AAZ90}).
	\smallskip
	
	(c) $\Rightarrow$ (d): This is also clear as every DVR is a discrete valuation and every Noetherian domain is Archimedean by Krull's Intersection Theorem.
	\smallskip
	
	(d) $\Rightarrow$ (c): It suffices to argue that $\dim V \le 1$. To this end, we suppose that~$P$ and $Q$ are prime ideals such that $P \subsetneq Q$. Now take $x \in Q \setminus P$ and $y \in P$. For each $n \in \mathbb{N}$, it is clear that $x^n\in Q \setminus P$, and so $x^n \mid_V y$. Therefore $y \in \bigcap_{n \in \mathbb{N}} Rx^n$, and the fact that $V$ is Archimedean guarantees that $y=0$. Hence $P$ is the zero ideal, and we can conclude that $\dim V \le 1$.
\end{proof}

Let us take a look at some necessary conditions for a field to be hereditarily atomic.

\begin{lemma} \label{lem:necessary conditions for HAFs}
	Let $F$ be a field. If $F$ is hereditarily atomic, then the following statements hold.
	\begin{enumerate}
		\item If $\emph{\ch}(F) = 0$, then $F$ is algebraic over $\mathbb{Q}$.
		\smallskip
		
		\item If $\emph{\ch}(F) = p \in \pp$, then the transcendence degree of $\ff_p \subseteq F$ is at most $1$.
		\smallskip
		
		\item Every valuation of $F$ is both discrete and Archimedean.
	\end{enumerate}
\end{lemma}

\begin{proof}
	(1) Suppose first that $\ch(F) = 0$. Assume, by way of contradiction, that $F$ is not algebraic over $\qq$, and consider the subring $R = \zz + t \qq[t]$ of $F$, where $t \in F$ is a transcendental element over~$\qq$. It is clear that $R$ is isomorphic to the subring $\zz + x \qq[x]$ of the ring of polynomials $\qq[x]$. Since~$\zz$ is not a field, it follows from \cite[Corollary~1.4]{AeA99} that $R$ is not atomic, which is a contradiction.
	\smallskip
	
	(2) Suppose now that $\ch(F) = p \in \pp$. Similar to part~(1), if $t_1, t_2 \in F$ were algebraically independent over $\ff_p$, then the subring $\ff_p[t_1] + t_2 \ff_p(t_1)[t_2]$ of $F$ would be isomorphic to the non-atomic subring $\ff_p[t_1] + x \ff_p(t_1)[x]$ of the polynomial ring $\ff_p(t_1)[x]$.
	\smallskip
	
	(3) For this, it suffices to note that if $F$ has a valuation that is not discrete or a valuation that is not Archimedean, then its corresponding valuation domain will not be atomic by Lemma~\ref{val}.	
\end{proof}

By virtue of Lemma~\ref{lem:necessary conditions for HAFs}, in order to characterize hereditarily atomic fields, we can just restrict our attention to algebraic extensions of $\qq$ and field extensions of $\ff_p$ of transcendence degree at most $1$. As the following example illustrates, not all such field extensions are hereditarily atomic.

\begin{example}
	Take $p \in \pp$, and consider the additive submonoid $M = \zz[\frac1p]_{\ge 0}$ of $\qq_{\ge 0}$, where $\zz[\frac1p]$ is the localization of $\zz$ at the multiplicative set $\{p^n : n \in \nn_0\}$. We claim that the algebraic extension $F := \ff_p(x^m : m \in M)$ of the field $\ff_p(x)$ is not hereditarily atomic. To argue this, observe that the monoid algebra $\ff_p[M]$ is antimatter and, therefore, non-atomic: indeed, as $\ff_p$ is a perfect field and $M = pM$, every polynomial expression in $\ff_p[M]$ is a $p$-th power in $\ff_p[M]$. Since $\ff_p[M]$ is a non-atomic subring of $F$, we conclude that $F$ is not hereditarily atomic.
\end{example}

In order to characterize the fields that are hereditarily atomic the only tool we still need is \cite[Theorem~3.6]{CGH23}, stated as Theorem~\ref{tool} below, which characterizes certain Pr\"ufer domains whose overrings are atomic. A directed system $\{R_\lambda\}_{\lambda \in \Lambda}$ of integral domains is called a \emph{directed integral system} of integral domains if, for all subindices $\alpha, \beta \in \Lambda$ with $\alpha \le \beta$, the ring extension $R_\alpha \subseteq R_\beta$ is integral and $[\text{qf}(R_\beta) : \text{qf}(R_\alpha)] < \infty$.

\begin{theorem}\label{tool}
	Let $D$ be a Pr\"{u}fer domain that is the union of a directed integral system of Dedekind domains. Then all the overrings of $D$ are atomic if and only if $D$ is a Dedekind domain.
\end{theorem}

We have decided to omit the proof of Theorem~\ref{tool} given in~\cite{CGH23} as it is somehow technical. We proceed to characterize the fields that are hereditarily atomic.

\begin{theorem}\label{thm:meat}
	Let $F$ be a field. 
	\begin{enumerate}
		\item If $\emph{char}(F) = 0$, then $F$ is hereditarily atomic if and only if $F$ is an algebraic extension of $\mathbb{Q}$ such that $\overline{\mathbb{Z}}_{F}$ is a Dedekind domain.
		\smallskip
		
		\item If $\emph{char}(F) = p \in \pp$, then $F$ is hereditarily atomic if and only if the transcendental degree of~$F$ over $\ff_p$ is at most~$1$ and $\overline{\ff_p[x]}_F$ is a Dedekind domain for every $x \in F$.
	\end{enumerate}
\end{theorem}

\begin{proof}
	(1) For the direct implications, suppose that $F$ is hereditarily atomic. It follows from Lemma~\ref{lem:necessary conditions for HAFs} that $F$ is an algebraic extension of $\qq$. Since $\zz$ is a Pr\"ufer domain with quotient field $\qq$, it follows that $\overline{\zz}_F$ is a (one-dimensional) Pr\"ufer domain \cite[Theorem~101]{iK74}. Note now that the set of all finite sub-extensions of $\overline{\zz}_F$ is a directed system of Dedekind domains. This directed system is clearly a directed integral system with directed union $\overline{\zz}_F$. As $F$ is hereditarily atomic, every overring of $\overline{\zz}_F$ is atomic and, therefore, Theorem~\ref{tool} guarantees that $\overline{\zz}_F$ is a Dedekind domain.
	
	Conversely, suppose that $F$ is an algebraic extension of $\qq$ such that $\overline{\zz}_F$ is a Dedekind domain. Let~$S$ be a subring of $F$. As $\overline{\zz}_F \subseteq \overline{S}_F$ holds, the fact that~$F$ is algebraic over both $\qq$ and $\qf(S)$ guarantees that $\qf(\overline{\zz}_F) = F = \qf(\overline{S}_F)$. Therefore $\overline{S}_F$ is an overring of $\overline{\zz}_F$. Since $\overline{\zz}_F$ is a Dedekind domain, so is $\overline{S}_F$ \cite[Theorem~6.21]{LM71}. Thus, $\overline{S}_F$ is Noetherian and, in particular, it must satisfy ACCP. On the other hand, $\overline{S}_F^\times \cap S = S^\times$ because $S \subseteq \overline{S}_F$ is an integral extension. From this, one infers that~$S$ also satisfies ACCP. As a consequence, $S$ is atomic. Hence $F$ is hereditarily atomic.
	\smallskip
	
	(2) Suppose first that $F$ is hereditarily atomic. It follows from Lemma~\ref{lem:necessary conditions for HAFs} that the transcendence degree of $F$ over $\ff_p$ is at most~$1$. Assume that $F$ is not algebraic over $\ff_p$, and fix a transcendental $x \in F$ over $\ff_p$. Then $F$ is an algebraic extension of $\ff_p(x)$. Since $\ff_p[x]$ is a Pr\"ufer domain and $\overline{\ff_p[x]}_F$ is one-dimensional, we can mimic the argument in the first paragraph of this proof to conclude that $\overline{\ff_p[x]}_F$ is a Dedekind domain.
	
	For the reverse implications, we first observe that if $F$ is an algebraic extension of $\ff_p$ for some $p \in \pp$, then every subring of $F$ is a field, whence $F$ is hereditarily atomic. We suppose, therefore, that the transcendental degree of $F$ over $\ff_p$ is~$1$. Let $S$ be a subring of $F$. If every element of $S$ is algebraic over $\ff_p$, then $S$ is a field, and so atomic. Otherwise, let $x \in S$ be a transcendental element over $\ff_p$. Then $\ff_p[x]$ is a subring of $S$. Since $F$ is an algebraic extension of $\ff_p(x)$ and $\overline{\ff_p[x]}_F$ is a Dedekind domain, we can argue that $S$ is hereditarily atomic by simply following the lines of the second paragraph of this proof. Thus,~$F$ is hereditarily atomic.
\end{proof}

\begin{corollary} \label{cor:F(x) is HA when $F$ is an algebraic extension of F_p}
	For $p \in \pp$, let $F$ be an algebraic extension of $\ff_p$. Then the field $F(x)$ is hereditarily atomic.
\end{corollary}

\begin{proof}
	It suffices to show that $F(x)$ is hereditarily atomic assuming that $F$ is the algebraic closure of~$\ff_p$. It is clear that $\overline{\ff_p[x]}_{F(x)} \subseteq \overline{F[x]}_{F(x)} = F[x]$ because $F[x]$ is integrally closed. Conversely, every element of $F[x]$ is integral over $\ff_p[x]$, and so $F[x] \subseteq \overline{\ff_p[x]}_{F(x)}$. Thus, $\overline{\ff_p[x]}_{F(x)} = F[x]$ is a Dedekind domain, and it follows from part~(2) of Theorem~\ref{thm:meat} that $F(x)$ is hereditarily atomic. 
\end{proof}

\medskip
\subsection{Polynomial and Power Series Extensions}

We can harness the characterization of hereditarily atomic fields given in Theorem~\ref{thm:meat} to determine the polynomial rings over fields that are hereditarily atomic.

\begin{proposition} \label{prop:polynomial rings that are HAD}
	Let $F$ be a field. Then the following conditions are equivalent.
	\begin{enumerate}
		\item[(a)] $F[x_1, \dots, x_n]$ is hereditarily atomic for any indeterminates $x_1, \dots, x_n$.
		\smallskip
		
		\item[(b)] $F[x]$ is hereditarily atomic.
		\smallskip
		
		\item[(c)] $F$ is an algebraic extension of $\ff_p$ for some $p \in \pp$.
	\end{enumerate}
\end{proposition}

\begin{proof}
	(a) $\Rightarrow$ (b): This is clear.
	\smallskip
	
	(b) $\Rightarrow$ (c): Suppose that $F[x]$ is hereditarily atomic. As we have observed, $\ch(F)$ cannot be zero, as otherwise $F[x]$ would contain a copy of the non-atomic domain $\zz + x\qq[x]$. Let $p$ be the characteristic of $F$. Since every subring of $F$ is atomic, it follows from part~(b) of Theorem~\ref{thm:meat} that $F$ is an algebraic extension of either $\ff_p$ or $\ff_p(y)$. Note, however, that~$F$ cannot be an extension of $\ff_p(y)$ because, in such a case, $\ff_p[y] + x\ff_p(y)[x]$ would be a non-atomic subring of $F[x]$ by \cite[Corollary~1.4]{AeA99}. Hence $F$ must be an algebraic extension of $\ff_p$.
	\smallskip
	
	(c) $\Rightarrow$ (a): Suppose now that $F$ is an algebraic extension of the finite field $\ff_p$ for some $p \in \pp$, and set $R = F[x_1, \dots, x_n]$. Let~$S$ be a subring of $R$. If $S$ is a subring of $F$, then $S$ must be atomic by Theorem~\ref{thm:meat} (indeed, in this case, $S$ is a field). Assume, therefore, that $S$ is not a subring of $F$, and take $a \in S \setminus F$ satisfying that for all $g,h \in R$ with $a = gh$, either $g \in F$ or $h \in F$. Write $a = g h$ for some $g, h \in S$ and suppose, without loss of generality, that $g \in F$. Since $F \cap S$ is a subring of $F$, it must be a field, and so $g \in F \cap S$ implies that $g \in S^\times$. Thus, $a \in \ii(S)$. Finally, if $f \in S$, then we can write $f = a_1 \cdots a_m$ for $a_1, \dots, a_m \in S \setminus F$ taking $m \in \nn$ as large as it can be. In this case, $a_1, \dots, a_m$ must be irreducibles in~$S$, and so $a_1 \cdots a_m$ is a factorization of~$f$ in~$S$. Hence $S$ is atomic, and we can conclude that $R$ is hereditarily atomic.
\end{proof}

The following corollary is an immediate consequence of Theorem~\ref{thm:meat} and Proposition~\ref{prop:polynomial rings that are HAD}.

\begin{corollary}[cf. Example~\ref{ex:rings of polynomials over Z are HAD}] \label{cor:polynomial rings of several variables over F_p are HAD}
	Let $F$ be an algebraic extension of $\ff_p$ for some $p \in \pp$. For $n \ge 2$, the ring of polynomials $F[x_1, \dots, x_n]$ is a hereditarily atomic domain whose quotient field is not hereditarily atomic.
\end{corollary}

We conclude this subsection emphasizing that, as for the property of being atomic, hereditary atomicity may not ascend from an integral domain $R$ to its ring of polynomials $R[x]$: for instance, the field~$\qq$ is hereditarily atomic by Theorem~\ref{thm:meat}, but the ring $\qq[x]$ contains the subring $\zz + x \qq[x]$, which is not atomic.

\medskip
\subsection{Laurent Polynomial Extensions}

We have just identified in Proposition~\ref{prop:polynomial rings that are HAD} a class of polynomial rings that are hereditarily atomic. In the next proposition we fully characterize the rings of Laurent polynomials that are hereditarily atomic.

\begin{theorem} \label{thm:Laurent polynomial rings}
	Let $R$ be an integral domain. Then $R[x^{\pm 1}]$ is hereditarily atomic if and only if~$R$ is an algebraic extension of $\ff_p$ for some $p \in \pp$.
\end{theorem}

\begin{proof}
	For the reverse implication, suppose that $R$ is an algebraic extension of $\ff_p$ for some $p \in \pp$. In particular, we can assume that $R$ is a subfield of the algebraic closure $K$ of $\ff_p$. Then $R[x^{\pm 1}]$ is a subring of $K(x)$. Since $K(x)$ is hereditarily atomic by Corollary~\ref{cor:F(x) is HA when $F$ is an algebraic extension of F_p}, we obtain that $R[x^{\pm 1}]$ is also hereditarily atomic.
	\smallskip
	
	For the direct implication, suppose that $R[x^{\pm 1}]$ is hereditarily atomic. We will argue first that~$R$ cannot have characteristic zero. Suppose, towards a contradiction, that $\text{char}(R) = 0$. Then $R[x^{\pm 1}]$ contains $\zz[x^{\pm 1}]$ as a subring. Now consider the subring
	\[
		T := \zz\bigg[x^n, \frac2{x^n} : n \in \nn \bigg]
	\]
	of $\zz[x^{\pm 1}]$. One can readily verify that $x^{-1} \notin T$. This, along with the inclusion $T^\times \subseteq \{\pm x^n : n \in \zz\}$, implies that $T^\times = \{\pm 1\}$. As a consequence, $\frac2{x^j} = x\big(\frac{2}{x^{j+1}}\big)$ ensures that $\frac2{x^j} \notin \ii(T)$ for any $j \in \nn_0$. Since~$T$ is a subring of $R[x^{\pm 1}]$, it must be atomic. Therefore we can write $2 = a_1(x) \cdots a_n(x)$ for some irreducibles $a_1(x), \dots, a_n(x)$ of $T$, where $n \ge 2$ because $2 \notin T^\times \cup \ii(T)$. Since the monoid
	\[
		M := \{cx^n : c \in \zz \setminus \{0\} \text{ and } n \in \zz\}
	\]
	is a divisor-closed submonoid of the multiplicative monoid $\zz[x^{\pm 1}] \setminus \{0\}$, it follows that $a_j(x) \in M$ for every $j \in \ldb 1,n \rdb$. After relabeling and taking associates, we can assume that $a_1(x) = 2x^{s_1}$ and $a_j(x) = x^{s_j}$ for every $j \in \ldb 2,n \rdb$, where $s_1, \dots, s_n \in \zz$ satisfy $s_1 + \dots + s_n = 0$. The fact that $a_1(x) \in \ii(T)$ guarantees that $s_1 \ge 1$. Thus, there must be a $k \in \ldb 2,n \rdb$ such that $s_k < 0$. However, in this case $a_k(x) = x^{s_k}$ would be a unit of $T$, which is a contradiction.
	\smallskip
	
	Hence $\text{char}(R) = p$ for some $p \in \pp$, and so $R$ contains $\ff_p$ as its prime subfield. We will argue that~$R$ is an algebraic extension of $\ff_p$. Assume, by way of contradiction, that there exists $w \in R$ that is transcendental over $\ff_p$. Let us prove that the subring
	\[
		S := \ff_p\bigg[x, \frac{w}{x^n} : n \in \nn_0 \bigg]
	\]
	of $R[x^{\pm 1}]$ is not atomic, which will yield the desired contradiction.
	
	If $x$ were a unit of $S$, then $x^{-1} = \sum_{i=0}^k g_i(x) w^i$ for some $g_0 \in \ff_p[x]$ and $g_1, \dots, g_k \in \ff_p[x^{\pm 1}]$, and so the fact that $w$ is transcendental over $\ff_p(x)$ (as an element of the extension $\text{qf}(R)(x)$) would imply that $x g_0(x) = 1$. Thus, $x \notin S^\times$. Similarly, if $\frac{w}{x^n}$ were a unit of $S$ for some $n \in \nn_0$, then $\frac{x^n}w = \sum_{i=0}^k g_i(x) w^i$ for some $g_0, \dots, g_k \in \ff_p[x,x^{-1}]$ and, after clearing denominators, we would obtain $F(x,w) = 0$ for some nonzero $F$ in the polynomial ring $\ff_p[X,W]$, contradicting that $\{x,w\}$ is algebraically independent over $\ff_p$. Thus, $\frac{w}{x^n} \notin S^\times$ for any $n \in \nn_0$.
	
	We proceed to argue that $w$ does not factor into irreducibles in $S$. Observe first that for every $m \in \nn_0$ the element $\frac{w}{x^m}$ is not irreducible in $S$ because $\frac{w}{x^m} = x\big(\frac{w}{x^{m+1}}\big)$ and, as we have already checked, $x, \frac{w}{x^{m+1}} \notin S^\times$. In particular, $w$ is not irreducible in $S$. Now write
	\begin{equation} \label{eq:product}
		w = \prod_{i=1}^n f_i\Big(x, w, \frac wx, \dots, \frac w{x^k} \Big) 
	\end{equation}
	in~$S$ for some $k \in \nn_0$, $n \in \nn_{\ge 2}$, and $f_1, \dots, f_n \in \ff_p[X, W_0, \dots, W_k]$ such that $f_i\big(x,w, \frac{w}x, \dots, \frac{w}{x^k}\big)$ is not a unit of $S$ for any $i \in \ldb 1,n \rdb$. Now observe that for every index $i \in \ldb 1,n \rdb$, there exists a unique $\ell_i \in \zz$ such that $g_i := x^{-\ell_i} f_i$ belongs to $\ff_p[x,w]$ and $g_i(0,w) \neq 0$. After setting $\ell := \ell_1 + \dots + \ell_n$, we obtain from~\eqref{eq:product} that $g_1(x,w) \cdots g_n(x,w) = x^{-\ell} w$, which implies that $\ell = 0$. As a consequence, we see that $w = g_1(x,w) \cdots g_n(x,w) \in \ff_p[x,w]$. Since $\ff_p[x,w]$ is a UFD and $w$ is irreducible in $\ff_p[w,x]$, we can assume, after a possible relabeling of subindices, that $g_1 = \alpha_1 w$ for some $\alpha_1 \in \ff_p^\times$ and $g_i = \alpha_i \in \ff_p^\times$ for every $i > 1$. Thus, $f_1 = \alpha_1 x^{\ell_1}w$ and $f_i = \alpha_i x^{\ell_i}$ for every $i > 1$. Now we observe that $f_1$ is not irreducible in $S$ because $f_1 = \alpha_1 x^N \big(\frac{w}{x^{N-\ell_1}}\big)$, where $N := |\ell_1| + 1$, and neither $\alpha_1 x^N$ nor $\frac{w}{x^{N-\ell_1}}$ is a unit of $S$. As a consequence, $w$ does not factor into irreducibles in $S$, and so $S$ is a subring of $R[x^{\pm 1}]$ that is not atomic, contradicting that $R[x^{\pm 1}]$ is hereditarily atomic. Hence $R$ is an algebraic extension of $\ff_p$, which concludes the proof.
\end{proof}

\bigskip
\section{Weak Atomicity}
\label{sec:weak atomicity}

In this final section, we consider two algebraic notions that are weaker than atomicity. In addition, as part of this section we present a ``Kaplansky-type'' theorem mimicking Theorem~\ref{kap}, which we mentioned in the introduction.

As mentioned before, one of the difficulties with atomicity is the fact that, in general, it is not preserved under localization. A central part of this problem comes from the observation that if one considers a multiplicative set generated by atoms (surely a natural set to consider) there is no guarantee that this set is saturated. With this in mind, we discuss quasi-atomicity and almost atomicity, two weaker notions of atomicity introduced in~\cite{BC15} and designed to ``repair'' the defect that not all multiplicative sets generated by atoms are saturated.

\medskip
\subsection{Quasi-Atomicity} 

First, we will take a look at integral domains each nonzero element divides an atomic element; this property encapsulates the notion of quasi-atomicity. Let us give the formal notion of a quasi-atomic element/domain similar to the way we introduced an atomic element/domain in the background section.

\begin{definition}
	Let $R$ be an integral domain.
	\begin{itemize}
		\item $r \in R^*$ is a \emph{quasi-atomic} element if $rs$ is atomic for some $s \in R$.
		\smallskip
		
		\item $R$ is \emph{quasi-atomic} if every nonzero element of $R$ is quasi-atomic.
	\end{itemize}
\end{definition}

It follows directly from the definitions that every atomic domain is quasi-atomic, and we will discuss examples of quasi-atomic domains that are not atomic later in this section. 
\smallskip

Let $R$ be an integral domain, and let $Q(R)$ denote the set consisting of all quasi-atomic elements of~$R$. It turns out that $Q(R)$ is a multiplicative subset of~$R$.

\begin{proposition}
	For any integral domain $R$, the set $Q(R)$ is a multiplicative set; that is $Q(R)$ is a submonoid of $R^*$.
\end{proposition}

\begin{proof}
	Suppose that $r_1, r_2 \in Q(R)$. By definition, there exist nonzero $s_1, s_2 \in R$ such that elements $r_1 s_1$ and $r_2 s_2$ are atomic. Therefore the element $r_1 r_2 s_1 s_2$ is also atomic, which implies that $r_1 r_2$ is quasi-atomic.
\end{proof}

The following is a ``Kaplansky-type'' theorem mimicking the earlier recorded Theorem \ref{kap}. This is a weaker version of the Kaplansky's Theorem, as irreducibility (maximal among the set of principal ideals) is markedly weaker than the notion of prime.

\begin{theorem}\label{thm:qa}
	Let $R$ be an integral domain. Then $R$ is quasi-atomic if and only if every nonzero prime ideal of $R$ contains an irreducible element.
\end{theorem}

\begin{proof}
	For the direct implication, suppose that $R$ is quasi-atomic, and let $\mathfrak{p}$ be a nonzero prime ideal of $R$. Select a nonzero element $r \in \mathfrak{p}$ and note that since~$R$ is quasi-atomic, there is an element $s \in R$ such that $rs = a_1 \cdots a_n$ for some irreducibles $a_1, \dots, a_n$ in $R$. As $\mathfrak{p}$ is prime and $a_1 \cdots a_n \in \mathfrak{p}$, there is an index $i \in \ldb 1,n \rdb$ such that $a_i \in \mathfrak{p}$. Hence $\mathfrak{p}$ contains an irreducible.
	
	For the reverse implication, suppose that every nonzero prime ideal of~$R$ contains an irreducible element. Assume, by way of contradiction, that $R$ is not quasi-atomic. Take an element $z \in R^*$ that is not quasi-atomic. Observe that the saturated multiplicative set  $Q(R)$ contains $R^\times \cup \mathcal{A}(R)$. We claim that $zR \cap Q(R) = \emptyset$: indeed, the existence of $r \in R$ such that $zr \in Q(R)$ would allow us to pick $s \in R$ such that $z(rs)$ is atomic, which is not possible because $z \notin Q(R)$. Hence $zR\cap Q(R)$ is empty. Now we expand the ideal $zR$ to a prime ideal $\mathfrak{p}_z$ that is maximal with respect to the property that $\mathfrak{p}_z \cap Q(R)$ is empty. Since $\mathfrak{p}_z$ is a nonzero prime ideal, it follows from our initial assumption that~$\mathfrak{p}_z$ must contain an irreducible element. However, this contradicts the fact that $\mathfrak{p}_z \cap Q(R)$ is empty, which completes the proof.
\end{proof}

As an immediate consequence, we obtain the following corollary (see \cite[Lemma~5.2]{BC15} and \cite[Theorem~8]{nLL19}).

\begin{corollary}
	Let $R$ be an integral domain. Then every nonzero prime ideal of $R$ must contain an irreducible element.
\end{corollary}

We proceed to establish a characterization of quasi-atomicity that is parallel to Theorem~\ref{atomicity:characterization}, the fundamental characterization of atomicity we provided in the background section.

\begin{proposition} \label{quasi-atomicity:characterization}
	Let $R$ be an integral domain that is not a field, and let $G(R)$ be the group of divisibility of $R$. Then $R$ is quasi-atomic if and only if for every positive element $r R^\times \in G(R)$ there are minimal positive elements $a_1 R^\times, \dots, a_k R^\times$ such that $r R^\times
		\le a_1 R^\times \cdots a_k R^\times$.
\end{proposition}

\begin{proof}
	As pointed out in the proof of Theorem~\ref{atomicity:characterization}, for each $a \in R^*$, the coset $a R^\times \in G(R)$ is minimal positive if and only if~$a$ is irreducible in~$R$.
	
	For the direct implication, assume that $R$ is quasi-atomic. Observe that the statement that $rR^\times$ is positive in $G(R)$ is equivalent to the statement that $r \in R$. Let $rR^\times$ be a positive element of $G(R)$. By quasi-atomicity, we can pick $s \in R$ (so $sR^\times$ is positive in $G(R)$) such that $rs = a_1 \cdots a_k$ for some irreducibles $a_1, \dots, a_k$ of $R$. Thus, in $G(R)$, the inequality $r R^\times \le a_1 R^\times \cdots a_k R^\times$ must hold.
	
	On the other hand, assume that every positive element $r R^\times \in G(R)$ there are minimal positive elements $a_1 R^\times, \dots, a_k R^\times$ such that $r R^\times
	\le a_1 R^\times \cdots a_k R^\times$. Now fix $r \in R$. By assumption, we can pick irreducibles $a_1, \dots, a_k$ of~$R$ such that $rR^\times \le a_1 R^\times \cdots a_kR^\times$. Hence $r$ divides $a_1 \cdots a_k$ in $R$. We conclude then that $R$ is quasi-atomic.
\end{proof}

It is worth emphasizing that not every integral domain is quasi-atomic (this was first shown in \cite{BC15}).

\begin{example}
	Let $R: = \mathbb{Z} + x\mathbb{Q}[x]$. Note that every rational prime is an irreducible (prime) element of $R$ and note that every polynomial of the form $xf(x)\in R$ is divisible by every rational prime and so the prime ideal $x\mathbb{Q}[x]$ contains no irreducible elements. It now follows from Theorem \ref{thm:qa} that $R$ cannot be quasi-atomic.
\end{example}

Integral domains having no atoms cannot be quasi-atomic unless they are fields. These domains, which were introduced and first studied in~\cite{CDM99} by Dobbs, Mullins, and the first author, are the extreme deviation from atomicity. An integral domain $R$ is called \emph{antimatter} provided that it contains no irreducible elements. It is not hard to verify that antimatter domains are precisely the integral domains whose groups of divisibility have no minimal positive elements. Any field is vacuously an antimatter domain. We proceed to discuss some less trivial examples of antimatter domains.

\begin{example}
	Any valuation domain $(V,\mathfrak{m})$ with a non-principal maximal ideal~$\mathfrak{m}$ is an antimatter domain. It follows that one can build an antimatter domain with any positive Krull dimension. See Example~\ref{anti} for a concrete illustration of a couple of cases.
\end{example}

\begin{example}
	Let $R$ be an integral domain that is not a field, and let $F$ be the quotient field of~$R$. Then $\overline{R}_{\overline{F}}$ (the integral closure of $R$ in the algebraic closure, $\bar{F}$, of $\mathbb{F}$) is an antimatter domain that is not a field. This is perhaps the easiest way to embed an integral domain into an antimatter domain that is not a field. To argue that $\overline{R}_{\overline{F}}$ is antimatter, we note that if $\alpha\in\overline{R}_{\overline{F}}$, then $\alpha$ is a root of the monic polynomial $x^n+r_{n-1}x^{n-1}+\cdots+r_1x+r_0 \in R[x]$. We now observe that $\sqrt{\alpha}$ is a root of $x^{2n}+r_{n-1}x^{2n-2}+\cdots+r_1x^2+r_0\in R[x]$, and so in $\alpha\in\overline{R}_{\overline{F}}$, we have the factorization $(\sqrt{\alpha})(\sqrt{\alpha})=\alpha$, and so every element in $\overline{R}_{\overline{F}}$ has a square root, which makes clear the fact that $\overline{R}_{\overline{F}}$ has no irreducible elements. Also as any integral extension $R \subseteq T$ has the property that $R^\times = R \cap T^\times$ and $\overline{R}_{\overline{F}}$ is integral over $R$ and not a field, we see that $\overline{R}_{\overline{F}}$ is not a field.
\end{example}

Here is another example of the behavior of antimatter domains in the context of valuation domains. This example will also help illustrate the concepts in the sequel.

\begin{example}\label{anti}
	Suppose that $V$ is a $2$-dimensional valuation domain with value group $\mathbb{Q}\oplus\mathbb{Z}$ ordered lexicographically. In this case, the maximal ideal of $V$ is a principal ideal, which we can write $pV$. Thus, every nonunit of $V$ is divisible by $p$. However, if we localize $V$ at its height-$1$ prime ideal, then the localization is $1$-dimensional and non-discrete, and hence antimatter. On the other hand, if we reverse the order and consider a $2$-dimensional valuation domain, $W$, with value group $\mathbb{Z} \oplus \mathbb{Q}$, then $W$ is antimatter. However, when localized at its height-$1$ prime ideal, $W$ passes locally to a PID.
\end{example}

As mentioned in Section~\ref{sec:atomic domains}, it was proved by Roitman~\cite{mR93} that the property of being atomic does not ascend in general from integral domains to their corresponding polynomial rings. In addition, it was first proved by the authors in~\cite{CG19} that atomicity does not ascend from torsion-free monoids to their corresponding monoid algebras over fields. In order to motivate further research on the weak notions of atomicity we have discussed in this section, we propose the following fundamental open question, which is parallel to the ascent of atomicity (as posed by Gilmer in~\cite[page 189]{rG84}). As for an integral domain, we say that a monoid~$M$ is \emph{quasi-atomic} if every element of $M$ divides an atomic element.

\begin{question} \label{quest:ascent of quasi-atomicity}
	Let $R$ be an integral domain, and let $M$ be a torsion-free monoid. Does the fact that both $R$ and $M$ are quasi-atomic imply that the monoid algebra $R[M]$ is quasi-atomic?
\end{question}

\medskip
\subsection{Almost Atomicity}

Now we take a look at integral domains whose nonzero elements can be written as the quotient of two atomic elements. This describes, roughly speaking, the notion of almost atomicity, which is a weaker notion of atomicity that is stronger than quasi-atomicity.

\begin{definition}
	Let $R$ be an integral domain.
	\begin{itemize}
		\item $r \in R^*$ is an \emph{almost atomic} element if $rs$ is atomic for some atomic element $s \in R$.
		\smallskip
		
		\item $R$ is \emph{almost atomic} if every nonzero element of $R$ is almost atomic.
	\end{itemize}
\end{definition}

Let $R$ be an integral domain, and let $S$ be a multiplicative set of $R$. Recall that $S$ of $R$ is defined to be saturated if $S$ is a divisor-closed submonoid of $R^*$. Every multiplicative set of $R$ is contained in a saturated multiplicative subset: the \emph{saturation} of $S$ is the set consisting of all the elements of $R$ dividing an element of $S$. Clearly, the saturation of a multiplicative set is a saturated multiplicative set. We let $A(R)$ denote the set consisting of all almost atomic elements of~$R$. As $Q(R)$, it turns out that $A(R)$ is a multiplicative subset of $R$.

\begin{proposition}\label{prop:sat}
	Let $R$ be an integral domain. 
	\begin{enumerate}
		\item $A(R)$ is a submonoid of $Q(R)$, and so a multiplicative set of $R$.
		\smallskip
		
		\item $Q(R)$ is the saturation of $A(R)$ in $R$.
	\end{enumerate}
\end{proposition}

\begin{proof}
	(1) Suppose that $r_1, r_2 \in A(R)$. By definition, there exist atomic elements $s_1, s_2 \in R$ such that elements $r_1 s_1$ and $r_2 s_2$ are atomic. Therefore both elements $s_1 s_2$ and $r_1 r_2 s_1 s_2$ are atomic, which implies that $r_1 r_2$ is almost atomic. Hence $A(R)$ is a submonoid of $Q(R)$ and so a multiplicative set of $R$.
	\smallskip
	
	(2) It is clear that $A(R)\subseteq Q(R)$. By definition of a quasi-atomic element, every $r \in Q(R)$ divides an atomic element and, therefore, an almost atomic element. Therefore $Q(R)$ is contained in the saturation of $A(R)$. For the reverse inclusion, let $r \in R^*$ be a divisor of an element of $A(R)$, and take $s \in R^*$ such that $rs \in A(R)$. Then we can take an atomic element $t \in R^*$ such that $r(st)$ is atomic, which implies that $r \in Q(R)$. Thus, the saturation of $A(R)$ is contained in $Q(R)$, and so we conclude that $Q(R)$ is the saturation of $A(R)$.
\end{proof}

It is prudent at this juncture to point out that the notions of quasi-atomicity and almost atomicity are not vacuous. Therefore we presently produce examples witnessing that the class of atomic domains is strictly contained in that of almost atomic domains, and also that the class of almost atomic domains is strictly contained in that of quasi-atomic domains.

\begin{example}
	Let us prove that the integral domain $R := \mathbb{Z} + \mathbb{Z}x + x^2\mathbb{Q}[x]$ is almost atomic but not atomic. It is clear that $R^\times = \{ \pm 1\}$. In addition, observe that, since $\zz^*$ is a divisor-closed submonoid of $R^*$, any rational prime is irreducible in~$R$. 
	
	To verify that $R$ is almost atomic, fix a nonzero polynomial $f(x) \in R$, and let us check that $f(x)$ is an almost atomic element. First, take rational primes $p_1, \dots, p_n \in \pp$ not necessarily distinct such that $p_1 \cdots p_n f(x) \in \mathbb{Z}[x]$. 
	Then write
	\[
		p_1 \cdots p_n f(x) = x^mg(x)
	\]
	for some $m \in \nn_0$ and $g(x) \in \zz[x]$ such that $\text{ord} \, g(x) \in \{0,1\}$. As $x$ is an irreducible of $R$, the monomial $x^m$ is an atomic element. We proceed to show that $g(x)$ is also an atomic element. To do so, write
	\[
		g(x) = c_1 \cdots c_k g_1(x) \cdots g_\ell(x)
	\]
	for some nonunits $c_1, \ldots, c_k, g_1(x), \dots, g_\ell(x)$ in~$R$ such that $c_1, \dots, c_k$ are constant and $g_1(x), \dots, g_\ell(x)$ are nonconstant. Since $\ord \, g(x) \in \{0,1\}$, it follows that $\ord \, g_i(x) \in \{0,1\}$ for every $i \in \ldb 1,\ell \rdb$. Therefore $c_1 \cdots c_k$ divides in $\zz$ the integer coefficient $c$ of the nonzero minimum-degree term of $g(x)$, and so $k$ is bounded by the length of the factorization of $c$ in $\zz$. On the other hand, the fact that $g_i(x)$ is nonconstant for every $i \in \ldb 1, \ell \rdb$ ensures that~$\ell$ is bounded by the degree of $g(x)$. Thus, after assuming that we have chosen $k+\ell$ as large as it can possibly be, we obtain that $ c_1 \cdots c_k g_1(x) \cdots g_\ell(x)$ is a factorization of $g(x)$ in $R$, and so $g(x)$ is atomic in~$R$. This in turn implies that $f(x)$ is an almost atomic element of $R$, whence $R$ is almost atomic. 
	
	In order to conclude that $R$ is not atomic, it is enough to argue that the monomial $\frac12 x^2$ cannot be factored into irreducibles in~$R$. We first remark that every nonzero polynomial $p(x) \in \qq[x]$ whose order is at least~$2$ is not irreducible in $R$: indeed, it can be written in $R$ as $p(x) = 2 \cdot \frac{p(x)}2$. Now note that any potential factorization of $\frac{1}{2}x^2$ in $R$ would consist of monomials of $\mathbb{Q}[x]$ that are irreducible in $R$, and so the degrees of these monomials must be less than~$2$. Then in order to obtain a factorization of $\frac{1}{2}x^2$ in $R$, we should be able to write $\frac12 x^2 = (mx)(nx)$ for some $m,n \in \mathbb{Z}$, but this is clearly impossible. Hence we conclude that the integral domain~$R$ is almost atomic but not atomic.
\end{example}

Next we produce an integral domain that is quasi-atomic but not almost atomic. This integral domain is a variant of that exhibited in~\cite[Example~7]{nLL19}, but for this one we use power series instead of polynomials.

\begin{example}\label{Q}
	Consider the integral domain $R := \mathbb{Z} + \mathbb{Z}x + x^2\mathbb{R}\ldb x \rdb$, and let us show that it is quasi-atomic but not almost atomic. For a nonzero $f(x) \in \rr\ldb x \rdb$ with order $m$, we call the coefficient of the term $x^m$ the \emph{initial coefficient} of $f(x)$. 
	
	We first claim that every power series in $R$ with order at most~$1$ is atomic. To argue this, it suffices to fix a nonzero $f(x) \in R$ with $\text{ord} \, f(x) \in \{0,1\}$ and show that $f(x)$ satisfies the ACCP. Let $(f_n(x)R)_{n \ge 0}$ be an ascending chain of principal ideals of $R$ with $f_0(x)=f(x)$. For each $n \in \nn_0$, the fact that $f_{n+1}(x)$ is a divisor of $f_n(x)$ in $R$ ensures that $\text{ord} \, f_{n+1}(x) \le \text{ord} \, f_n(x)$, and so we can take $N \in \nn$ such that $\text{ord} f_n(x) = \text{ord} \, f_N(x)$ for every $n \ge N$. Moreover, as $\zz$ satisfies the ACCP, we can further assume that~$N$ is large enough so that when $n \ge N$ the initial coefficient of $f_n(x)$ equals that of $f_N(x)$. This in turn implies that for each $n \ge N$, the power series $g_n(x) := f_N(x)/f_n(x)$ belongs to $R$ and satisfies that $g_n(0) = 1$, from which we can deduce that $g_n(x) \in R^\times$. As a consequence, the chain of principal ideals $(f_n(x)R)_{n \ge 0}$ stabilizes, whence $f(x)$ satisfies the ACCP, as desired.
	
	Now fix a nonzero $f(x) \in R$ with $m := \text{ord} \, f(x) \ge 2$, and then take a power series $g(x) \in \rr\ldb x \rdb$ such that $f(x) = x^m g(x)$. We proceed to show that $f(x)$ is atomic if and only if $g(0) \in \zz$. As we can write $f(x) = x^{m-1}(xg(x))$ and both $x^{m-1}$ and $xg(x)$ belong to $R$, the reverse implication immediately follows from the already-established fact that every power series in $R$ with order at most~$1$ is atomic. For the direct implication, suppose that $f(x)$ is atomic in $R$. Observe that each power series in $R$ with order at least~$2$ is divisible by all rational primes, and so~$R$ does not contain any irreducibles of order greater than~$1$. Thus, we can write $f(x) = a_1(x) \cdots a_\ell(x)$ for some irreducibles $a_1(x), \dots, a_\ell(x)$ in $R$ whose orders belong to $\{0,1\}$. Thus, the initial coefficients of $a_1(x), \dots, a_\ell(x)$ are integers, and so $g(0) \in \zz$ because $g(0)$ is the initial coefficient of $f(x) = a_1(x) \cdots a_\ell(x)$. 
	
	Thus, the atomic elements of $R$ are its nonzero elements whose initial coefficients are integers (they include the units of $R$, which are the elements having their constant coefficients in $\{\pm 1\}$). To complete the argument, take a nonzero $f(x) \in R$ and, as in the previous paragraph, write $f(x) = x^m g(x)$, now assuming that $m \in \nn_0$ and $g(x) \in \rr\ldb x \rdb$ such that $g(0) \neq 0$. Then note that the initial coefficient of $\big(\frac{x^2}{g(0)} \big) f(x)$ is~$1$, and so it is atomic in $R$. Hence $R$ is quasi-atomic. Now suppose that $g(0) \notin \mathbb{Q}$. In this case, there is no atomic element $a(x)$ in $R$ such that the initial coefficient of $a(x)f(x)$ belongs to $\mathbb{Z}$, which implies that $R$ is not almost atomic. Thus, the integral domain $R$ is quasi-atomic but not almost atomic, as desired.
\end{example}

We proceed to provide a characterization for almost atomic domains in terms of their corresponding divisibility group, similar to the characterizations of atomicity and quasi-atomicity we gave in Theorem~\ref{atomicity:characterization} and Proposition~\ref{quasi-atomicity:characterization}, respectively.

\begin{proposition} \label{almost atomicity:characterization}
	Let $R$ be an integral domain that is not a field, and let $G(R)$ be the group of divisibility of $R$. Then $R$ is almost atomic if and only if $G(R)$ is generated by the set of minimal positive elements of $G(R)$.
\end{proposition}

\begin{proof}
	As pointed out in the proof of Theorem~\ref{atomicity:characterization}, for each $a \in R^*$, the coset $a R^\times \in G(R)$ is minimal positive if and only if~$a$ is irreducible in~$R$.
	\smallskip
	
	For the direct implication, assume that $R$ is almost atomic. Take $k \in \qf(R)^\times$ and write $k = \frac{r}{s}$ for some $r,s \in R$. Using the almost atomicity of $R$, we can find irreducibles $a_1,\dots, a_m$ and $b_1, \dots, b_n$ of $R$ such that $r a_1 \cdots a_m$ and $sb_1 \cdots b_n$ are both atomic elements in $R$. This means that $k = \frac{r}{s}$ is a quotient of products of atoms, and so $G(R)$ is generated by its minimal positive elements (i.e., cosets represented by irreducibles).
	
	Conversely, assume that $G(R)$ is generated by the set of minimal positive elements of $G(R)$. Suppose that $d \in R$ is a nonzero nonunit. If $d$ is a product of atoms, then we are done. Otherwise, the coset $d R^\times$ can be written in the form
	\[
		d R^\times = \prod_{i=1}^m a_i R^\times \prod_{j=1}^n b^{-1}_j R^\times
	\]
	for atoms $a_1, \dots, a_m$ and $b_1, \dots, b_n$ of~$R$. This means that there exists $u \in R^\times$ such that the equality $d b_1 \cdots b_n = u a_1 \cdots a_m$ holds, and so $d$ is an atomic element. Hence $R$ is almost atomic.
\end{proof}

As for an integral domain, we say that an (additive) monoid $M$ is \emph{almost atomic} if every element of $M$ can be written as the difference of two atomic elements. We conclude this section highlighting the ascent of almost atomicity as an open question for the interested reader.

\begin{question}
	Let $R$ be an integral domain, and let $M$ be a torsion-free monoid. Does the fact that both $R$ and $M$ are almost atomic imply that the monoid algebra $R[M]$ is almost atomic?
\end{question}

\medskip
\subsection{A Homological Method}

In order to generalize the study of factorizations, one might wish to remove the restriction of atomicity. In a certain sense, the classical study of factorizations in the setting of integral domains was predicated on factoring nonzero nonunit elements into irreducibles. There are two assumptions here: the first is that one can find irreducible divisors of a selected nonzero nonunit, and the second is that the selected element can be written as a finite product of some of these irreducibles. It stands to reason that a ``typical'' integral domain will be neither antimatter nor atomic, and one fundamental question is how to study factorizations or, more generally, the structure of the group of divisibility in such a theater. In this final subsection, we offer a potential approach to this question. Much of the material in this subsection is a expository condensation of the last half of the paper~\cite{CG16} and some of the material can be found in~\cite{BC15}.

To deal with factorizations in a non-atomic domain, it would seem that localization could be an effective tool. In this direction, it is natural to want to use localization to excise the non-atomic elements, but in general this approach can prove to be hazardous given the fact that the set of irreducibles is not saturated (and hence the non-atomic elements may not form a multiplicative set). We illustrate this with the following example.

\begin{example}
	Let $n\geq 2$ be a natural number, and let $V$ be an $n$-dimensional discrete valuation domain, that is, a valuation domain with value group isomorphic to $\mathbb{Z}^n$ under the lexicographical order. Since a valuation domain that is not a field is atomic if and only if its value group is $\mathbb{Z}$, we see that $V$ is not atomic. If the maximal ideal of $V$ is generated by $p$, then $p$ divides every nonunit of $V$. Hence any saturated multiplicative subset of $V$ that contains a nonunit, must contain $p$. The upshot of this is that if one were to try to form a multiplicative subset of $V$ consisting of non-atomic elements of $V$, then any such set must also contain all of the irreducibles. It turns out in this case (and in our general approach) that a more effective strategy is to form the multiplicative subset generated by the irreducibles of $V$, to peel back the first atomic layer and see if this localization yields deeper factorization data. In this particular case, if we adopt the approach of recursively localizing at the set of irreducibles, then we obtain the sequence of integral domains
	\[
		V_n:=V \subset V_{n-1} \subset \cdots \subset V_1 \subset V_0
	\]
	with each $V_i$ obtained from $V_{i+1}$ by localizing at the multiplicative subset of $V_{i+1}$ generated by its irreducibles; equivalently, each $V_i$ is the localization of $V_{i+1}$ at its prime ideal of coheight $1$. Note that~$V_1$ is the first atomic domain in the sequence and $V_0$ is the quotient field of $V$.
\end{example}

We now introduce some terminology. Let $R$ be an integral domain and $S$ a multiplicative subset of~$R$. Observe that the multiplicative group generated by the irreducibles of $R$ is identical to the multiplicative group generated by the elements that are almost atomic. It is interesting to note that the Grothendieck groups of the monoid generated by the irreducibles and the monoid generated by the almost atomic elements coincide, but in general this group does not pick up the saturation of the set of irreducibles in the integral domain as a whole (Example \ref{Q} is helpful in illustrating this point).

Measuring this defect is an interesting task and speaks to the theme of the ``layers'' of factorizations mentioned earlier in this section and lends itself to some interesting techniques.

Let $R_0$ be an integral domain, and let $S_0$ be the multiplicative set generated by the (almost) irreducibles of $R$. We inductively define, for each $n \in \nn$, the multiplicative set $S_n$ generated by the irreducibles of $R_n$, where $R_n = (R_{n-1})_{S_{n-1}}$. For each $n\in\mathbb{N}$, we introduce the following definitions.
\begin{enumerate}
	\item $G_n$ denotes the group of divisibility of $R_n$.
	\smallskip
	
	\item $A_n$ denotes the subgroup of $G_n$ generated by the (almost) irreducibles.
	\smallskip
	
	\item $Q_n$ denotes the subgroup of $G_n$ generated by the quasi-irreducibles.
\end{enumerate}

\begin{proposition}\label{QAQS}
	With notation as before, the sequence of ring injections
	\[
		R_0 \longrightarrow R_1 \longrightarrow R_2 \longrightarrow \cdots
	\]
	gives rise to the sequence of group epimorphisms
	\[
		G_0 \longrightarrow G_1 \longrightarrow G_2 \longrightarrow \cdots
	\]
	for which the projection $G_n \longrightarrow G_{n+1}$ given by $kR_n^\times \mapsto kR_{n+1}^\times$ is induced by the inclusion $R_n \longrightarrow R_{n+1}$.
\end{proposition}

Valuation domains illustrate this nicely, as the following example shows.

\begin{example}
	Let $V_n$ be an $n$-dimensional discrete valuation domain with value group $\mathbb{Z}^n$ ordered lexicographically. The sequence of localizations in Proposition \ref{QAQS} corresponds to the sequence of successive localizations 
	\[
		V_n \longrightarrow V_{n-1} \longrightarrow \cdots \longrightarrow V_1 \longrightarrow V_0
	\]
	and the corresponding sequence of groups of divisibility corresponds to the successive quotients of the value groups. We obtain the following corresponding sequence of totally ordered groups:
	\[
		\mathbb{Z}^n \longrightarrow \mathbb{Z}^{n-1}\longrightarrow \cdots \longrightarrow\mathbb{Z} \longrightarrow 0.
	\]
\end{example}

We make the following definitions to catalog some of the possibilities for these sequences of homomorphisms.

\begin{definition}\label{definition-atomic} With $G_n$ defined as above, we define the following.
	\begin{enumerate}
		\item If there exists $n \in \bbn_0$ such that $G_n = 0$, then we say that the group $G_0$ is $n$-\textit{atomic}.
		\smallskip
		
		\item If $G_n$ is not $m$-atomic for any $m \in \bbn_0$ and there exists $n \in \mathbb{N}_0$ such that the projection $\pi_n: G_n \to G_{n+1}$ is the identity, then we say that the group $G_0$ is $n$-\textit{antimatter}.
		\smallskip
		
		\item If $G_0$ is not $m$-atomic for any $m \in \bbn_0$ and not $n$-antimatter for any $n \in \bbn_0$, then we say that the group $G_0$ is \textit{mixed}.
	\end{enumerate}
\end{definition}

Consider the sequence $G_0 \overset{\pi_0}\longrightarrow G_1 \overset{\pi_1}\longrightarrow G_2 \overset{\pi_2}\longrightarrow \cdots$ from above. We now construct a more detailed picture by considering the atomic and quasi-atomic subgroups, and their pullbacks. Recall that $A_n$ is the group generated by (almost) irreducibles and $Q_n$ the group generated by the quasi-irreducibles. Since each successive homomorphism annihilates the irreducibles, it is natural to consider the preimages of the groups generated by (almost) irreducibles and quasi-irreducibles, respectively. For the rest of this section, we set
\[
	\widehat{A_{n}} := \pi_{n}^{-1}(A_{n+1}) \quad \text{and} \quad \widehat{Q_{n}} = \pi_{n}^{-1}(Q_{n+1}).
\]
Then we obtain the commutative diagram of Figure~\ref{detailedCommutativeDiagram}, where upward arrows denote inclusion and rightward arrows denote epimorphisms $\pi_n$ with appropriate (co)domain restrictions.  Commutativity of this diagram is a standard diagram chase we omit. We note $Q_n = \ker \pi_n$ and the $n^{\text{th}}$ rightward arrow is induced by $\pi_n$, explaining why the bottom rightward arrows all go to the identity.
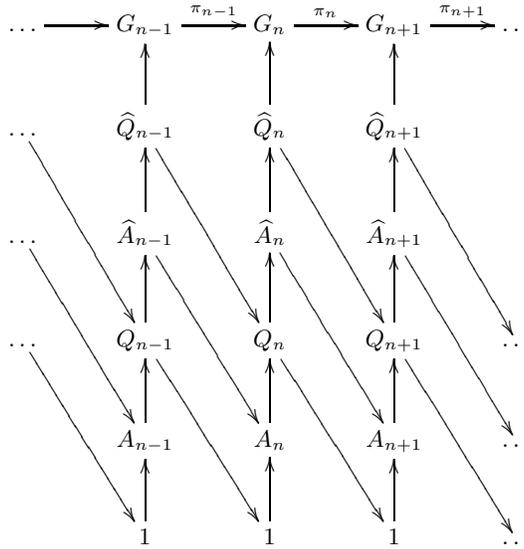
\begin{figure}[h]
	\begin{equation*}
		\xymatrix{
			\dots \ar[r] & G_{n-1} \ar[r]^-{\pi_{n-1}} & G_n \ar[r]^-{\pi_n} & G_{n+1} \ar[r]^-{\pi_{n+1}}  & \dots \\
			\dots \ar[ddr] & \widehat{Q}_{n-1}  \ar[u] \ar[ddr] & \widehat{Q}_{n} \ar[u] \ar[ddr] & \widehat{Q}_{n+1} \ar[ddr] \ar[u] & \\
			\dots \ar[ddr] & \widehat{A}_{n-1} \ar[u] \ar[ddr] & \widehat{A}_{n} \ar[u] \ar[ddr] & \widehat{A}_{n+1} \ar[u] \ar[ddr] \\
			\dots \ar[ddr] & Q_{n-1} \ar[u] \ar[ddr] & Q_{n} \ar[u] \ar[ddr] & Q_{n+1} \ar[u]  \ar[ddr] & \dots  \\
			& A_{n-1} \ar[u] &A_{n} \ar[u] & A_{n+1} \ar[u] & \dots\\
			& 1 \ar[u]& 1 \ar[u]& 1 \ar[u] & \dots
		}
	\end{equation*}
	\caption{A commutative diagram with the subgroups and morphisms used for constructing cochain complexes.}
	\label{detailedCommutativeDiagram}
\end{figure}

Due to the containment $A_{n+1} \subseteq Q_{n+1}$, we obtain that
\[
	A_{n} \subseteq Q_{n} \subseteq \widehat{A}_{n} \subseteq \widehat{Q}_{n}.
\]
As the following lemma indicates, each of the subgroups $A_n$, $Q_n$, $\widehat{A}_n$, and $\widehat{Q}_n$ yields a cochain complex with maps induced by $\{\pi_{n} : n\geq 0\}$ restricted appropriately. We only provide here a sketch of proof as most of the work needed in a full proof consists of routine diagram chase.

\begin{lemma}
	The following form cochain complexes:
	\begin{align}
		A_{\bullet}&= \dots \longrightarrow 0  \longrightarrow A_{0} \longrightarrow  A_{1} \longrightarrow  A_{2} \longrightarrow  \dots \\
		Q_{\bullet}&= \dots \longrightarrow 0  \longrightarrow Q_{0} \longrightarrow Q_{1} \longrightarrow  Q_{2} \longrightarrow \dots \\
		\widehat{A}_{\bullet}&= \dots  \longrightarrow 0 \longrightarrow \widehat{A}_{0} \longrightarrow \widehat{A}_{1} \longrightarrow \widehat{A}_{2} \longrightarrow \dots \\
		\widehat{Q}_{\bullet}&= \dots \longrightarrow  0 \longrightarrow \widehat{Q}_{0} \longrightarrow \widehat{Q}_{1} \longrightarrow \widehat{Q}_{2} \longrightarrow \dots \\
		\widehat{A}_{\bullet}/Q_{\bullet}&= \dots \longrightarrow  0 \longrightarrow \widehat{A}_{0}/Q_{0} \longrightarrow \widehat{A}_{1}/Q_{1} \longrightarrow \widehat{A}_{2}/Q_{2} \longrightarrow \dots \\
		\widehat{Q}_{\bullet}/Q_{\bullet}&= \dots \longrightarrow  0 \longrightarrow \widehat{Q}_{0}/Q_{0} \longrightarrow \widehat{Q}_{1}/Q_{1} \longrightarrow \widehat{Q}_{2}/Q_{2} \longrightarrow \dots
	\end{align}
	where the maps are naturally induced by the epimorphisms $\pi_n$.  Furthermore, the following statements hold.
	\begin{enumerate}
		\item The sequences $A_{\bullet}$, $Q_{\bullet}$, $\widehat{A}_{\bullet}/Q_{\bullet}$, $\widehat{Q}_{\bullet}/Q_{\bullet}$ are each trivial in the sense that each differential is the trivial homomorphism.
		\smallskip
		
		\item The sequence $\widehat{Q}_{\bullet}$ is exact.
	\end{enumerate}
	\label{lemma-cochain-complexes}
\end{lemma}

\begin{proof}[Sketch]
	We only show that $\widehat{A}_\bullet$ is a cochain complex, for all the others are proven similarly. All we must show is that $\pi_{n+1} \circ \pi_n = 0$. Consider the map $\pi_n: \widehat{A}_n \to \widehat{A}_{n+1}$, and recall that the equality $\widehat{A}_n = \pi_n^{-1}(A_{n+1})$ holds. Certainly $\pi_n \circ \pi_n^{-1}(A_{n+1}) = A_{n+1}$. Of course, $\pi_{n+1}$ is the natural map $G_{n+1} \onto \frac{G_{n+1}}{Q_{n+1}}$, and $A_{n+1} \subseteq Q_{n+1} = \ker \pi_{n+1}$. In particular, the image of $\pi_n$ is a subset of the kernel of $\pi_{n+1}$.
\end{proof}

We now form short exact sequences of cochain complexes of po-group homomorphisms. One can use some chase diagram arguments to argue that both
\begin{equation} \label{eq:two SES}
	0 \to Q_{\bullet} \to \widehat{A}_{\bullet} \to \frac{\widehat{A}_{\bullet}}{Q_{\bullet}} \to 0 \quad \text{and} \quad 0 \to Q_{\bullet} \to  \widehat{Q}_{\bullet} \to \frac{\widehat{Q}_{\bullet}}{Q_{\bullet}} \to 0
\end{equation}
are short exact sequences where the cochain maps between the complexes are induced by inclusion or projection where appropriate. Let $X_{\bullet}$ denote any of the complexes from Lemma \ref{lemma-cochain-complexes}, and let $\delta_i \colon X_i \to X_{i+1}$ denote the $i^{\text{th}}$ differential map. Now, for $n \in \bbz$, define the $n^{\text{th}}$ cohomology group of $X_{\bullet}$ as follows: $H^n(G,X_{\bullet})=\ker(\delta_{n})/\text{Im}(\delta_{n-1})$. Then
\begin{align*}
	H^0(G,A_\bullet) =& A_0, & H^n(G,A_\bullet) =& A_n, \\
	H^0(G,Q_\bullet) =& Q_0, & H^n(G,Q_\bullet) =& Q_n, \\
	H^0(G,\widehat{A}_\bullet) =& Q_0, & H^n(G,\widehat{A}_\bullet) =& \frac{Q_n}{A_n},\\
	H^0(G,\widehat{Q}_\bullet) =& Q_0, & H^n(G,\widehat{Q}_\bullet) =& 0, \\
	H^0\Big(G,\frac{\widehat{A}_\bullet}{Q_\bullet}\Big) =& \frac{\widehat{A}_0}{Q_0} = A_1, & H^n\Big(G,\frac{\widehat{A}_\bullet}{Q_\bullet}\Big) =& \frac{\widehat{A}_n}{Q_n} = A_{n+1}, \\
	H^0\Big(G,\frac{\widehat{Q}_\bullet}{Q_\bullet}\Big) =& \frac{\widehat{Q}_0}{Q_0} = Q_{1}, & H^n\Big(G,\frac{\widehat{Q}_\bullet}{Q_\bullet}\Big) =& \frac{\widehat{Q}_n}{Q_n} = Q_{n+1}.
\end{align*}
These cohomology groups confirm that $\widehat{A}_\bullet$ is the sequence with some factorization content: this is because $H^n(G,\widehat{A}_\bullet) = Q_n/A_n$ contains data sensitive to the gap between almost atomicity and quasi-atomicity in the $n^{\text{th}}$ stage of localization.

\begin{theorem}\label{corollary-hom}
	The long exact sequences in cohomology induced by the exact sequences of complexes in~\eqref{eq:two SES} are precisely the sequences given below.
	
	\begin{align*}
		& 0 \to  Q_0 \to Q_0 \to A_1 \to  Q_1 \to \frac{Q_1}{A_1} \to A_2 \to  Q_2 \to \frac{Q_2}{A_2} \to A_3 \to \cdots \\
		& 0 \to 	 Q_0 \to Q_0 \to Q_1 \to Q_1 \to 0 \to Q_2 \to Q_2 \to 0 \to \cdots \label{eq-les3}
	\end{align*} 
\end{theorem}

\begin{proof}[Sketch]
	We proceed to sketch the proof for the short exact sequence of cochain complexes $0 \to Q_\bullet \to \widehat{A}_\bullet \to \widehat{A}_\bullet/Q_\bullet \to 0$ (the other sequence comes from $0 \to Q_\bullet \to \widehat{Q}_\bullet \to \widehat{Q}_\bullet/Q_\bullet \to 0$ and is shown similarly).  Indeed, the long exact sequence is as follows and has maps induced naturally as described above.
	\begin{align*}
		H^0(G, Q_\bullet) \longrightarrow H^0(G, \widehat{A}_\bullet) \longrightarrow H^0(G, \widehat{A}_\bullet/Q_\bullet) \longrightarrow H^1(G, Q_\bullet) \longrightarrow \cdots
	\end{align*}
	We simply substitute our cohomology groups here.
	\[Q_0 \overset{=}{\to} Q_0 \overset{0}{\to} A_1 \overset{\epsilon}{\to} Q_1 \overset{\pi}{\to} Q_1/A_1 \overset{0}{\to} A_2 \overset{\epsilon}{\to} Q_2 \overset{\pi}{\to} Q_2/A_2 \overset{0}{\to} A_3 \to \cdots\] Each $\epsilon$ denotes the natural inclusion of groups, each $\pi$ denotes the natural epimorphism of groups, and each $0$ denotes mapping all elements to the group identity. This breaks up into $0 \to Q_0 \overset{=}{\to} Q_0 \to 0$ and $0 \to A_n \to Q_n \to Q_n/A_n \to 0$, which are exact.
\end{proof}

These cohomology groups arise naturally, and we notice that we recover the factorization information describing the gap between the atomic subgroup of divisibility and the quasi-atomic subgroup of divisibility at the $n^{\text{th}}$ stage of localization, just as expected. We conclude presenting a concrete example of an integral domain with a $4$-atomic group of divisibility, and we interpret this through the lens of our long exact sequences.

\begin{example}\label{ex-homo}
	Let $K$ be a field, and consider the 4-dimensional valuation domain
	\[
		R = K\bigg[x_1, x_2, x_3, x_4, \frac{x_2}{x_1^j}, \frac{x_3}{x_2^j}, \frac{x_4}{x_3^j} : j \in \nn \bigg]_{\m}
	\]
	where $\m$ is the maximal ideal generated by all indeterminate elements over $K$. The integral domain~$R$ has group of divisibility order-isomorphic to $\bbz^4$ ordered lexicographically. In particular, $R$ has a single irreducible, namely, $x_1$. After the first localization, we obtain only one irreducible, namely, $\frac{x_2}{1}$. Localizing twice yields a single irreducible $\frac{x_3}{1}$. Localizing a third time yields a single irreducible, namely, $\frac{x_4}{1}$. Localizing a fourth time yields the quotient field. Observe, however, that our quasi-atomic subgroup coincides with our atomic subgroup, which is isomorphic to $\bbz$, so our sequence of groups becomes
	\[
		\bbz^4 \to \bbz^3 \to \bbz^2 \to \bbz \to 0
	\] 
	(with each product ordered lexicographically). Therefore $R$ is $4$-atomic. We resolve this sequence in detail in Figure \ref{fig:detailedCommutativeDiagram}. 
	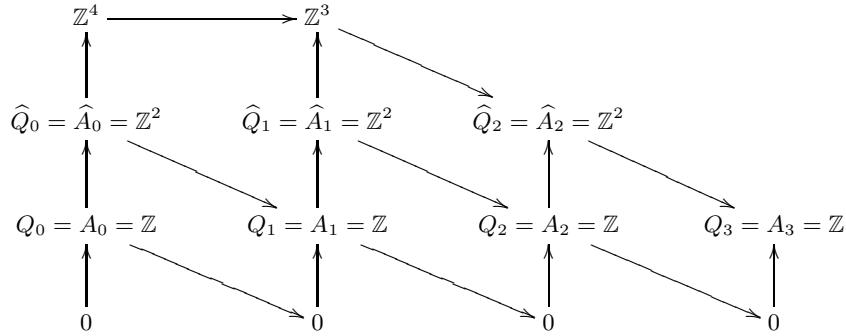
\begin{figure}[h]
		\xymatrix{
			\bbz^4 \ar[r] & \bbz^3 \ar[dr] & & \\
			\widehat{Q}_0 = \widehat{A}_{0} = \bbz^2 \ar[u] \ar[dr] & \widehat{Q}_1 =\widehat{A}_{1} = \bbz^2 \ar[u] \ar[dr] & \widehat{Q}_2 = \widehat{A}_{2} = \bbz^2 \ar[dr] & \\
			Q_0 = A_0 = \bbz \ar[u] \ar[dr] & Q_1 = A_1 = \bbz \ar[u] \ar[dr] & Q_2 = A_2 = \bbz \ar[u] \ar[dr] & Q_3 = A_3 = \bbz \\
			0 \ar[u] & 0 \ar[u] & 0 \ar[u] &0 \ar[u]
		}
		\caption{A detailed expansion of the resolution of the sequence presented in Example \ref{ex-homo}.} \label{fig:detailedCommutativeDiagram}.
	\end{figure}
	As a consequence, Theorem \ref{corollary-hom} only yields one distinct long exact sequence, which is illustrated in the following diagram:
	\[\xymatrix{
		0 \to  Q_0=\bbz \ar[r] & Q_0=\bbz \ar[r] &  A_1 = \frac{\bbz^2}{\bbz} \ar[r] &  Q_1 = \bbz \ar[r] &  \frac{Q_1}{A_1} = 0 \ar[dllll] \\
		A_2 = \frac{\bbz^2}{\bbz} \ar[r] &  Q_2=\bbz \ar[r] &  \frac{Q_2}{A_2}=0 \ar[r] & A_3 = \frac{\bbz^2}{\bbz} \ar[r] & Q_3 = \bbz \ar[r] & 0.
	}\]
	This long exact sequence breaks into four non-trivial isomorphisms. The first of these is $0 \to \bbz \overset{=}{\to} \bbz \to 0$ and the rest are all $0 \to \frac{\bbz^2}{\bbz} \overset{\simeq}{\longrightarrow} \bbz \to 0$.  
\end{example}

\bigskip
\section*{Acknowledgments}

The authors thank the anonymous referees not only for carefully reading an earlier version of this paper, but also for making many comments and suggestions that helped improve the quality of the same. During the preparation of this paper, the second author was supported by the NSF award DMS-2213323.

\bigskip

\end{document}